\documentclass{amsart}

\usepackage{amssymb}
\usepackage{url}
\usepackage{amscd}
\usepackage{texdraw}

\theoremstyle{plain}

\newtheorem*{df}{Definition of $f_{t}$}
\newtheorem*{dff}{Definition of $\widetilde{f}_{2\pi - t}$}

\newtheorem*{ex1}{Example 1}
\newtheorem*{ex2}{Example 2}
\newtheorem*{A}{Theorem A}
\newtheorem*{B}{Theorem B}
\newtheorem*{C}{Theorem C}
\newtheorem*{L'}{Lemma 3.11'}
\newtheorem*{sub1}{Sublemma 1}
\newtheorem*{sub2}{Sublemma 2}

\newtheorem*{T}{Thurston's Characterization Theorem}
\newtheorem*{A.1}{Theorem A.1}
\newtheorem*{A.2}{Theorem A.2}
\newtheorem*{A.3}{Theorem A.3}
\newtheorem*{A.4}{Theorem A.4}

\newtheorem{corollary}{Corollary}[section]
\newtheorem{lemma}{Lemma}[section]
\newtheorem{remark}{Remark}[section]

\newtheorem{definition}{Definition}[section]
\newtheorem{proposition}{Proposition}[section]
\input txdtools

\begin{document}
\title[Siegel Rational Maps with Prescribed Combinatorics]
{Dynamics of Siegel Rational Maps with Prescribed Combinatorics}
\author{Gaofei Zhang}

\address{Department of  Mathematics \\ Nanjing University, Nanjing, 210093, P. R. China  }
\email{zhanggf@hotmail.com}


\maketitle

\begin{abstract}
We extend Thurston's combinatorial criterion for postcritically
finite rational maps to a class of rational maps with bounded type
Siegel disks. The combinatorial characterization of this class of
Siegel rational maps plays a special role in the study of general
Siegel rational maps. As one of the  applications, we prove that for
any quadratic rational map with a bounded type Siegel disk, the
boundary of the Siegel disk is a quasi-circle which passes through
one or both of the critical points.

\end{abstract}

\section{Introduction}

Let $f: S^{2} \to S^{2}$ be an orientation-preserving branched
covering map. We call
$$\Omega_{f} = \{x \big{|}\: \deg_{x}f > 1\}$$
the $\emph{critical set}$ of $f$, and
$$P_{f} = \overline{\bigcup_{1 \le k < \infty}f^{k}(\Omega_{f})}$$
the $\emph{postcritical set}$. A branched covering map of the
topological two sphere is called $\emph{postcritically finite}$ if
its $\emph{postcritical set}$ is a finite set. Let $f, g:S^{2} \to
S^{2}$ be two orientation-preserving branched covering maps. We say
$f$ and $g$ are $\emph{combinatorially equivalent}$ if there exist
two homeomorphisms $\phi, \phi': (S^{2}, P_{f}) \to (S^{2}, P_{g})$,
such that the diagram

$$
     \begin{CD}
           (S^{2},P_{f})           @  >\phi'   >  >(S^{2}, P_{g})         \\
           @V f VV                         @VV g V\\
           (S^{2},P_{f})           @  >\phi   >  >      (S^{2}, P_{g})
     \end{CD}
$$
commutes, and $\phi$ is isotopic to $\phi'$ rel $P_{f}$. Thurston
proved that an orientation-preserving and $\emph{postcritically
finite}$ branched covering map with hyperbolic orbifold is
$\emph{combinatorially}$ $\emph{equivalent}$ to a rational map if
and only if it has no Thurston obstructions \cite{Th}. A detailed
proof of this theorem was presented in Douady and Hubbard's paper
\cite{DH}. Since then, it has been a tantalizing problem to see to
what extent such a combinatorial characterization is possible beyond
the $\emph{postcritically finite}$ setting.  Some progress has been
made towards this direction. For instance, McMullen proved that for
any rational map, there exist no Thurston obstructions outside all
the possible rotation domains(Siegel disks or Herman
rings)\cite{McM1}. On the other hand, it was illustrated by Cui,
Jiang, and Sullivan that there are $\emph{geometrically finite}$
branched covering maps which have no Thurston obstructions, but are
not $\emph{combinatorially equivalent}$ to  rational maps
\cite{CJS}.  Here we say a branched covering map is
$\emph{geometrically finite}$ if its $\emph{postcritical set}$ is an
infinite set but has finitely many accumulation points. The example
implies  that, to make a $\emph{postcritically infinite}$ branched
covering map  $\emph{combinatorially equivalent}$ to a rational map,
besides the non-existence of  Thurston obstructions, some additional
conditions have to be imposed on its local combinatorial structure
around the accumulation points of the $\emph{postcritical set}$. For
a $\emph{geometrically finite}$ branched covering map, such a local
condition was found by Cui, Jiang, and Sullivan, which they called
$\emph{locally linearizable}$.  According to \cite{CJS},  a
$\emph{geometrically finite}$ branched covering map is called
$\emph{locally linearizable}$ if it can  be $\emph{combinatorially
equivalent}$ to some "normalized" one such that the later map is
either holomorphically attracting or super-attracting in a
neighborhood of each accumulation point of the $\emph{postcritical
set}$. They proved that a $\emph{geometrically finite}$ branched
covering map is $\emph{combinatorially equivalent}$ to a
sub-hyperbolic rational map if and only if it is $\emph{locally
linearizable}$ and has no Thurston obstructions. Cui also studied
under what condition, a $\emph{geometrically finite}$ branched
covering map is $\emph{combinatorially equivalent}$ to a rational
map with parabolic cycles. The situation in this case becomes more
subtle where a new type of obstructions, called $\emph{invariant
connecting arcs}$, have to be considered as well as Thurston
obstructions.  We refer the reader to \cite{C1} for the details.

The main purpose of this work is to extend Thurston's combinatorial
criterion for $\emph{postcritically finite}$ rational maps to  a
class of Siegel rational maps, and then applies it to quadratic
rational maps with $\emph{bounded type}$ Siegel disks. Here we call
an irrational number $0< \theta < 1$ of $\emph{bounded type}$ if
$\sup\{a_{i}\} < \infty$ where $ [a_{1}, \cdots,a_{n},\cdots]$ is
its continued fraction. We shall assume throughout this paper that
$0< \theta< 1$ is an irrational number of $\emph{bounded type}$.

\begin{definition}\label{geom}
We use
$R_{\theta}^{geom}$
to denote the class of all the rational maps $g$ such that

\begin{itemize}
\item[1.]$g$ has a Siegel disk $D_{g}$ with  rotation number $\theta$,  and
\item[2.]$\partial D_{g}$ is a quasi-circle, and
\item[3.]$P_{g}-\overline{D_{g}}$ is a finite set.
\end{itemize}
\end{definition}

\begin{remark}\label{GS}
Assume  that $f$ has a bounded type Siegel disk $D$ such that
$\overline{D} \subset U$ where $U$ is a domain on which $f$ is holomorphic.
Then $\partial D$ must contain at least one critical point of $f$ \cite{GS}.
It follows that for any $g \in R_{\theta}^{geom }$,
$\partial D_{g} \cap \Omega_{g} \ne \emptyset$.
\end{remark}

\begin{definition}

We use $R_{\theta}^{top}$
to denote the class of all the orientation-preserving  branched covering maps
$f: S^{2} \to S^{2}$
such that
\begin{itemize}
\item[1.] $f|\Delta: z \to   e^{2 \pi  i\theta}z$ is a rigid rotation where
$\Delta = \{z \big{|}\:|z| < 1\}$ is the unit disk , and
\item[2.]$\partial \Delta \cap \Omega_{f} \ne \emptyset$ , and
\item[3.]$P_{f} - \overline{\Delta}$ is a finite set.
\end{itemize}
We call the unit disk $\Delta$ the $\emph{rotation disk}$ of $f$.
\end{definition}

For a branched covering map $f \in R_{\theta}^{top}$, we say $f$ is
$\emph{realized}$ by a Siegel
rational map $g \in R_{\theta}^{geom}$, if
$f$ and $g$ are $\emph{combinatorially equivalent}$ to each other,
and furthermore,
when restricted to the Siegel disk,
the combinatorial equivalence is a holomorphic conjugation. More precisely,

\begin{definition}\label{realize}
Let $f \in R_{\theta}^{top}$ and $g \in R_{\theta}^{geom}$.
Let $\Delta$ be the unit disk, and
$D_{g}$ be
the Siegel disk of $g$.
We say
$f$ is realized by $g$ if
\begin{itemize}
\item[1.] $f = \phi_{1} ^{-1} \circ g \circ \phi_{2}$, and
\item[2.] $\phi_{1}$ is isotopic to $\phi_{2}$ relative to $P_{f}$,
and
\item[3.] $\phi_{1}|_{\Delta}=\phi_{2}|_{\Delta}: \Delta \to D_{g}$ is holomorphic.
\end{itemize}
\end{definition}

We now present a quick summary of our results. The first theorem
extends Thurston's combinatorial
criterion for $\emph{postcritically finite}$ rational maps to the class $R_{\theta}^{geom}$.
The proof is given in Section 2.
\begin{A}
Let $0< \theta< 1$ be an irrational number of bounded type.
Then a branched covering map $f \in R_{\theta}^{top}$   can be realized by a Siegel rational map
$g \in R_{\theta}^{geom}$
if and only if $f$
has no Thurston obstructions on the outside of the rotation disk.
\end{A}

The necessary part is a direct consequence of a theorem of McMullen.
For a proof, see Appendix B of \cite{McM1}. We need only to prove
the sufficient part. The idea of the proof is as follows.  First we
construct a symmetric branched covering map $F$ such that when
restricted on the outside of the unit disk,  $F$ has the same
combinatorial structure as that of $f$. Based on the branched
covering map $F$, we construct a sequence of symmetric and
$\emph{postcritically finite}$ branched covering maps $\{F_{n}\}$
such that $F_{n} \to F$ uniformly, and $|P_{F_{n}} - \partial
\Delta| = |P_{F} - \partial \Delta|$(Proposition~\ref{basic
construction}). Then we show that for $n$ large enough, $F_{n}$ has
no Thurston obstructions, and hence by Thurston's theorem, it is
$\emph{combinatorially equivalent}$ to some rational map
$G_{n}$(Lemma~\ref{no Thurston obstructions}).  Since $F_{n}$ is
symmetric about the unit circle, it follows that $G_{n}$ is a
Blaschke product. We then prove that the sequence $\{G_{n}\}$ is
contained in some compact set of $\mathcal{R}$$_{2d-1}$, the space
of all the rational maps of degree $2d-1$(Lemma~\ref{com-F}). By
passing to a convergent subsequence, we may assume that $G_{n} \to
G$ where $G$ is a Blaschke product of degree $2d-1$. Then we show
that $F$ and $G$ are $\emph{combinatorially equivalent}$ to each
other($\S2.4$). The proof of Theorem A is then completed by a
standard quasiconformal surgery on $G$($\S2.5$).

The second theorem shows that the Julia set of any $f \in
R_{\theta}^{geom}$ has zero Lebesgue measure. In particular, it
implies the combinatorial rigidity of the maps in
$R_{\theta}^{geom}$. The proof is given in Section 3.
\begin{B}
Let $f \in R_{\theta}^{geom}$. Then the Julia set of $f$ has zero Lebesgue measure.
In particular, if $f \in R_{\theta}^{top}$ has no Thurston obstructions outside the
rotation disk $\Delta$, then up to a M\"{o}bius conjugation,
there is a unique Siegel rational map $g \in R_{\theta}^{geom}$ to realize $f$.
\end{B}

The main part of the proof is to show that the Julia set of any
Siegel rational map in $R_{\theta}^{geom}$ has zero Lebesgue
measure. The assertion of the rigidity then follows easily. For a
quadratic polynomial with a bounded type Siegel disk, the zero
measure statement was already proved by Petersen \cite{Pe1}.
Petersen's proof is based on a delicate geometric object, the so
called Petersen's puzzle. Since for a map in $R_{\theta}^{geom}$,
the boundary of the Siegel disk may contain several critical points,
each of which may have a different degree, there seems no easy way
to construct the puzzles which is suitable for all the cases.    To
avoid this difficulty, we will introduce a new method, the
$\emph{minimal neighborhood method}$, which allows us to treat all
these cases in a uniform way. One advantage of this method is that
it may also be applied in the study of the Julia sets of entire
functions with bounded type Siegel disks where Petersen puzzles are
not available\cite{Zh1}.

 Let us briefly sketch the proof of
the zero measure statement of Theorem B. Let $g \in
R_{\theta}^{geom}$. We first show that there is a Blaschke product
$G$ which models $g$. That is to say, the dynamics of $g$ on the
outside of the Siegel disk is quasiconformally conjugate to the
dynamics of $G$ on the outside of the unit disk. Therefore it
suffices to show that the set
$$
J_{\widehat{G}} = J_{G} - \bigcup_{k=0}^{\infty}G^{-k}(\Delta)
$$
has zero Lebesgue measure. Assume that it is not true. It follows
that there is a Lebesgue point of $J_{\widehat{G}}-
\bigcup_{k=0}^{\infty}G^{-k}(\partial \Delta)$, say $z_{0}$, such
that $G^{k}(z_{0}) \to \partial \Delta$ as $k \to
\infty$(Lemma~\ref{Lyu}). Now we define a sequence $\{m(k)\}_{k=1}$
such that for each $m(k)$, the point $z_{m(k)}$ is the
$\emph{nearest}$ one to $\partial \Delta$ among all the points
$z_{0}, z_{1}, \cdots, z_{m(k)}$. Here by $\emph{nearest}$ we mean
that $z_{m(k)}$ is contained in some $\emph{minimal neighborhood}$
which is attached to the unit circle(see Definition~\ref{key-s}).
The importance of the sequence $\{m(n)\}$ is that for each $m(n)$,
there is a number $\tau(n) < m(n)$, such that the inverse branch of
$G$ which maps $z_{\tau(n) + 1}$ to $z_{\tau(n)}$ strictly contracts
the hyperbolic metric in some hyperbolic Riemann surface, and
moreover, $\tau(n) \to \infty$ as $n \to
\infty$(Lemma~\ref{fin-lem}). This allows us to construct a sequence
of nested neighborhoods of $z_{0}$ such that the pre-images of
$\Delta$ count a definite part in each of these
neighborhoods($\S3.5$). It follows  that $z_{0}$ is not a Lebesgue
point of $J_{\widehat{G}}- \bigcup_{k=0}^{\infty}G^{-k}(\partial
\Delta)$. But this is a contradiction with our assumption.

As an application of Theorem A and Theorem B, in $\S4$, we prove
\begin{C}
For any bounded type irrational number $\theta$, there is a constant $1 < K < \infty$ dependent
only on $\theta$, such that for any quadratic rational map with a Siegel disk of rotation number
$\theta$, the boundary of the Siegel disk is a $K-$quasi-circle which passes through one or both
of the critical points of $f$.
\end{C}

It was conjectured by Douady and Sullivan that the boundary of a
Siegel disk for a rational map is a Jordan curve. The conjecture is
still open and is far from being solved. For Siegel disks of
polynomial maps, however, there has been some progress towards this
conjecture \cite{D}, \cite{PZ} and \cite{Z1}. We especially refer
the readers to \cite{Z2} for a survey of all the relative results in
this aspect.

Let us sketch the proof of Theorem C as follows. In $\S4.1$, we consider a quadratic rational
map $g$ with a Siegel disk of rotation number $\theta$.  By a M\"{o}bius conjugation, we may
normalize $g$ such that the Siegel disk is centered at the origin, and $g'(1) = 0$, $g(\infty) = \infty$.
Let $\Sigma$ denote the space of all such maps. Each map in $\Sigma$ has exactly two critical points $1$
and some $c \ne 1$. We denote such map by $g_{c}$.   It follows that the maps in  $\Sigma$ are parameterized
by their critical points  which are distinct from $1$. Under this parameterization, the space $\Sigma$ is
homeomorphic to $ \widehat{\Bbb C}-\{0, 1, -1\}$.

In $\S4.2$, we consider  a family of degree-2 topological branched
covering maps $f_{t} \in R_{\theta}^{top}, 0< t < 2 \pi$ such that
both the critical points of $f_{t}$ are on the unit circle and span
an angle $t$ (see Figure 15). Clearly such $f_{t}$ has no Thurston
obstructions outside the $\emph{rotation disk}$. It follows that for
each $0< t< 2\pi$, there is a unique $c(t) \in {\Bbb C}-\{0, 1,
-1\}$, such that $g_{c(t)}$ $\emph{realizes}$ $f_{t}$(in the sense
of Lemma~\ref{rel-t}).  Similarly, we consider the family of
topological branched covering maps $\tilde{f}_{t} \in
R_{\theta}^{top}$(see Figure 16), and for each $0< t< 2\pi$, we get
a unique $\tilde{c}(t) \in {\Bbb C}-\{0, 1, -1\}$ such that
$g_{\tilde{c}(t)}$ $\emph{realizes}$ $\tilde{f}_{t}$(in the sense of
Lemma~\ref{rel-t}).

In $\S4.3$,  we prove that there is a uniform $1< K < \infty$, which
is independent of $t$, such that the boundary of the Siegel disk for
any map $g_{c(t)}$ is a $K-$quasi-circle(Lemma~\ref{distortion}).

In $\S4.4$, we prove that  $\gamma = \{c(t)\:\big{|}\:0< t < 2\pi\}$
is a continuous curve segment which connects $1$ and $-1$.  By the
same way, we get that $\tilde{\gamma} = \{\tilde{c}(t)\:\big{|}\:0<
t < 2\pi\}$ is also a continuous segment which connects $1$ and
$-1$.   We then show that $\xi = \gamma \cup \tilde{\gamma} \cup
\{1, -1\}$ is a simple closed curve, which separates $0$ and the
infinity, and moreover, $\xi$ is invariant under $c \to
1/c$(Lemma~\ref{c-curve}).

Let $\Omega_{\infty}$ be the unbounded component of $\widehat{\Bbb
C}-\xi$. In  $\S4.5$, $\S4.6$, and $\S4.7$, we show that  for any
four distinct integers $0 \le k < l< m< n$, the cross-ratios of
$g_{c}^{k}(1), g_{c}^{l}(1), g_{c}^{m}(1)$, and $g_{c}^{n}(1)$ are
holomorphic functions on $\Omega_{\infty}$ and have no zeros.
Moreover, each cross-ratio function can be continuously extended to
$\partial \Omega_{\infty} = \xi$. This implies that the modulus of
each cross-ratio function obtains its maximum and minimum on the
boundary $\xi$(Proposition~\ref{nov}). This is the key idea of the
proof.

In $\S4.8$, for each $c \in \Omega_{\infty}$, we define a map
$T_{c}: \{e^{2 k \pi i\theta}\:\big{|} \: k \ge 0\} \to
\widehat{\Bbb C}$ by $T(e^{2 k \pi i \theta} ) = g_{c}^{k}(1)$.  We
show that  $T_{c}$ can be continuously extended to a homeomorphism
$T_{c}: \partial \Delta \to \overline{\{g_{c}^{k}\}_{k \ge 0}(1)}$.
It follows that $\gamma_{c} = \overline{\{g_{c}^{k}\}_{k \ge 0}(1)}$
is a Jordan curve(Lemma~\ref{Jordan}). We then show that for every
four ordered points $z_{1}, z_{2}, z_{3}, z_{4}$ on $\gamma_{c}$,
$$\big{|}\frac{(z_{1} - z_{3})(z_{2}-z_{4})}{(z_{2} -
z_{3})(z_{1}-z_{4})}\big{|} \ge \delta$$ for some $\delta >0$. This,
together with Lemma 9.8\cite{Po}(see also Lemma\ref{q-c}), implies
that $\gamma_{c}$ is actually a quasi-circle. The same cross ratio
argument also implies that $\gamma_{c}$ moves continuously as $c$
varies on $\Omega_{\infty}$(Lemma~\ref{e-x}). It follows that
$\gamma_{c}$ is the boundary of the Siegel disk of $g_{c}$ which is
centered at the origin. This proves Theorem C.

For reader's convenience, in $\S5$, we give a brief introduction
of Thurston's characterization
theory on $\emph{postcritically finite}$ rational maps. The version
we present here is slightly
different from the one in \cite{DH}: the $\emph{postcritical set}$ $P_{f}$
is replaced by a $f-$invariant
set which contains $P_{f}$ as its subset.   We also present several results
on short simple closed geodesics
on hyperbolic Riemann surfaces, which will be used in several places in this paper.
We numbered them by Theorem A.1,
Theorem A.2, Theorem A.3, and Theorem A.4. The reader may refer to \cite{DH} for the
details of the proofs.

This work is based on my  Ph.D. thesis at CUNY \cite{Zh2}. I would
like to express my gratitude to my advisor, Prof. Yunping Jiang for
suggesting this problem, and also for his constant encouragement.
Further thanks are due to Prof. Linda Keen and Prof. Frederick
Gardiner for many useful conversations during the writing of the
paper.

\begin{center}
\section{Realize a Siegel Disk with Prescribed Combinatorics}
\end{center}

\subsection{ Constructing Symmetric Branched Covering Maps }

\subsubsection{Notations}  Let $S^{2}$ denote the topological two
sphere. Let $\Delta$ and $\Bbb T$ denote the unit disk and unit
circle, respectively.  For a set $P \subset S^{2}$, let $|P|$ denote
the cardinality of the set P. Let ${\Bbb P}^{1}$ denote the Riemann
sphere with the standard complex structure. Given a point $w \in
{\Bbb P}^{1} $, let $w^{*}$ denote the symmetric image of $w$ about
the unit circle, i.e., $w^{*} = 1 /\overline{w}$. For a set $W
\subset {\Bbb P}^{1}$, let $W^{*} = \{w^{*} \: \big{|}\:w \in W \}$.
For $x \in S^{2}$, and $\delta > 0$, let $B_{\delta}(x)$ denote the
open disk with center at $x$ and radius $\delta$ with respect to the
spherical metric. For $x, y \in S^{2}$, we use $d_{S^{2}}(x, y)$ to
denote the spherical distance between $x$ and $y$. For two maps $f,
g: S^{2} \to S^{2}$, the distance between $f$ and $g$ is defined to
be $d(f, g) = \sup_{x\in S^{2}} d_{S^{2}}(f(x), g(x))$. For two
subsets $A, B \subset S^{2}$, define $d_{S^{2}}(A, B) = \inf_{x \in
A, y\in B}d_{S^{2}}(x, y)$.

\begin{figure}
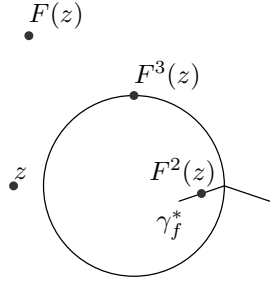

\bigskip
\begin{center}
\centertexdraw { \drawdim cm \linewd 0.02 \move(0 1)

\move(0 -2) \lcir r:1.2

\move(-1.6 -2) \fcir f:0.2 r:0.06

\move(-1.6 -1.9) \htext{$z$}

\move(-1.4 0) \fcir f:0.2 r:0.06

\move(-1.4 0.1) \htext{$F(z)$}

\move(0 -0.8) \fcir f:0.2 r:0.06

\move(0 -0.7) \htext{$F^{3}(z)$}

\move(1.2 -2) \lvec(0.6 -2.2)

\move(1.2 -2) \lvec(1.8 -2.2) \move(0.3 -2.7)
\htext{$\gamma_{f}^{*}$}

 \move(0.9 -2.1) \fcir f:0.2
r:0.06 \move(0.2 -2) \htext{$F^{2}(z)$} }
\end{center}
\vspace{0.2cm} \caption{A critical orbit of $F$ which falls into
$\Delta$}
\end{figure}

\subsubsection{The Choice of the Infinity}
Suppose $f \in R_{\theta}^{top}$ has no Thurston obstructions outside
the $\emph{rotation disk}$ $\Delta$. Let $d \ge 2$ be the degree of $f$.
By a standard topological argument, it follows that  $f$ has at least one
fixed point in the outside of the unit disk. There are two cases. In the first case,
$P_{f}$ contains a fixed point of $f$. In this case, up to a combinatorial equivalence,
we may assume that $\infty \in P_{f}$ and $f(\infty) = \infty$. In the second case, $P_{f}$
does not contain any fixed point of $f$. In this case, up to a combinatorial equivalence,
we  assume that the infinity is one of the fixed points of $f$.

\subsubsection{Construction of F}
Since $\Omega_{f} \cap \partial \Delta \ne \emptyset$,  we may also
assume that $1 \in \Omega_{f}$. It follows that there is  a curve
segment, say $\gamma_{f}$, which is attached to $1$ from the outside
of the unit disk, such that $f(\gamma_{f}) \subset \partial \Delta$.
Let
$$
X = \{z \in \Omega_{f} - \overline{\Delta}  \;\big{|}\;f^{i}(z) \in
\Delta -\{0\} \hbox{ for some } i > 0\}.
$$
For each $z \in X$, let $i_{z} > 0$ be the smallest integer such
that $f^{i_{z}}(z) \in \Delta$.  Let $\widetilde{X} =
\{f^{i_{z}}(z)\: \big{|}\: z \in X\}$ and $\sigma: S^{2} \to S^{2}$
be a homeomorphism  such that $\sigma|(S^{2} - \Delta)  = id$ and
$\sigma (\widetilde{X}) \subset \gamma_{f}^{*}$. Note that by our
notation, $\gamma_{f}^{*}$ is the symmetric image of $\gamma_{f}$
about the unit circle. Let $\widetilde{f} = \sigma \circ f$. Define
a symmetric branched covering map of the sphere by

\begin{equation}\label{branched covering map}
F(z) =
\begin{cases}
\widetilde{f}(z) & \text{ if $|z| \ge 1$}, \\
(\widetilde{f}(z^{*}))^{*}& \text{ for otherwise}.
\end{cases}
\end{equation}

From the construction of $F$, it follows that   $P_{F} - \partial
\Delta$ is a finite set, and moreover,  for $z \in X$, $F^{i_{z}}(z)
\in \gamma_{f}^{*}$, and hence $F^{i_{z}+1}(z) \in \partial \Delta$
(see Figure 1).

\subsubsection{Construction of $F_{n}$}
Let $\theta_{n} = p_{n}/q_{n}$ be a sequence of rational numbers
such that $\theta_{n} \to \theta$ as $n$ goes to $\infty$. Let
$O_{n} = \{e^{ 2\pi ik\theta_{n}}\:\big{|}\: 0 \le k < q_{n}\}$. Let
$A(a, b)$ be the annulus with outer radius $a$ and inner radius $b$.
Since $P_{F} -\partial \Delta$ is a finite set, there are $0< r< 1<
R$ such that $(A(R, r) -\partial \Delta) \cap (\Omega_{F} \cup
P_{F}) = \emptyset$. Set
$$Y = \{z \in (\Omega_{F} \cup P_{F})-\partial \Delta\:\big{|}\:F(z) \in \partial \Delta\}, $$
and
$$
Z = (\Omega_{F} \cap \partial \Delta) \cup F(Y).
$$
Clearly, $Z$ is a finite set.  It follows that for every $n$ large
enough, there is a homeomorphism $\sigma_{n}: \partial \Delta \to
\partial \Delta$ such that \begin{itemize}
\item[1.] $\sigma_{n}(1) = 1$,
\item[2.] $\sigma_{n}^{-1}(Z) \subset O_{n}$,
\item[3.] $\sigma_{n}$ preserves the orbit relations among the points in
the set $Z$ in the following sense: If there is an $m > 0$ and $x, y \in Z$
such that $F^{m}(x) = y$
then $e^{2 \pi i m\theta_{n}}\sigma_{n}^{-1}(x) = \sigma_{n}^{-1}(y)$,
\item[4.] $\sigma_{n} \to id$ uniformly as $n \to \infty$.
\end{itemize}
We then extend $\sigma_{n}$ to be a homeomorphism of the sphere to
itself, which is still denoted by $\sigma_{n}$, such that
\begin{itemize}
\item[1.] $\sigma_{n} = id$ outside $A(R, r)$,
\item[2.] $\sigma_{n}(z)^{*} = \sigma_{n}(z^{*})$,
\item[3.] as $n \to \infty$,  $\sigma_{n} \to id$ uniformly with respect to the
          spherical metric.
\end{itemize}
Now for every $n$ large enough, let us define a homeomorphism
$h_{n}: \partial \Delta \to \partial \Delta$ by $$h_{n}(z)= e^{2 \pi
i \theta_{n}} \sigma_{n}^{-1}(e^{-2 \pi i \theta}z).$$ We then
extend $h_{n}$ to be a homeomorphism of the sphere to itself, which
is still denoted by $h_{n}$,  such that
\begin{itemize}
\item[1.] $h_{n} = id$ outside  $A(R, r)$,
\item[2.] $h_{n}(z)^{*} =
h_{n}(z^{*})$,
\item[3.]  as $n \to \infty$, $h_{n}(z) \to id$ uniformly with respect to the spherical metric.
\end{itemize}
Let $\widetilde{F}_{n} = h_{n}\circ  F \circ \sigma_{n}$. It follows
that
\begin{itemize}
\item[1.] $(\widetilde{F}_{n}|\partial \Delta) (z) = e^{2 \pi i \theta_{n}} z$,
\item[2.] $P_{F}- \partial \Delta  = P_{\widetilde{F}_{n}} - \partial \Delta$,
\item[3.] $\Omega_{F} - \partial \Delta = \Omega_{\widetilde{F}_{n}} - \partial \Delta$.
\end{itemize}
For each $\xi \in Y$, take a small closed topological disk $U_{\xi}$
containing $\xi$ in its interior such that
\begin{itemize}
 \item[1.] all $U_{\xi}, \: \xi \in Y$ are disjoint with each other,
 and $U_{\xi} \cap \partial \Delta  = \emptyset$,
\item[2.] $\widetilde{F}_{n}(U_{\xi}) \subset A(R, r)$,
\item[3.] $U_{\xi}^{*} = U_{\xi^{*}}$,
\item[4.] $\widetilde{F}_{n}(U_{\xi})$ is a closed topological disk and
$\widetilde{F}_{n} (\partial U_{\xi}) = \partial  \widetilde{F}_{n}(U_{\xi})$.
\end{itemize}
For each $\xi \in Y$, let us define a homeomorphism $g_{n, \xi}:
\widetilde{F}_{n}(U_{\xi}) \to \widetilde{F}_{n}(U_{\xi})$ such that
\begin{itemize}
\item[1.]  $g_{n, \xi} = id$ on $\partial \widetilde{F}_{n}(U_{\xi})$,
\item[2.]  $g_{n,\xi}(\widetilde{F}_{n}(\xi)) = \sigma_{n}^{-1}(F(\xi))$,
\item[3.]  $g_{n, \xi}(z)^{*} = g_{n,\xi^{*}}(z^{*})$,
\item[4.]  as $n \to \infty$, $g_{n,\xi} \to id$ uniformly with respect
to the spherical metric.
\end{itemize}
Now let us define
\begin{equation}\label{covering map sequence}
F_{n}(z) =
\begin{cases}
g_{n, \xi} \circ \widetilde{F}_{n}(z) & \text{ for $z \in \bigcup_{\xi \in Y}U_{\xi}$}, \\
\widetilde{F}_{n}(z)& \text{ for otherwise}.
\end{cases}
\end{equation}

Let $Z_{n} = (\Omega_{F_{n}} \cap \partial \Delta) \cup F_{n}(Y)$.
It follows from the construction that $|Z_{n}| = |Z|$ for all $n$
large enough. Moreover, for each $x \in Z$, there is an $x_{n} \in
Z_{n}$, such that $x_{n} \to x$ as $n \to \infty$.  It follows that
for all $n$ large enough, the map $x \to x_{n}$ is a one-to-one
correspondence between $Z$ and $Z_{n}$. By the construction of
$F_{n}$, the reader shall easily supply a proof of the following
proposition:

\begin{proposition}\label{basic construction}
The sequence $\{F_{n}\}$ satisfy the following properties,
\begin{itemize}
\item[1.] $F_{n} \to F$ uniformly with respect to the spherical metric,
\item[2.] $F_{n}$ is an orientation-preserving and postcritically finite
branched covering map such that $F_{n}(z)^{*} = F_{n}(z^{*})$,
\item[3.] $|P_{F_{n}} - \overline{\Delta}| = |P_{F} - \overline{\Delta}|$
for every $n$ large enough,
\item[4.] $(F_{n} |\partial \Delta)(z) = e^{2 \pi  i \theta_{n}} z$,
\item[5.] $P_{F_{n}} \cap \partial \Delta = O_{n}$.
\item[6.] For every $n$ large enough, $F_{n}$ preserves the orbit relations
among the points in  the set $Z$ in the following sense:  if for $x, y \in Z$
and some integer $m \ge 0$, $F^{m}(x) = y$, then for the correspondent points
$x_{n}$ and $y_{n}$, $F_{n}^{m}(x_{n}) = y_{n}$.
\item[7.] For every $n$ large enough, there is a curve segment $\gamma_{n}$
attached to $1$ from the outside of the unit disk such that
$F_{n}(\gamma_{n}) \subset \partial \Delta$, and moreover, if for some
$z \in (\Omega_{F_{n}} \cup P_{F_{n}}) - \overline{\Delta}$,
$F_{n}(z) \in \Delta -\{0\}$, then $F_{n}(z) \in \gamma_{n}^{*}$.
\end{itemize}
\end{proposition}

\begin{remark}\label{additional}
Note that the combinatorial structure of $f$ in the inside of the
rotation disk is not reflected by $F$. We will use an additional
argument to take care of this  in $\S2.5$.
\end{remark}

\subsection{ No Thurston Obstructions of $F_{n}$ for Large $n$}
Let $P_{F_{n}}'$ and $P_{F}'$  denote the set $P_{F_{n}} \cup \{0,
\infty\}$ and the set $P_{F} \cup \{0, \infty\}$, respectively. For
a finite subset $P \subset S^{2}$   with $|P| \ge 4$, we say a
simple closed curve $\gamma \subset S^{2} - P$ is
$\emph{non-peripheral}$ if each component of $S^{2} - \gamma$
contains at least two points of $P$. Let $\phi: S^{2} \to {\Bbb
P}^{1}$ be a homeomorphism. For each $\emph{non-peripheral}$ curve
$\gamma \subset S^{2} - P$,  there is a unique simple closed
geodesic $\eta \subset {\Bbb P}^{1} - \phi(P)$  in the homotopy
class of $\phi(\gamma)$. We use $\|\gamma\|_{\phi,P}$ to denote the
hyperbolic length of $\eta$. We say $\gamma$ is a $(\phi,
P)-geodesic$ if $\eta = \phi(\gamma)$.

\subsubsection{Thurston's pull back}
Now let $n \ge 1$ be fixed. Let $\phi_{0} = Id$. For $m= 1,2,
\cdots$, let $\tau_{m}$ be the complex structures on $S^{2}$ which
is obtained by pulling back the standard complex structure
$\tau_{0}$ by $F_{n}^{m}$. Associated to  each $\tau_{m}$ is a
quasiconformal homeomorphism $\phi_{m}:S^{2} \to {\Bbb P}^{1}$ which
fixes $0,1$ and $\infty$.  Let $G_{m} = \phi_{m} \circ F_{n} \circ
\phi_{m+1}^{-1}$, then the following diagram
$$
     \begin{CD}
           (S^{2},P_{F_{n}}')           @  >\phi_{m+1}   >  >({\Bbb P}^{1},\phi_{m+1}(P_{F_{n}}'))         \\
           @V F_{n} VV                         @VV G_{m} V\\
           (S^{2},P_{F_{n}}')           @  >\phi_{m}   >  >         ({\Bbb
P}^{1},\phi_{m}(P_{F_{n}}'))
     \end{CD}
$$commutes and $G_{m}$ is a rational map of the Riemann sphere ${\Bbb P}^{1}$.

Since $F_{n}(z^{*}) = F_{n}(z)^{*}$, by induction we have
$\phi_{m}(z^{*}) = \phi_{m}(z)^{*}$ and hence  $G_{m}(z^{*}) =
G_{m}(z)^{*}$ for all $m=0,1,\cdots$. Therefore, $G_{m}$ is a
Blaschke product on ${\Bbb P}^{1}$. By the assumption that
$f(\infty) = \infty$, it follows that  $F(\infty) = \infty$, and
therefore, $G_{m}(\infty) = \infty$. We write
$$
G_{m}(z) = \lambda_{m} z \prod_{1 \le k \le d-1}\frac{z -p_{k,m}}{1
-\overline{p}_{k,m}z} \prod_{1 \le k \le d-1}\frac{z -q_{k,m}}{1
-\overline{q}_{k,m}z}
$$
where $d \ge 2$ is the degree of $f$, and $p_{k,m} \in {\Bbb
C}-\overline{\Delta}, q_{k,m} \in \Delta , 1 \le k \le d-1$, and
$\lambda_{m} = e^{2 \pi i\alpha_{m}} $ for some  real constant $0
\le \alpha_{m} < 1$.

\vspace{0.5cm}

\subsubsection{Analysis of short simple closed geodesics}
Let $\gamma$ be a short simple closed $(\phi_{m},
P_{F_{n}}')-geodesic$. If $\gamma$ intersects the unit circle, we
use $D(\gamma)$ to denote the component of $S^{2} - \gamma$ which
does not contain the origin. Otherwise, we use $D(\gamma)$ to denote
the component of $S^{2} - \gamma$ which does not contain the unit
circle.
\begin{lemma}\label{symmetric geodesic}
Let $\gamma$ be a simple closed $(\phi_{m}, P_{F_{n}}')-geodesic$
which intersects the unit circle such that $\|\gamma\|_{\phi_{m},
P_{F_{n}}'} < log(\sqrt 2 +1)$.  Then $\gamma$ is symmetric about
the unit circle. In particular, $\gamma \cap \partial \Delta$
contains exactly two points.
\end{lemma}

\begin{proof}
Let $\gamma^{*}$ be the symmetric image  of $\gamma$ about the unit
circle. Clearly, $\gamma^{*}$  is also a simple closed  $(\phi_{m},
P_{F_{n}}')-geodesic$ and $\|\gamma^{*}\|_{\phi_{m}, P_{F_{n}}'}  =
\|\gamma\|_{\phi_{m}, P_{F_{n}}'} < log(\sqrt 2 +1)$.  Since $\gamma
\cap \gamma^{*} \ne \emptyset$, by Theorem A.1, we get that $\gamma
= \gamma^{*}$.

\end{proof}

\begin{lemma}\label{lower bound of symmetric geodesic}
For every $n$ large enough, there is a $\delta>0$ independent of $m$
such that for  every simple closed  $(\phi_{m},
P_{F_{n}}')-geodesic$ $\gamma$ which intersects the unit circle, we
have $\|\gamma\|_{\phi_{m},P_{F_{n}}'} \ge \delta$.
\end{lemma}

The idea behind the proof is as follows.  Let $\gamma$ be a simple
closed geodesic which intersects the unit circle. If $\gamma$ is
short enough,  its images under the forward iterations of $F_{n}$
generate a set of short simple  closed geodesics which intersect
the unit circle. The number of the short simple closed geodesics in
this set can be very large if $\gamma$ is short enough. But on the
other hand, there can not be too many such short simple closed
geodesics, for otherwise, there would be two of them which intersect
with each other, and this is a contradiction with Theorem A.1.

\begin{proof}
We prove it by contradiction. We claim that for every $n$ large
enough, there exist $\delta' > 0$ and $ 1< C < \infty$ independent
of $m$,  such that if $\gamma \subset S^{2}-P_{F_{n}}'$ is a  simple
closed $(\phi_{m}, P_{F_{n}}')-geodesic$ with $\|\gamma\|_{\phi_{m},
P_{F_{n}}'} < \delta'$, there is a simple closed $(\phi_{m},
P_{F_{n}}')-geodesic$ $\xi$ which is symmetric about the unit circle
such that $\|\xi\|_{\phi_{m}, P_{F_{n}}'} < C \delta'$ and $D(\xi)
\cap \partial \Delta \cap P_{F_{n}}'$ contains at least two points.
Let us prove the claim. Suppose it is not true. Then $D(\gamma) \cap
\partial \Delta \cap P_{F_{n}}'$ contains at most one point. Now
take $\delta'>0$ small, so that the simple closed geodesics
generated in the following are all short enough.  Let $N  = |P_{F}'
- \overline{ \Delta}|$, and hence $|P_{F_{n}}' -\partial \Delta| =
N$ by (3) of Proposition~\ref{basic construction}.  For each $k =1,
2, \cdots, N+2$, Let $\eta_{k} \subset S^{2} -
F_{n}^{-k}(P_{F_{n}}')$ be the shortest simple  closed $(\phi_{m},
F_{n}^{-k}(P_{F_{n}}'))-geodesic$ which is homotopic to $\gamma$ in
$S^{2} - P_{F_{n}}'$. By Theorem A.3, we have
\begin{equation}
\|\eta_{k}\|_{\phi_{m}, F_{n}^{-k}(P_{F_{n}}')}  < C_{1}
\|\gamma\|_{\phi_{m}, P_{F_{n}}'}
\end{equation}
where $1 < C_{1}< \infty$ depends only on $k$ and $|P_{F_{n}}'|$.
From  Theorem A.2,  we conclude that $F_{n}^{k}(\eta_{k})$ covers a
simple closed $(\phi_{m-k}, P_{F_{n}}')-geodesic$ $\xi_{k}'$. Hence
\begin{equation}
\|\xi_{k}'\|_{\phi_{m-k}, P_{F_{n}}'} \le \|\eta_{k}\|_{\phi_{m},
F_{n}^{-k}(P_{F_{n}}')}
\end{equation}

Let $\xi_{k} \subset S^{2}- P_{F_{n}}'$ be the simple closed
$(\phi_{m}, P_{F_{n}}')-geodesic$ which is homotopic to $\xi_{k}'$
in $S^{2}- P_{F_{n}}'$. By Theorem A.4 and the fact Thurston's pull
back does not increase the Teichm\"{u}ller distance (see Proposition
3.3, \cite{DH}), it follows that there is a constant $1 < C_{2}<
\infty$ independent of $m$, such that
\begin{equation}
\|\xi_{k}\|_{\phi_{m}, P_{F_{n}}'} < C_{2} \|\xi_{k}'\|_{\phi_{m-k},
P_{F_{n}}'}
\end{equation}

Now by taking $\delta'$  small, we conclude that $\xi_{1}, \cdots,
\xi_{N+2}$ are all short simple closed $(\phi_{m},
P_{F_{n}}')-geodesics$  which intersect the unit circle. By
Lemma~\ref{symmetric geodesic}, they are all symmetric about the
unit circle. Now let us show that the domains  $D(\xi_{1}),\cdots,
D(\xi_{N+2})$ are disjoint with each other.  Suppose this is not
true.  Then by Theorem A.1, we have $D(\xi_{j}) \subset D(\xi_{i})$
for some $1 \le i \ne j \le N+2$. We may assume that $|D(\xi_{i})
\cap \partial \Delta \cap P_{F_{n}}'|\le 1$, for otherwise the claim
is proved. It then follows that $\xi_{i}$ intersects either exactly
one of the connected components of $\partial \Delta - P_{F_{n}}'$ or
two of them which are adjacent to each other.  Let $I$ be a
component  of $\partial \Delta - P_{F_{n}}'$ which intersects both
$\xi_{i}$ and $\xi_{j}$. Let  $l = |j - i| \le N+1$. Then $I$ is
either periodic under $F_{n}^{l}$ or is mapped by $F_{n}^{l}$ to one
of its adjacent component of $\partial \Delta - P_{F_{n}}'$.  Since
$(F_{n}|\partial \Delta)(z) = e^{2 \pi i\theta_{n}} z$ and
$\theta_{n} \to \theta$ as $n \to \infty$, both of the two cases are
impossible when $n$ is large enough.

If none of $D(\xi_{i}), 1\le i\le N+2$ contains at least two points
in $P_{F_{n}}' - \partial \Delta$, we have for every $1 \le  i\le
N+2$, $|D(\xi_{i}) \cap (P_{F_{n}}' - \partial \Delta)| \ge 2$ and
hence $|P_{F_{n}}' - \partial \Delta| \ge 2N+2$. This is a
contradiction with that $|P_{F}' - \overline{\Delta}| = N$. This
proves the claim.

Now we may assume that $D(\gamma) \cap \partial \Delta \cap
P_{F_{n}}'$ contains at least two points. There are two cases. In
the first case, $(\partial \Delta  \cap P_{F_{n}}')- D(\gamma)=
\emptyset$. It follows that $\gamma$ intersects exactly one of the
connected components of $\partial \Delta -P_{F_{n}}'$. When $n$ is
large enough, by the same argument as before, we get $N+2$ short
simple closed $(\phi_{m}, P_{F_{n}}')-geodesics$ $\xi_{1}, \cdots,
\xi_{N+2}$. It follows from (5) of Proposition~\ref{basic
construction} that every $\xi_{i}$ also intersects exactly one of
the connected components of $\partial \Delta -P_{F_{n}}'$, for $1
\le i \le N+2$. It follows that each $D(\xi_{i})$ either contains
all the points in $\partial \Delta \cap P_{F_{n}}'$, or contains
none of them. We claim that there are $\xi_{i},\xi_{j}$ such that
$D(\xi_{i}) \subset D(\xi_{j})$ for some $1 \le i \ne j \le N+2$. In
fact, if this is not true, by Theorem A.1, the domains $D(\xi_{1}),
\cdots, D(\xi_{N+2})$ are disjoint with each other. It follows that
there are at least $N+1$ domains of $D(\xi_{i}), 1 \le i\le N+2$
which contain none of the points in $\partial \Delta \cap
P_{F_{n}}'$. Therefore, each of these domains must contain at least
two points in $P_{F_{n}}' -
\partial \Delta$, and this implies that $|P_{F_{n}}' -
\partial \Delta | \ge 2(N+1)$, which is a contradiction with that
$|P_{F}' -\overline{\Delta}| = N$. The claim follows. Now assume
that $D(\xi_{i}) \subset D(\xi_{j})$ for some $1 \le i \ne j \le
N+2$.  We claim that each component of $\partial \Delta -P_{F_{n}}'$
intersects at most one of the curves in  $\xi_{1}, \cdots, \xi_
{N+2}$. In fact, if some component, say $I$, of $\partial \Delta
-P_{F_{n}}'$ interests both  $\xi_{l}$ and $\xi_{m}$ for some $1 \le
l <  m\le N+2$, then $I$ is periodic under $F_{n}^{m-l}$, which is
impossible when $n$ is large enough.    It follows that $D(\xi_{j})$
must contain all the other $D(\xi_{k}), 1 \le k \le N+2, k \ne j$,
and hence the $N+1$ domains $D(\xi_{k}), 1 \le k \le N+2, k \ne j$,
must be disjoint with each other, and moreover,  each of them
contains none of the points in $\partial \Delta \cap P_{F_{n}}'$.
By counting the number of the points in $P_{F_{n}}' -\partial
\Delta$, we get a contradiction again.

In the second case, $(\partial \Delta  \cap P_{F{n}})- D(\gamma) \ne
\emptyset$. Let $I = \partial \Delta \cap D(\gamma)$. Since $O_{n} =
P_{F_{n}}' \cap \partial \Delta$ is a periodic cycle of $F_{n}$ with
period $q_{n}$, it follows that there is an integer $0< k< q_{n}$
such that (1)
 $F_{n}^{k}(I) \cap I \cap P_{F_{n}}' \ne \emptyset$, (2)
 $(I - F_{n}^{k}(I))\cap P_{F_{n}}' \ne \emptyset$ and (3)
 $(F_{n}^{k}(I) - I)\cap P_{F_{n}}' \ne \emptyset$.
Let $ \eta_{k} \subset S^{2} - F_{n}^{-k}(P_{F_{n}}') $ be a simple
closed  $(\phi_{m}, F_{n}^{-k}(P_{F_{n}}'))-geodesic$ which is
homotopic to $\gamma$ in $S^{2} - P_{F_{n}}'$. By Theorem A.2, $
F_{n}^{k}(\eta_{k}) $ covers a short simple closed $(\phi_{m-k},
P_{F_{n}}')-geodesic$ $\xi_{k}'$. By theorem A.4, there is a short
simple closed $(\phi_{m}, P_{F{n}})-geodesic$ $\xi_{k} \subset S^{2}
- P_{F_{n}}'$ which is homotopic to $\xi_{k}'$ in $S^{2} -
P_{F_{n}}'$.  It follows that $D(\gamma) \cap D(\xi_{k}) \ne
\emptyset$ and neither of them is contained in the other one. This
implies that $\gamma \cap \xi_{k} \ne \emptyset$. This is a
contradiction with Theorem A.1.

\end{proof}

\begin{lemma}\label{invarance}
Let $\gamma$ be a simple closed $(\phi_{m}, P_{F_{n}}')-geodesic$
which is contained in the inside of the unit disk. If
$\|\gamma\|_{\phi_{m}, P_{F_{n}}'}$ is small enough, then each
non-peripheral component of $F_{n}^{-1}(\gamma)$ is totally
contained in the inside of the unit disk also.
\end{lemma}

\begin{proof}
 Suppose $\|\gamma\|_{\phi_{m}, P_{F_{n}}'}$ is small enough.
 Let $\eta$ be a non-peripheral component of $F_{n}^{-1}(\gamma)$.
 Clearly, $\eta$ is a short simple closed
 $(\phi_{m+1}, F_{n}^{-1}(P_{F_{n}}'))-geodesic$. Since
 $F_{n}: \partial\Delta \to \partial \Delta$ is a homeomorphism,
 it follows that $\eta$ does not intersect the unit circle. Suppose $\eta$
 is in the outside of the unit disk. Take a point, say $x \in D(\gamma)$.
 Let us contract $\gamma$ continuously to $x$. There are two cases.
 In the first case,  $D(\gamma)$ does not contain any critical value
 $v = F_{n}(c)$ for some $c \in D(\eta) \cap \Omega_{F_{n}}$.  Then
 we can lift the contraction by $F_{n}$.  It follows that $D(\eta)$
 will contract to some point in $D(\eta)$. Let $z \in D(\eta) \cap P_{F_{n}}'$.
 In this case, $F_{n}(z) \in D(\gamma)$. By (7) of proposition~\ref{basic construction},
 it follows that  $F_{n}(z) \in \gamma_{n}^{*}$, where $\gamma_{n}$ is
 a curve segment which is attached to the  point $1$ and lies in the
 outside of the unit disk such that $F_{n}(\gamma_{n}) \subset \partial \Delta$.
Now let $\xi$ be one of the shortest simple closed $(\phi_{m},
F_{n}^{-1}(P_{F_{n}}'))-geodesics$ which are homotopic to $\gamma$
in $S^{2} - P_{F_{n}}'$. It follows that $\xi \cap (\partial \Delta
\cup \gamma_{n}^{*}) \ne \emptyset$. This implies that $F_{n}(\xi)
\cap \partial \Delta \ne \emptyset$. But on the other hand,
$\|\xi\|_{\phi_{m}, F_{n}^{-1}(P_{F_{n}}')}$ goes to $0$ as
$\|\gamma\|_{\phi_{m}, P_{F_{n}}'}$ goes to $0$ by Theorem A.3, and
so does $\|F_{n}(\xi)\|_{\phi_{m-1}, P_{F_{n}}'}$. This is a
contradiction with Lemma~\ref{lower bound of symmetric geodesic}.

In the second case, the contraction of $D(\gamma)$ can not be lifted
to a contraction of $D(\eta)$. This implies that there is a point $w
\in \Omega_{F_{n}}\cap D(\eta) \subset \Omega_{F_{n}}  -
\overline{\Delta}$ such that $F_{n}(w) \in D(\gamma)$. As before,
this implies that $F_{n}(w) \in \gamma_{n}^{*}$. By using the same
argument as above, we  get a contradiction again.

\end{proof}

\begin{lemma}\label{no Thurston obstructions}
$F_{n}$ has no Thurston obstructions in $S^{2} - P_{F_{n}}'$ for
every $n$ large enough.
\end{lemma}

\begin{proof}
First let us prove that $F_{n}$ has no Thurston obstructions in
$S^{2} - P_{F_{n}}$.  Suppose $\Gamma$ is a $F_{n}-stable$ family
which consists of all the short simple closed geodesics. By
Lemma~\ref{lower bound of symmetric geodesic}, if  $\gamma \in
\Gamma$, then $\gamma$ is disjoint from the unit circle. Since
$F_{n}$ is symmetric, therefore, the symmetric image of $\gamma$
about the unit circle, $\gamma^{*}$, must also belong to $\Gamma$.
We order the curves in $\Gamma$ as $\{\gamma_{1}, \cdots,
\gamma_{l}, \gamma_{1}^{*}, \cdots, \gamma_{l}^{*}\}$, where
$\gamma_{i} \subset \Delta$,  and $\gamma_{i}^{*}$ is the symmetric
image of $\gamma_{i}$ about the unit circle, $1 \le i \le l$. Now
let $A$ be the associated Thurston linear transformation matrix of
$\Gamma$(see \cite{DH} or $\S5$ for the definition). By
Lemma~\ref{invariance}, any non-peripheral component of
$F_{n}^{-1}(\gamma_{i})$ must be homotopic to one of the curves in
$\gamma_{i}, 1 \le i\le l$, and by the same reason, any
non-peripheral component of $F_{n}^{-1}(\gamma_{i}^{*})$ must be
homotopic to one of the curves in $\gamma_{i}^{*}, 1 \le i\le l$. It
follows that $\{\gamma_{1}, \cdots, \gamma_{l}\}$ is a $f-stable$
family. Let $B$ be its associated Thurston linear transformation
matrix. Then we have
$$
                   A =
                   \begin{pmatrix}
                   B  & 0 \\
                   0  & B
                   \end{pmatrix}
$$
Since $f$ has no Thurston obstructions outside the unit disk, so
$\|B\| < 1$.  Therefore, $\|A\|<1$.

Now let us prove that $F_{n}$ has no Thurston obstructions in $S^{2}
- P_{F_{n}}'$.  By the choice of the infinity, Let us assume that
$P_{F_{n}}' \ne P_{F_{n}}$, for otherwise, the lemma has been
proved.  Let us suppose that $F_{n}$ has Thurston obstructions in
$S^{2} - P_{F_{n}}'$.  Since $F_{n}$ has no Thurston obstructions in
$S^{2} - P_{F_{n}}$,  it follows that any short simple closed
$(\phi_{m}, P_{F_{n}}')-geodesics$ $\gamma \subset S^{2} -
P_{F_{n}}'$ must be homotopic to a point in $S^{2} - P_{F_{n}}$.
This implies that there are exactly two short simple closed
geodesics in $S^{2} - P_{F_{n}}'$, say $\gamma$ and $\gamma^{*}$,
such that $\gamma$ is contained in the outside of the unit disk, and
$D(\gamma)$ contains exactly two distinct points in $P_{F_{n}}'$,
one is the infinity, and the other one, say $x$, is a point in
$P_{F_{n}}$. By the same argument as before, we can show that any
non-peripheral component of $F_{n}^{-1}(\gamma)$ is contained in the
outside of the unit disk, and hence homotopic to $\gamma$ in $S^{2}
- P_{F_{n}}'$.  Similarly, any non-peripheral component of
$F_{n}^{-1}(\gamma^{*})$ is contained in the inside of the unit
disk, and hence homotopic to $\gamma^{*}$ in $S^{2} - P_{F_{n}}'$.
It follows that the associated Thurston linear transformation matrix
is a $2 \times 2$ diagonal matrix, and hence equal to the identity
matrix. This implies that there is a simple closed curve $\gamma'$
which is homotopic to $\gamma$ in  $S^{2} - P_{F_{n}}'$  and $F_{n}:
\gamma' \to \gamma$ is a homeomorphism. Now continuously contract
$\gamma$ to $x$. Since $D(\gamma) -\{x\}$ contains no critical
values of $F_{n}$, it follows that the contraction can be lifted to
a contraction of $\gamma'$, which then must contract to $x$ too.
This implies that $F_{n}(x) = x$. But this is a contradiction with
our choice of the infinity. The proof of the lemma is completed.
\end{proof}

\subsection{ The Compactness of $\{G_{n}\}$ and Bounded Geometry of $P_{G_{n}}$}

\subsubsection{The sequence of Blaschke products $\{G_{n}\}$ }
For $n$ large enough, by Lemma 2.4, $F_{n}$ has no Thurston
obstructions in $S^{2} - P_{F_{n}}'$.  By Thurston's
characterization theorem on $\emph{postcritically finite}$ rational
maps, it follows that there is a Blaschke product $G_{n}$ which is
$\emph{combinatorially equivalent}$ to $F_{n}$ rel $P_{F_{n}}'$(see
\cite{DH}, or $\S5$). That is to say, there is a pair of
homeomorphisms $\phi_{n}, \psi_{n}$ of the sphere which are isotopic
to each other rel $P_{F_{n}}'$, such that $G_{n} = \phi_{n} \circ
F_{n} \circ \psi_{n}^{-1}$.  In this section, we will show that the
sequence $\{G_{n}\}$ is contained in a compact set of the space of
all the rational maps of degree $2d -1$, and moreover, the geometry
of $P_{G_{n}}$ is uniformly bounded.

\subsubsection{Analysis of short simple closed geodesics in
${\Bbb P}^{1}- (X_{L}^{n} \cup P_{n})$} We would like to mention
that all the proofs in this subsection does not rely on the
condition that $\theta$ is of bounded type. The only arithmetic
condition of $\theta$ we used is that it is an irrational number.

Let $L \ge 1$ be an integer. Define
\begin{equation}\label{finite set}
X_{L}^{n} = \{G_{n}^{k}(x) \big{|}\:x \in \Omega_{G_{n}}, -L \le k
\le L\} \cap \partial \Delta
\end{equation}
and
\begin{equation}\label{outside pc}
P_{n} = (P_{G_{n}} - \partial \Delta ) \cup \{0, \infty\}.
\end{equation}

Let $I \subset \partial \Delta$ be an arc segment(it may be open,
closed, or half open and half closed). Define
$$
\sigma_{n}(I) = \frac{|I \cap P_{G_{n}}|} {|\partial \Delta \cap
P_{G_{n}}|}.
$$
Since $P_{G_{n}} \cap \partial \Delta$ consists of a periodic orbit
and since $G_{n}|\partial \Delta: \partial \Delta \to \partial
\Delta$ is a homeomorphism, it follows that $\sigma_{n}$ is
$G_{n}-$invariant, i.e., for any $I \subset
\partial \Delta$, $$\sigma_{n}(I)  = \sigma_{n}(G_{n}(I)).$$ Let $x, y
\in
\partial \Delta$  be two distinct points. They separate $\partial
\Delta$ into two arc segments  $I$ and $J$.  Let $\overline{I}$ abd
$\overline{J}$ denote the closure of $I$ and $J$, respectively.
Define
$$
d_{\sigma_{n}}(x, y) = \min\{\sigma_{n}(\overline{I}),
\sigma_{n}(\overline{J})\}.
$$
It is clear that
\begin{equation}\label{tri-ine}
d_{\sigma_{n}}(x, z) \le d_{\sigma_{n}}(x, y)+ d_{\sigma_{n}}(y, z).
\end{equation}

\begin{lemma}\label{metric}
For any $k \ge 1$, there is an $\epsilon > 0$ such that for any $x
\in \partial \Delta$, the following inequality holds for all $n$
large enough
$$
d_{\sigma_{n}}(x, G_{n}^{k}(x)) \ge \epsilon.
$$
\end{lemma}
\begin{proof}
Assume that $n$ is large wnough. Then $x$ and $G_{n}^{k}(x)$
separate $\partial \Delta$ into two arc intervals $I$ and $J$. Since
$\theta_{n}$ converges to $\theta$, there is an  $m \ge 1$ dependent
only on $\theta$ and $k$ such that for all $n$ large enough,
$$
\partial \Delta \subset \bigcup_{i=0}^{m} G_{n}^{ik}(I) \hbox{ and }
\partial \Delta \subset \bigcup_{i=0}^{m} G_{n}^{ik}(J).
$$
Since $\sigma_{n}$ is $G_{n}-$invariant, it follows that
$$
\min\{\sigma_{n}(I), \sigma_{n}(J)\} \ge 1/(m+1).
$$
This implies Lemma~\ref{metric}.
\end{proof}

As before, let $N = |P_{F}' - \overline{\Delta}|$. It follows that
for every $n$ large enough,
\begin{equation}\label{N-value}
|P_{n}|= 2 N.
\end{equation}
\begin{lemma}\label{tech-1}
Let $L  \ge  N +2$ and $M \ge 1$ be some integers.  Then for any $1
\le k \le M$,  and every $n$ large enough,  $X_{L}^{n}\cup P_{n}
\subset  G_{n}^{-k}(X_{L+M}^{n}\cup P_{n})$. Moreover,  the map
$$G^{k}: {\Bbb P}^{1} - G_{n}^{-k}(X_{L+M}^{n}\cup P_{n}) \to {\Bbb
P}^{1} - (X_{L+M}^{n}\cup P_{n})$$ is a holomorphic covering map.
\end{lemma}
\begin{proof}
Let $z \in X_{L}^{n}\cup P_{n}$ and $1 \le k \le M$. We have two
cases. In the first case, $z \in X_{L}^{n}$. It follows  from
(\ref{finite set}) that $G_{n}^{k}(z) \in X_{L+M}^{n}$. In the
second case, $z \in P_{n}$. Then from(7) of Proposition~\ref{basic
construction},    there is some critical point $c \in
\Omega_{G_{n}}$ and some integer $0 \le i \le N+1$ such that $$z =
G_{n}^{i}(c).$$ Since $L \ge N+2$, it follows from (\ref{finite
set}) that $G_{n}^{k}(z) \in X_{L+M}^{n}$. This proves the first
assertion.

The second assertion follows since  $L \ge N + 2$, and therefore,
for any $c \in \Omega_{G_{n}}$, the forward orbit segment
$$\{G_{n}^{i}(c)\:\big{|}\: 1 \le i \le M\}$$  is contained in $X_{L+M}^{n} \cup
P_{n}$.
\end{proof}

For $L> 0$ and $n$ large enough, set
$$
r_{L}^{n} = \max\{\sigma_{n}(\overline{I})\:\big{|}\:I \:\hbox{is an
interval component of}\: \partial \Delta - X_{L}^{n}\}.
$$
\begin{lemma}\label{ques-1}
Let $\epsilon > 0$ be an arbitrary number. Then there exist $L'$ and
$N'$ such that $r_{L}^{n} < \epsilon$ provided that $L > L'$ and $n
> N'$.
\end{lemma}
\begin{proof}
Let us consider the combinatorial model $F_{n}$ instead of $G_{n}$.
That is, replace $G_{n}$ by $F_{n}$ in the definitions of
$X_{L}^{n}$, $\sigma_{n}$, and $r_{L}^{n}$. Let us still keep the
same notations.

For $\epsilon > 0$ given, let $K$ be the least integer such that $K
\ge 1/\epsilon$. Since $0< \theta< 1$ is irrational, there is a $0<
\delta < 1$ which depends only on $\theta$ such that for any  closed
arc segment $I \subset \partial \Delta$ with $|I| < \delta$, the
$K+1$ arc segments $e^{2 \pi i k \theta} I, 0 \le k \le K $ are
disjoint. For such $\delta$, there is an integer $L'$  which depends
only on $\delta$ and $\theta$ such that for all $L> L'$, every
component of
$$
\Xi = \partial \Delta - \{e^{2 \pi i k \theta}\:\big{|}\: -L \le k
\le L \}
$$
has Euclidean length less than $\delta/2$. It follows that for every
component $I$ of $\Xi$, the closure of the arc segments $e^{2 \pi i
k \theta} I, 0 \le k \le K$,  are disjoint. Since $\theta_{n} \to
\theta$ as $n \to \infty$, it follows that there is an $N'> 0$ such
that for all $n > N'$, and any component $I$ of
$$
\Xi_{n} = \partial \Delta - \{e^{2 \pi i k \theta_{n}}\:\big{|}\: -L
\le k \le L \},
$$
the closure of the $K+1$ arc segments $e^{2 \pi i k \theta_{n}} I, 0
\le k \le K $, are disjoint. Since $\sigma_{n}$ is
$F_{n}-$invariant, we have
$$
\sigma_{n}(\overline{I}) = \sigma_{n}(e^{2 \pi i
 \theta_{n}} \overline{I}) = \cdots = \sigma_{n}(e^{2 \pi i K \theta_{n}} \overline{I}).
$$
From the disjointness, we have
$$
\sigma_{n}(\overline{I}) + \sigma_{n}(e^{2 \pi i
 \theta_{n}} \overline{I}) + \cdots + \sigma_{n}(e^{2 \pi i K \theta_{n}} \overline{I}) \le
 1.
$$
It  follows that $\sigma_{n}(\overline{I}) < 1/K \le \epsilon$. The
lemma then follows since $I$ is an arbitrary component of $\Xi_{n}$
and $\{e^{2 \pi i k \theta_{n}}\:\big{|}\: -L \le k \le L \}$ is
contained in $X_{L}^{n}$.

\end{proof}

Recall that $N = |P_{F}' - \overline{\Delta}|$. As a consequence of
Lemma~\ref{metric} and ~\ref{ques-1}, we have
\begin{corollary}\label{L-value}
There exist integers $L_{0}$ and $N_{0}$ such that when $L
> L_{0}$ and $n > N_{0}$, the inequality
$$
d_{\sigma_{n}}(x, G_{n}^{k}(x)) > 3 r_{L}^{n}
$$
holds  for any $x \in
\partial \Delta$ and every $1 \le k \le N+1$.
\end{corollary}
For a simple closed geodesic $\xi$ which intersects the unit
circle, we use $D(\xi)$ to denote the bounded component of $S^{2}
-\gamma$.

\begin{lemma}\label{dico}
Let $L_{0}$ be the number in Corollary~\ref{L-value}. Let $L >
\max\{N+2, L_{0}\}$. Then there is a $\delta
> 0$ and $1< C < \infty$  such that for every $n$ large enough and
any simple closed geodesic $\gamma \subset{\Bbb P}^{1} - (X_{L}^{n}
\cup P_{n})$ with $\|\gamma\|_{{\Bbb P}^{1} - (X_{L}^{n} \cup
P_{n})} \le \delta$ and $\gamma \cap \partial \Delta \ne \emptyset$,
one of the following two cases must be true: \begin{itemize}
\item[1.] $|(\partial \Delta - D(\gamma)) \cap X_{L}^{n}| \ge 2$ and
$|D(\gamma) \cap X_{L}^{n}| \ge 2$, \item[2.] there is a short
simple closed geodesic $\eta \subset {\Bbb P}^{1} - (X_{L+N+2}^{n}
\cup P_{n})$ with $$\|\eta\|_{{\Bbb P}^{1} - (X_{L+N+2}^{n} \cup
P_{n})} \le C \|\gamma\|_{{\Bbb P}^{1} - (X_{L}^{n} \cup P_{n})}$$
such that
 $|(\partial \Delta - D(\eta)) \cap X_{L+N+2}^{n}| \ge 2$ and
 $|D(\eta) \cap X_{L+N+2}^{n}| \ge 2$.
 \end{itemize}
\end{lemma}

\begin{proof}
Let $\gamma \subset {\Bbb P}^{1} - (X_{L}^{n} \cup P_{n})$ be a
short simple closed geodesic with $$\gamma \cap \partial \Delta \ne
\emptyset \hbox{ and } \|\gamma\|_{{\Bbb P}^{1} - (X_{L}^{n} \cup
P_{n})} \le \delta$$ for some $\delta > 0$. Assume that the first
case does not hold, that is, either $|(\partial \Delta - D(\gamma))
\cap X_{L}^{n}| < 2$ or $|D(\gamma) \cap X_{L}^{n}| < 2$.  Let us
prove the second case must hold. Let $\gamma \cap \partial \Delta =
\{x, y\}$. It follows that
\begin{equation}\label{dis}
d_{\sigma_{n}}(x, y) \le 2 r_{L}^{n}.
\end{equation}
See Figure 2 for an illustration(Here $|D(\gamma) \cap X_{L}^{n}|
=1$).
\begin{figure}
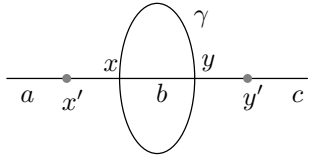

\bigskip
\begin{center}
\centertexdraw { \drawdim cm \linewd 0.02 \move(0 1)

\move(-2 -0.9) \lvec(2 -0.9)

\move(-1.8 -1.2) \htext{$a$} \move(0 -1.2) \htext{$b$} \move(1.8
-1.2) \htext{$c$}

\move(0.5 -0.2) \htext{$\gamma$}

 \move(-0.7 -0.8) \htext{$x$} \move(0.6
-0.8) \htext{$y$} \move(0 -0.9) \lellip rx:0.5 ry:1
 \move(-1.25 -1.3) \htext{$x'$}
 \move(1.15 -1.3) \htext{$y'$}
 \move(1.2 -0.9) \fcir f:0.5 r:0.06
 \move(-1.2 -0.9) \fcir f:0.5 r:0.06
 }
\end{center}
\vspace{0.2cm} \caption{$X_{L}^{n} \cap D(\gamma) = \{b\}$ with $a,
b, c $ being three adjacent points in $X_{L}^{n}$}
\end{figure}

Let us assume that $\delta$ is so small  that the simple closed
geodesics generated in the following are all short enough. In
Lemma~\ref{tech-1}, taking $M = N+2$, then for any $1 \le k \le
N+2$, we have
$$
{\Bbb P}^{1} - G_{n}^{-k}(X_{L+N+2}^{n}\cup P_{n}) \subset {\Bbb
P}^{1} - ( X_{L}^{n} \cup P_{n}).
$$
Since $|G_{n}^{-k}(X_{L+N+2}^{n}\cup P_{n}) - (X_{L}^{n} \cup
P_{n})|$ depends only on $L, N$, and $d$,  by Theorem A.3, there is
a constant $C$ dependent only on $L, N$, and  $d$, such that for
each $1 \le k \le N+2$, there is a simple closed geodesic $$\xi_{k}'
\subset {\Bbb P}^{1} - G_{n}^{-k}(X_{L+N+2}^{n}\cup P_{n})$$ which
is homotopic to $\gamma$ in ${\Bbb P}^{1} - ( X_{L}^{n} \cup P_{n})$
such that
$$
\|\xi_{k}'\|_{{\Bbb P}^{1} - G_{n}^{-k}(X_{L+N+2}^{n}\cup P_{n})}
\le C \|\gamma\|_{{\Bbb P}^{1} - (X_{L}^{n}\cup P_{n})}.
$$
When $\delta$ is small, by Lemma~\ref{tech-1} and Theorem A.2,
$G_{n}^{k}(\xi_{k}')$ covers a simple closed geodesic  $\xi_{k}$ in
${\Bbb P}^{1} - ( X_{L+N+2}^{n}\cup P_{n})$, and hence
$$
\|\xi_{k}\|_{{\Bbb P}^{1} - (X_{L+N+2}^{n}\cup P_{n})} \le
\|\xi_{k}'\|_{{\Bbb P}^{1} - G_{n}^{-k}(X_{L+N+2}^{n}\cup P_{n})}.
$$

Now it suffices to prove  that there is some $\xi_{k}, 1 \le k \le N
+ 2$,  such that $$|(\partial \Delta - D(\xi_{k})) \cap
X_{L+N+2}^{n}| \ge 2$$ and $$|D(\xi_{k}) \cap X_{L+N+2}^{n}| \ge
2.$$ Assume at least one of the above two inequalities were not
true. We will get a contradiction as follows.

We first claim that for every $n$ large enough,  each component of
$\partial \Delta  - X_{L+N+2}^{n}$ intersects at most one of
$\xi_{k}, 1 \le k \le N +2$,  and in particular,
\begin{equation}\label{ne}
 \xi_{ i } \ne \xi_{j}
 \end{equation}
  for $1\le i \ne j \le N
+2$.     Suppose this is not true. Then there exist $1 \le i < j \le
N+2$ and  a component of $\partial \Delta  - X_{L+N+2}^{n}$, say
$I$, such that $\xi_{i}\cap I \ne \emptyset$ and $\xi_{j}\cap I \ne
\emptyset$.

Recall that $\xi_{i}'$ and $\xi_{j}'$ cover $\xi_{i}$ and $\xi_{j}$,
respectively.  Let $x' \in \xi_{i}' \cap
\partial \Delta$ and $y' \in \xi_{j}' \cap
\partial \Delta$ such that $G_{n}^{i}(x') \in \xi_{i}\cap I $ and
$G_{n}^{j}(y') \in \xi_{j}\cap I $. Since both $\xi_{i}'$ and
$\xi_{j}'$ are homotopic to $\gamma$ in ${\Bbb P}^{1} - ( X_{L}^{n}
\cup P_{n})$, it follows from (\ref{dis}) that
$$
d_{\sigma_{n}}(x', y') \le 2r_{L}^{n}.
$$
See Figure 2 for an illustration(Since $x'$ and $y'$ belong to the
arc interval $(a, c)$ whose $\sigma_{n}-$length is not more than $2
r_{L}^{n}$).

Since $\sigma_{n}$ is $G_{n}-$invariant, we have
\begin{equation}\label{ime-1}
d_{\sigma_{n}}(G_{n}^{i}(x'), G_{n}^{i}(y')) \le 2 r_{L}^{n}.
\end{equation}
On the other hand, since $G_{n}^{j}(y'), G_{n}^{i}(x') \in I$, we
get
\begin{equation}\label{ime-2}
d_{\sigma_{n}}(G_{n}^{j}(y'), G_{n}^{i}(x')) =
d_{\sigma_{n}}(G_{n}^{j-i}(G_{n}^{i}(y')), G_{n}^{i}(x')) \le
\sigma_{n}(I) \le  r_{L}^{n}. \end{equation}
 It follows from (\ref{tri-ine}), (\ref{ime-1}), and (\ref{ime-2}), that
$$
d_{\sigma_{n}}(G_{n}^{j}(y'), G_{n}^{i}(y')) \le 3 r_{L}^{n}.
$$ This is a contradiction with the definition of $L_{0}$ in Corollary~\ref{L-value}
and the choice of $L$. The claim
follows.

Now there are two cases. In the first case, all the domains
$D(\xi_{i}), 1 \le i \le N+2$ are disjoint. Since each component of
$\partial \Delta  - X_{L+N+2}^{n}$ intersects at most one of
$\xi_{i}, 1 \le i \le N +2$,  it follows from the claim that for
every $1 \le i \le N+2$,
\begin{equation}\label{sec-c}
|(\partial \Delta - D(\xi_{i})) \cap X_{L+N+2}^{n}| \ge 2.
\end{equation}
This is because otherwise, there would be two domains $D(\xi_{i})$
and $D(\xi_{j})$ with $1 \le i \ne j \le N+2$ such that one is
contained in the other one, and this is impossible since we have
assumed that all the domains $D(\xi_{i})$, $1 \le i \le N+2$, are
disjoint in this case.
\begin{figure}
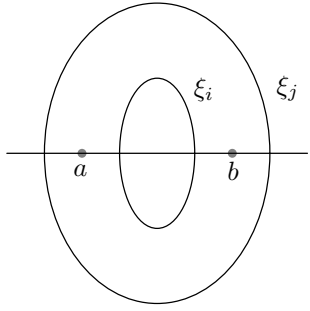

\bigskip
\begin{center}
\centertexdraw { \drawdim cm \linewd 0.02 \move(0 1)

\move(-1 -0.9) \fcir f:0.5 r:0.06 \move(1 -0.9) \fcir f:0.5 r:0.06
\move(-2 -0.9) \lvec(2 -0.9)

\move(-1.1 -1.2) \htext{$a$}  \move(0.95 -1.25) \htext{$b$}

\move(0.5 -0.2) \htext{$\xi_{i}$} \move(1.6 -0.2) \htext{$\xi_{j}$}
 \move(0 -0.9) \lellip rx:0.5 ry:1

\move(0 -0.9) \lellip rx:1.5 ry:2

}

\end{center}
\vspace{0.2cm} \caption{$D(\xi_{j})$ contains at two points in
$X_{L+N+2}$}
\end{figure}
\begin{figure}
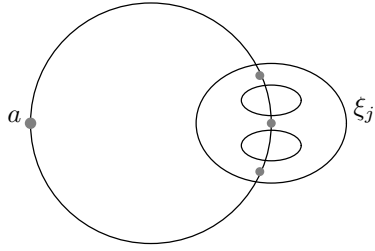

\bigskip
\begin{center}
\centertexdraw { \drawdim cm \linewd 0.02 \move(0 1)

\move(0 -1.2) \lcir r:1.6 \move(1.6 -1.2) \lellip rx:1 ry:0.8
\move(1.6 -0.9) \lellip rx:0.4 ry:0.2 \move(1.6 -1.5) \lellip rx:0.4
ry:0.2  \move(1.45 -0.57) \fcir f:0.5 r:0.06 \move(1.6 -1.2) \fcir
f:0.5 r:0.06 \move(-1.6 -1.2) \fcir f:0.5 r:0.08 \move(-1.9 -1.2)
\htext{$a$} \move(2.7 -1.2) \htext{$\xi_{j}$} \move(1.45 -1.85)
\fcir f:0.5 r:0.06
 }
\end{center}
\vspace{0.2cm} \caption{$D(\xi_{j})$ contains all $D(\xi_{k})$ for
$k \ne j$}
\end{figure}
From (\ref{sec-c}), it follows that
$$
|D(\xi_{i}) \cap X_{L+N+2}^{n}| \le 1
$$
for every $1 \le i \le N+2$. For otherwise, the lemma has been
proved. Since $\xi_{i}$ is non-peripheral, it follows that
$D(\xi_{i}) \cap P_{n}$ is non-empty, and by the symmetric property
of $\xi_{i}$ and $P_{n}$, we have
$$
|D(\xi_{i}) \cap P_{n}| \ge 2.
$$
We thus get
$$
|P_{n}| \ge \sum _{1 \le i \le N+2}|D(\xi_{i}) \cap
P_{n}| \ge 2(N+2).
$$
This is a contradiction with (\ref{N-value}).

In the second case, there are two domains $D(\xi_{i})$ and
$D(\xi_{j})$ such that
$$
D(\xi_{i}) \subset D(\xi_{j})
$$ where $1 \le i \ne j\le N+2$.  By the claim which we proved previously,
it follows that none of the components of $\partial \Delta -
X_{L+N+2}^{n}$ intersect both $\xi_{i}$ and $\xi_{j}$.  This implies
that
$$
|D_{\xi_{j}} \cap X_{L+N+2}^{n}| \ge 2.
$$
See Figure 3 for an illustration.  We thus get by assumption that
\begin{equation}\label{single}
|(\partial \Delta - D(\xi_{j})) \cap X_{L+N+2}^{n}| \le 1.
\end{equation}
It follows  that all the  domains $D(\xi_{k}), k \ne j$, are
contained in $D(\xi_{j})$. This is because if some $D(\xi_{k})$, $k
\ne j$, is not contained in $D(\xi_{j})$,  from (\ref{single}), it
follows that one component of $\partial \Delta - X_{L+N+2}^{n}$
would intersect both $\xi_{j}$ and  $\xi_{k}$, but  this again
contradicts with the claim we previously proved. See Figure 4 for an
illustration. In this figure, we assume that $(\partial \Delta -
D(\xi_{j})) \cap X_{L+N+2}^{n}$ contains a single point $a$.

Now we  claim that all the domains $D(\xi_{k}), k \ne j$, are
disjoint with each other. In fact, if $D(\xi_{i'}) \subset
D(\xi_{j'})$ for some $i'$ and $j'$ such that $i' \ne j', i' \ne j,
j' \ne j$, then by the same  argument as above, we get that
$D(\xi_{j'})$ contains all the other domains $D(\xi_{k}), 1 \le k
\le N +2, k \ne j'$. In particular,
$$
D(\xi_{j}) \subset D(\xi_{j'})
$$ and hence $\xi_{j} = \xi_{j'}$. This
is a contradiction with (\ref{ne}).  The claim follows. Since for
any $1 \le k \le N+2, k \ne j$,  $D(\xi_{k}) \subset D(\xi_{j})$ and
each component of $\partial \Delta - X_{L+N+2}^{n}$ can not
intersect both $\xi_{j}$ and $\xi_{k}$, it follows that for every $1
\le k \le N+2$ and $k \ne j$,  $$|(\partial \Delta - D(\xi_{k}))
\cap X_{L+N+2}^{n}| \ge 2.$$ See Figure 4 for an illustration.

By assumption, we have
$$|D(\xi_{k}) \cap X_{L+N+2}^{n}| \le 1.$$ As before, it follows that
$$|D(\xi_{k}) \cap P_{n}| \ge 2$$ for every $1 \le k \le N+2, k \ne
j$. This implies that $$|P_{n}| \ge   \sum _{1 \le k \le N+2, k \ne
j}|D(\xi_{k}) \cap P_{n}| \ge 2(N+1).$$ This is a contradiction with
(\ref{N-value}). The proof of  Lemma~\ref{dico} is completed.

\end{proof}

\begin{lemma}\label{L-N}
For any $L > 0$ there is an $\epsilon > 0$ such that for every $n$
large enough,  we have
$$
d_{\sigma_{n}}(x, y) > \epsilon
$$
for any two distinct points $x, y \in X_{L}^{n}$.
\end{lemma}
\begin{proof}
As in the proof of Lemma~\ref{ques-1}, we may consider the
combinatorial model $F_{n}$ instead of $G_{n}$.  That is,
$$
 X_{L}^{n} = \{F_{n}^{k}(x) \big{|}\:x \in
\Omega_{F_{n}}, -L \le k \le L\} \cap \partial \Delta.
$$
Let us also define
$$
X_{L} = \{F^{k}(x) \big{|}\:x \in \Omega_{F}, -L \le k \le L\} \cap
\partial \Delta.
$$
Now for $L > 0$ given, $X_{L}^{n} \to X_{L}$ as $n \to \infty$.
 Let $I$ be the smallest component of $\partial \Delta - X_{L}$.
 Since $\theta$ is an irrational number,  there is a least integer $m > 0$ such that
$$
\partial \Delta \subset \bigcup_{0\le l \le m} e^{2 \pi i l
\theta}I.
$$
Since each $e^{2 \pi i l \theta} I$ is open, $0 \le l \le m$, it
follows that there is an $N_{1}> 0$, such that for all $n
> N_{1}$, and any component $I$ of $\partial \Delta - X_{L}^{n}$, we have
$$
\partial \Delta \subset \bigcup_{0\le l \le m} e^{2 \pi i l
\theta_{n}}I.
$$
Since $\sigma_{n}$ is $F_{n}-$invariant, it follows that for any $x$
and $y$ in $X_{L}^{n}$,
$$
d_{\sigma_{n}}(x, y) > 1/(m+1)
$$
for all $n > N_{1}$.
\end{proof}

\begin{lemma}\label{trans}
For any $0< \epsilon <1$, there exist $0<\mu<1/2$ and an integer
$L(\epsilon) \ge 1$ dependent only on $\epsilon$ and $\theta$ such
that for all $n$ large enough and any arc segment $I$ with $\epsilon
\le \sigma_{n}(I) \le 1 - \epsilon$, there is an integer $1 \le l
\le L(\epsilon)$ such that the following inequalities hold:
\begin{itemize}
\item[1.]  $\sigma_{n}(I \cap G_{n}^{l}(I)) > \mu \epsilon$,
\item[2.]  $\sigma_{n}(I - G_{n}^{l}(I)) > \mu \epsilon$,
\item[3.]  $\sigma_{n}( G_{n}^{l}(I)-I) > \mu \epsilon$.
\end{itemize}
\end{lemma}
\begin{proof}
As in the proofs of Lemma~\ref{ques-1} and \ref{L-N},  let us
consider the combinatorial model $F_{n}$ instead of $G_{n}$. In
particular, in the definition of $\sigma_{n}$,  $G_{n}$ is replaced
by $F_{n}$, and $\sigma_{n}$ is thus $F_{n}-$invariant.

Claim 1:  For any $0< \delta < 1$, there exist $0<\nu<1/2$ and an
integer $K(\delta) \ge 1$ dependent only on $\delta$ and $\theta$
such that for any arc segment $I$ with $\delta \le |I| \le 2\pi -
\delta$, there is an integer $1 \le l \le K(\delta)$ such that the
following inequalities hold:
\begin{itemize}
\item[1.]  $|I \cap e^{2 \pi i l\theta} I | > \nu\delta$,
\item[2.]  $|I - e^{2 \pi i l\theta} I | >\nu\delta$,
\item[3.]  $|e^{2 \pi i l\theta} I -I| > \nu\delta$.
\end{itemize}
By using the fact that $\theta$ is an irrational number, the claim
can be proved by a compacting argument. We  leave the details to the
reader.

Claim 2:   For any $0< \delta < 1$, there exist $0<\nu<1/2$ and an
integer $K(\delta) \ge 1$ dependent only on $\delta$ such that for
all $n$ large enough and any arc segment $I$ with $\delta \le |I|
\le 2\pi - \delta$, there is an integer $1 \le l \le K(\delta)$ such
that the following inequalities hold:
\begin{itemize}
\item[1.]  $|I \cap e^{2 \pi i l\theta_{n}} I | > \nu\delta$,
\item[2.]  $|I - e^{2 \pi i l\theta_{n}} I | >\nu\delta$,
\item[3.]  $|e^{2 \pi i l\theta_{n}} I -I| > \nu\delta$.
\end{itemize}
Claim 2 follows from Claim 1 and the fact that  $\theta_{n} \to
\theta$ as $n \to \infty$.

Claim 3: For any $\epsilon
> 0$, there is a $\delta
> 0$ dependently only on $\epsilon$ and $\theta$ such that for all
$n$ large enough and  any arc segment $I$,  $\sigma_{n}(I) > \delta$
provided that $|I| > \epsilon$, and moreover, the converse is also
true: $|I| > \delta$,  provided that $\sigma_{n}(I) > \epsilon$.

Let us prove the first assertion. Suppose $|I| > \epsilon$. Then
there is an integer $K \ge 1$ dependent only on $\theta$ and
$\epsilon$ such that  for all $n$ large enough, the following holds:
$$
\partial \Delta \subset \bigcup_{0 \le l \le K} e^{2 \pi i
l\theta_{n}} I.
$$
Since $\sigma_{n}$ is $F_{n}-$invariant, it follows that
$\sigma_{n}(I) > 1/(K + 1)$.  This proves the first assertion.

Let us prove the second assertion now. It is sufficient to prove
that $\sigma_{n}(I) < \epsilon$ provided that $|I| < \delta$. For
$\epsilon > 0$ given, let $K$ be the least integer such that $K >
1/\epsilon$.  Since $\theta$ is an irrational number, it follows
that there is a $\delta > 0$ such that for any arc segment $I$ with
$|I| < \delta$, the closure of the arc segments
$$
e^{ 2 \pi i l \theta} I, 0 \le l \le K
$$
are disjoint with each other. Since $\theta_{n} \to \theta$ as $n
\to \infty$, it follows that for all $n$ large enough, and any arc
segment $I$ with $|I| < \delta$, the closure of the arc segments
$$
e^{ 2 \pi i l \theta_{n}} I, 0 \le l \le K
$$
are disjoint with each other. Since $\sigma_{n}$ is
$F_{n}-$invariant, it follows that
$$
\sigma_{n}(I) < 1/(1 + K ) < \epsilon.
$$
This completes the proof of Claim 3. Now the lemma follows directly
as a consequence of Claim 2 and 3.
\end{proof}

\begin{lemma}\label{tech-2}
Let $\tau > 0$. Then there exist $K > 0$ and $N_{0}> 0$ dependent
only on $\tau$ and $\theta$, such that  for all $L \ge K$ and $n \ge
N_{0}$ and any arc segment $I$ with $\sigma_{n}(I) > \tau$, the
following inequality
$$
|I \cap X_{L}^{n}| \ge 2
$$
holds.
\end{lemma}
\begin{proof}
As in the proofs of Lemma~\ref{ques-1} and \ref{L-N},  let us
consider the combinatorial model $F_{n}$ instead of $G_{n}$.  Assume
that $\sigma_{n}(I) > \tau$ for some $\tau > 0$. From Claim 3 in the
proof of Lemma~\ref{trans}, there exist $\epsilon > 0$ and $N_{1}>0$
which depend only on $\theta$ and $\tau$ such that for all $n \ge
N_{1}$, the following inequality
$$
|I| > \epsilon
$$
holds provided that $\sigma_{n}(I) > \tau$.  For such $\epsilon >
0$, since $\theta$ is irrational, it follows that there exists a  $K
> 0$ which depends only on $\theta$ and
 $\epsilon$ such that for any $I \subset \partial \Delta$ with $|I|
> \epsilon/2$,
$$
|I \cap \{e^{2 \pi i l \theta}\:\big{|}\: -K \le l \le K\}| \ge 2.
$$
Since $\theta_{n} \to \theta$, it follows
that there exists an $N_{2}> 0$ which depends only on $K$ and
$\epsilon$, such that for all $n \ge  N_{2}$ and  any arc segment $I
\subset
\partial \Delta$ with $|I|
> \epsilon$,
$$
|I \cap \{e^{2 \pi i l \theta_{n}}\:\big{|}\: -K \le l \le K\}| \ge
2.
$$
Let $N_{0} = \max\{N_{1}, N_{2}\}$. Then for all $L \ge K$ and $n
\ge N_{0}$, we have
$$
|I \cap \{e^{2 \pi i l \theta_{n}}\:\big{|}\: -L \le l \le L\}| \ge
|I \cap \{e^{2 \pi i l \theta_{n}}\:\big{|}\: -K \le l \le K\}| \ge
2.
$$
The lemma follows.
\end{proof}
\begin{lemma}\label{ml}
For any $L $ large enough there is a $\delta > 0$  such that for all
$n$ large enough, the hyperbolic length of every simple closed
geodesic in ${\Bbb P}^{1} - (X_{L}^{n} \cup P_{n})$, which
intersects the unit circle, is greater than $\delta$.
\end{lemma}
The  proof is by contradiction. Assuming that the Lemma were not
true. The basic idea is to construct two short simple closed
geodesics so that they intersect with each other. This is realized
by first constructing two short simple closed geodesics $\eta'$ and
$\eta''$ which intersect with each other, but which belong to
different hyperbolic Riemann surfaces. The next step is to find a
common hyperbolic Riemann surface in which there exist two simple
closed geodesics $\xi'$ and $\xi''$ which are, respectively,
homotopic to $\eta'$ and $\eta''$, and most importantly,  separate
some set $Z \subset \partial \Delta$ in the same way as $\eta'$ and
$\eta''$. This implies that $\xi'$ and $\xi''$ must intersect with
each other. This is a contradiction with Theorem A.1(see Figure 5
for an illustration).
\begin{proof}
Let $L_{0}$ be the number defined in Corollary~\ref{L-value}.
Suppose  that $L > \max\{N +2, L_{0}\}$  and that $\gamma$ is a
simple closed geodesic in ${\Bbb P}^{1} - (X_{L}^{n} \cup P_{n})$,
which intersects the unit circle, and has length less than $\delta$.
By Lemma~\ref{dico} and replacing $L$ by $L + N + 2$ if necessary,
we will have a  short simple closed geodesic $\eta$ in ${\Bbb P}^{1}
- (P_{n} \cup X_{L}^{n})$ such that
$$
|D(\eta) \cap X_{L}^{n}| \ge 2
$$ and
$$
|(\partial \Delta -
D(\eta)) \cap X_{L}^{n}| \ge 2.
$$

Let $$I\subset \partial \Delta \cap D(\eta)$$ be the maximal closed
arc segment such that $\partial I \subset X_{L}^{n}$. Here $\partial
I$ denote the set of the two end points of $I$. Similarly, let $$J
\subset \partial \Delta - D(\eta)$$ be the maximal closed arc
segment such that $\partial J \subset X_{L}^{n}$. From
Lemma~\ref{L-N} and the above two inequalities, it follows that
there is an $\epsilon > 0$ which  depends only on $L$ and $\theta$
such that
\begin{equation}\label{com}
\min\{\sigma_{n}(I),
\sigma_{n}(J)\} \ge \epsilon
\end{equation}
for all $n$ large enough. For such $\epsilon$, let $0< \mu < 1/2$
and $L(\epsilon) \ge 1$ be the numbers given in Lemma~\ref{trans}.
Now take $\tau = \mu \epsilon$ in Lemma~\ref{tech-2} and let $K > 0$
be the value there. Let
$$
S = K +  L + L(\epsilon).
$$
By Lemma~\ref{trans} and (\ref{com}), there is an $0 < l <
L(\epsilon)$ such that the following inequalities hold for all $n$
large enough:
\begin{itemize}
\item[1.]  $\sigma_{n}(I \cap G_{n}^{l}(I)) > \tau$,
\item[2.]  $\sigma_{n}(I - G_{n}^{l}(I)) > \tau$,
\item[3.]  $\sigma_{n}( G_{n}^{l}(I)-I) > \tau$.
\end{itemize}
Let $Z = X_{S}^{n} \cap G_{n}^{-l}(X_{S}^{n})$. It follows that
$X_{K}^{n} \subset Z$. From  Lemma~\ref{tech-2} and the above three
inequalities, we have
\begin{itemize}
\item[i.]  $|I \cap G_{n}^{l}(I) \cap Z| \ge
2$,
\item[ii.]  $|(I - G_{n}^{l}(I)) \cap Z| \ge
2$,
\item[iii.]  $|( G_{n}^{l}(I) -I ) \cap Z| \ge
2$.
\end{itemize}

\begin{figure}
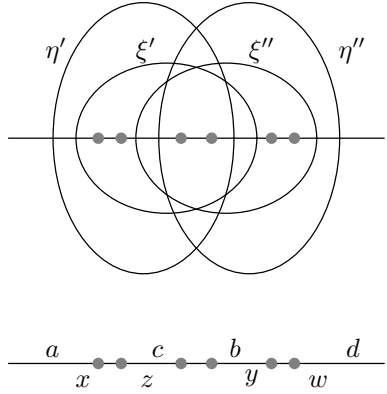

\bigskip
\begin{center}
\centertexdraw { \drawdim cm \linewd 0.02 \move(0 1)

\move(-2.5 -1) \lvec(2.5 -1)

\move(-0.7 -1) \lellip rx:1.2 ry:1.8 \move(0.7 -1) \lellip rx:1.2
ry:1.8

\move(-0.4 -1) \lellip rx:1.2 ry:1 \move(0.4 -1) \lellip rx:1.2 ry:1

\move(-2 0) \htext{$\eta'$} \move(1.9 0) \htext{$\eta''$}

\move(-0.8 0) \htext{$\xi'$}

\move(0.7 0) \htext{$\xi''$}

\move(-0.2 -1) \fcir f:0.5 r:0.08 \move(0.2 -1) \fcir f:0.5 r:0.08

\move(-1.3 -1) \fcir f:0.5 r:0.08 \move(-1 -1) \fcir f:0.5 r:0.08

\move(1.3 -1) \fcir f:0.5 r:0.08 \move(1 -1) \fcir f:0.5 r:0.08

\move(-2.5 -4) \lvec(2.5 -4)

\move(-0.2 -4) \fcir f:0.5 r:0.08 \move(0.2 -4) \fcir f:0.5 r:0.08

\move(-1.3 -4) \fcir f:0.5 r:0.08 \move(-1 -4) \fcir f:0.5 r:0.08

\move(1.3 -4) \fcir f:0.5 r:0.08 \move(1 -4) \fcir f:0.5 r:0.08

\move(-2 -3.9) \htext{$a$} \move(-0.58 -3.9) \htext{$c$} \move(0.45
-3.9) \htext{$b$} \move(2 -3.9) \htext{$d$}

\move(-1.6 -4.3) \htext{$x$} \move(-0.74 -4.3) \htext{$z$}
\move(0.65 -4.3) \htext{$y$} \move(1.5 -4.3) \htext{$w$}

 }
\end{center}
\vspace{0.2cm} \caption{The geodesics $\eta', \eta'', \xi'$, $\xi''$
and the points in $Z$}
\end{figure}

Now let us assume that $\delta$, and hence $\|\gamma\|_{{\Bbb P}^{1}
- (P_{n} \cup X_{L}^{n})}$ and $\|\eta\|_{{\Bbb P}^{1} - (P_{n} \cup
X_{L}^{n})}$ are small enough so that Theorem A.2 and Theorem A.3
can be applied in the following discussion.

From  Lemma~\ref{tech-1}(by taking $M = K + L(\epsilon)$), we have
$$
(P_{n} \cup X_{L}^{n}) \subset G_{n}^{-l}(P_{n} \cup X_{S}^{n})
$$
and
$$
| G_{n}^{-l}(P_{n} \cup X_{S}^{n}) - (P_{n} \cup X_{L}^{n})| \le
C,
$$
where $C$ only depends on $L, N, S$, and the degree of $F$.  By
Theorem A.3, there exists a simple closed geodesic $\eta'$ in
 ${\Bbb P}^{1} - G_{n}^{-l}(P_{n} \cup X_{S}^{n})$, which is homotopic to
 $\eta$ in ${\Bbb P}^{1} - (P_{n} \cup X_{L}^{n})$ such that
\begin{equation}\label{eta'}
\|\eta'\|_{{\Bbb P}^{1} - G_{n}^{-l}(P_{n} \cup X_{S}^{n})} \le C'
\|\eta\|_{{\Bbb P}^{1} - (P_{n} \cup X_{L}^{n})}
\end{equation} where $C'$ depends only on $L, N, S$, and the degree
of $F$. Since $\eta'$  is homotopic to
 $\eta$ in ${\Bbb P}^{1} - (P_{n} \cup X_{L}^{n})$, we have
 $$
 I \subset  \partial \Delta \cap D(\eta')  \hbox{ and } J \subset
 \partial \Delta - D(\eta').
 $$
Let $$ \partial \Delta \cap \eta' = \{a, b \}.$$

By theorem A.2, $G_{n}^{l}(\eta')$ covers a short simple closed
geodesic $\eta''$ in ${\Bbb P}^{1} - (P_{n} \cup X_{S}^{n})$, and
therefore,
\begin{equation}\label{eta''}
\|\eta''\|_{{\Bbb P}^{1} - (P_{n} \cup X_{S}^{n})} \le
\|\eta'\|_{{\Bbb P}^{1} - G^{-l}(P_{n} \cup X_{S}^{n})}.
\end{equation}
Most importantly, $\eta''$ separates $G_{n}^{l}(I)$ and
$G_{n}^{l}(J)$, that is, one of them is contained in $D(\eta'')$ and
the other one is contained in the outside of $D(\eta'')$.  Let
$$
\partial \Delta \cap \eta'' = \{G_{n}^{l}(a), G_{n}^{l}(b)\} = \{c, d\}.
$$

From the inequalities (i), (ii), and (iii), it follows that one can
label the four intersection points $a, b, c$, and $d$ appropriately,
such that they are distributed   in the order of $a, c, b$, and $d$,
 and moreover, each of the three
segments $[a, c]$, $[c, b]$, and $[b, d]$ contains at least two
points in $Z$. See Figure 5 for an illustration.

From above it follows that both $\eta'$ and $\eta''$  are
non-peripheral curves in ${\Bbb P}^{1} - (P_{n} \cup Z)$. Let $\xi'$
be the simple closed geodesic in ${\Bbb P}^{1} - (P_{n} \cup Z)$
which is homotopic to $\eta'$ in ${\Bbb P}^{1} - (P_{n} \cup Z)$.
Since $${\Bbb P}^{1} - (P_{n} \cup Z) \supset {\Bbb P}^{1} -
G^{-l}(P_{n} \cup X_{S}^{n})$$ by the definition of $Z$, it follows
that
\begin{equation}\label{xi'}
\|\xi'\|_{{\Bbb P}^{1} - (P_{n} \cup Z)} \le \|\eta'\|_{{\Bbb P}^{1}
- G^{-l}(P_{n} \cup X_{S}^{n})}.
\end{equation}
Suppose $\xi'$ intersects with $\partial \Delta$ at the two points
$x$ and $y$. Similarly, let $\xi''$ be the simple closed geodesic in
${\Bbb P}^{1} - (P_{n} \cup Z)$ which is homotopic to $\eta''$ in
${\Bbb P}^{1} - (P_{n} \cup Z)$. Since $${\Bbb P}^{1} - (P_{n} \cup
Z) \supset {\Bbb P}^{1} - (P_{n} \cup X_{S}^{n})$$ by the definition
of $Z$, it follows that
\begin{equation}\label{xi''}
\|\xi''\|_{{\Bbb P}^{1} - (P_{n} \cup Z)} \le  \|\eta''\|_{{\Bbb
P}^{1} - (P_{n} \cup X_{S}^{n})}.
\end{equation}
Suppose $\xi''$ intersects with $\partial \Delta$ at the two points
$z$ and $w$. One can label $x, z, y$, and $w$ so that they  are in
the same order as $a, c, b$ and $d$. This implies that $\xi'$ and
$\xi''$ separates the points in $Z$ in the same way as $\eta'$ and
$\eta''$. It follows that $$\xi' \cap \xi'' \ne \emptyset.$$ See
Figure 5 for an illustration.  But (\ref{eta'}), (\ref{eta''}),
(\ref{xi'}) and (\ref{xi''}) imply that both $\xi'$ and $\xi''$ can
be short to any extent provided that $\delta$ is small. This is a
contradiction with Theorem A.1.

\end{proof}

\begin{lemma}\label{fil}
There exist $L\ge N+2$ and a $\delta>0$ such that for every $n$
large enough,  any simple closed geodesic $\gamma$ in ${\Bbb P}^{1}
- (P_{n} \cap X_{L}^{n})$ has length greater  than $\delta$.
\end{lemma}

Our argument  is an adapted version of the one used in  $\S8$ of
\cite{DH}. In fact,  all the short simple closed geodesics which do
not intersect the unit circle, consist of a $G_{n}-$stable family
$\Gamma$. Since $|P_{G_{n}}'-\partial \Delta| = |P_{F}' -\partial
\Delta|$ does not depend on $n$, there is an $m
> 0$ independent of $n$ such that $\|A^{m}\| < 1/2$ where $A$ is the
associated linear transformation matrix. Then by using a similar
argument with the one in \cite{DH}, it follows that the simple
closed geodesics in $\Gamma$ can not be too short. In the following
proof, we will present the details at the place where the situation
here is different from that in $\S8$ of \cite{DH}, and only give a
sketch if they are the same.

\begin{proof}
By Lemma~\ref{ml}, we can take $N+2 \le L< \infty$ and $\epsilon >0$
such that for all $n$ large enough, any simple closed geodesic in
${\Bbb P}^{1} - (P_{n} \cup X_{L}^{n})$, which intersects the unit
circle, has length greater than $\epsilon$.

Let $0< \delta < \epsilon$ and  consider the family
$\Gamma_{\delta}$ of all the simple closed geodesics in ${\Bbb
P}^{1} - (P_{n}  \cup X_{L}^{n})$ which has length $\le \delta$. By
using the same argument as in the proof of Proposition 8.1 in
\cite{DH}, it follows  that if $\Gamma_{\delta} \ne \emptyset$ for
$\delta$ small enough,  then there is a  $G_{n}-stable$ family in
${\Bbb P}^{1} - P_{G_{n}} \cup \{0, \infty\}$, say $\Gamma$, which
consists of short simple closed geodesics in ${\Bbb P}^{1} - (P_{n}
\cup X_{L}^{n})$ and which satisfies certain gap property. Roughly
speaking, the gap property means that there is a uniform $\tau > 0$,
such that every  simple closed geodesic in ${\Bbb P}^{1} - (P_{n}
\cup X_{L}^{n})$ either belongs to $\Gamma$ or has hyperbolic length
greater than $\tau$. We refer the reader to $\S8$ of \cite{DH} for
more details about this property.

Now  let $A$ be the associated linear transformation matrix of
$\Gamma$. According to Thurston's characterization theorem (see
$\S5$),  we have $\|A\| < 1$. Since the curves in $\Gamma$ do not
intersect the unit circle, and $|P_{n}| = 2N$, it follows that the
number of the curves in $\Gamma$ has an upper bound which is
independent of $n$. This implies that the number of all the possible
linear transformation matrixes also has an upper bound independent
of $n$. Therefore, there is an $0< m < \infty$, which is independent
of $n$, such that $\|A^{m}\| \le 1/2$ where $A$ is the Thurston
linear transformation matrix for any such $G_{n}-stable$ family
$\Gamma$. Note that $m$ does not depend on $L$. In the following  we
may assume that $L> m$ by increasing $L$ if necessary.

Now assume that $\gamma \in \Gamma$ is a short simple closed
geodesic in ${\Bbb P}^{1} - (P_{n} \cup X_{L}^{n})$. Recall that for
a hyperbolic Riemann surface $X$, we  use $\|\gamma\|_{X}$ to denote
the hyperbolic length of the simple closed geodesic which is
homotopic to $\gamma$ in $X$. Let us consider the set of all the
simple closed geodesics in ${\Bbb P}^{1} - (P_{n} \cup
X_{L+2m}^{n})$ which are homotopic to $\gamma$ in ${\Bbb P}^{1} -
(P_{n} \cup X_{L}^{n})$ and with length less than $\log(\sqrt 2 +
1)$. The number of the curves in this set is not more than
$$
| X_{L+2m}^{n}- X_{L}^{n}|
$$
which is independent of $n$.  By Lemma~\ref{ml}, it follows that
among all these curves, only the one, which does not intersect the
unit circle (therefore, is homotopic to $\gamma$ in ${\Bbb P}^{1} -
P_{n} \cup X_{L+2m}^{n}$), can be short. The length of all the other
curves has a positive lower bound which is independent of $n$.  By
Theorem A.3 we get
\begin{equation}\label{com-1}
\|\gamma\|_{{\Bbb P}^{1} -(P_{n} \cup X_{L}^{n})}^{-1} \le
\|\gamma\|_{{\Bbb P}^{1} - (P_{n} \cup X_{L+2m}^{n})}^{-1} + C_{1},
\end{equation}
where $C_{1}$ is some constant independent of $n$.

Let
$$
Y = \partial \Delta  \cap G_{n}^{-m}(X_{L}^{n}).
$$
It follows that
$${\Bbb P}^{1} - (P_{n} \cup X_{L+2m}^{n}) \subset {\Bbb P}^{1} -
(P_{n} \cup Y).$$ So
\begin{equation}\label{com-2}
\|\gamma\|_{{\Bbb P}^{1} - (P_{n} \cup X_{L+2m}^{n})}^{-1} <
\|\gamma\|_{{\Bbb P}^{1} - (P_{n} \cup Y)}^{-1}.
\end{equation}
From (\ref{com-1}) and (\ref{com-2}), we have
\begin{equation}\label{com-3}
\|\gamma\|_{{\Bbb P}^{1} -(P_{n} \cup X_{L}^{n})}^{-1} \le
\|\gamma\|_{{\Bbb P}^{1} - (P_{n} \cup Y)}^{-1} + C_{1}.
\end{equation}

Note that
$$
{\Bbb P}^{1} - G_{n}^{-m}(P_{n} \cup X_{L}^{n}) \subset {\Bbb P}^{1}
- (P_{n} \cup Y).
$$ For each $\gamma_{i} \in \Gamma$, let
$$
\gamma_{i,\: j,\:\alpha} \subset {\Bbb P}^{1} - G_{n}^{-m}(P_{n}
\cup X_{L}^{n}), \:\:\alpha  \in \Lambda_{i,j},$$ be all the
components of $G_{n}^{-m}(\gamma_{j})$,  which are homotopic to
$\gamma_{i}$ in ${\Bbb P}^{1} - (P_{n} \cup Y)$, and whose length is
less then $\log(\sqrt 2 +1)$. Here $\Lambda_{i,j}$ is a finite set.

By the gap property of $\Gamma$, there is a uniform positive lower
bound $B > 0$ independent of $n$ such that every simple closed
geodesic in ${\Bbb P}^{1} - G_{n}^{-m}(P_{n} \cup X_{L}^{n})$, which
is homotopic to some $\gamma_{i}$ in ${\Bbb P}^{1} - (P_{n} \cup
Y)$, but does not belong to $\{\gamma_{i,\: j,\:\alpha} \}$, must
have length greater than $B$(This is the place where the gap
property is required). This, together with Theorem A.3, and the fact
that  $$|G_{n}^{-m}(P_{n} \cup X_{L}^{n}) - (P_{n}\cup Y)|$$ depends
only on $L, N, m$, and the degree of $F$, implies that there is a
$0< C_{2}< \infty$ independent of $n$ such that
\begin{equation}\label{gap}
\|\gamma_{i}\|^{-1}_{{\Bbb P}^{1} - (P_{n} \cup Y)} \le \sum_{j}
\sum_{\alpha} \|\gamma_{i,j,\alpha}\|^{-1}_{{\Bbb P}^{1} -
G_{n}^{-m}(P_{n} \cup X_{L}^{n})} + C_{2}.
\end{equation}

Since $L > m$, it follows that
$$
G^{m}: {\Bbb P}^{1} - G_{n}^{-m}(P_{n} \cup X_{L}^{n}) \to {\Bbb
P}^{1} - (P_{n} \cup X_{L}^{n})$$
is a holomorphic covering map.
This, together with the inequality $\|A\|^{m} \le \frac{1}{2}$
implies
\begin{equation}\label{th}
\sum_{i} \sum_{j} \sum_{\alpha} \|\gamma_{i,j,\alpha}\|^{-1}_{{\Bbb
P}^{1} - G_{n}^{-m}(P_{n} \cup X_{L}^{n})} \le \frac{1}{2}\sum_{i}
\|\gamma_{i}\|^{-1}_{{\Bbb P}^{1} - (P_{n} \cup X_{L}^{n})}.
\end{equation}
From (\ref{com-3}), (\ref{gap}), and (\ref{th}), we have
$$
\sum_{i} \|\gamma_{i}\|^{-1}_{{\Bbb P}^{1} - (P_{n} \cup X_{L}^{n}
)} \le \frac{1}{2}\sum_{i} \|\gamma_{i}\|^{-1}_{{\Bbb P}^{1} -
(P_{n} \cup X_{L}^{n} )} +C,
$$ and hence
\begin{equation}
\sum_{i} \|\gamma_{i}\|^{-1}_{{\Bbb P}^{1} - (P_{n} \cup X_{L}^{n}
)} \le 2C.
\end{equation}
where $0< C< \infty$ depends only on $L, m, N$ and the degree of
$F$.  Lemma~\ref{fil} follows.
\end{proof}

Let $\mathcal{Z}_{n}$ and $\mathcal{P}_{n}$ denote the set of the
zeros and poles of $G_{n}$, respectively. The following two lemmas
imply the bounded geometry of $P_{G_{n}}$ and the compactness of the
sequence $\{G_{n}\}$.
\begin{lemma}\label{Z-P}
There is a $\delta>0$ independent of $n$ such that for any two
points in $P_{n} \cup \mathcal{Z}_{n} \cup \mathcal{P}_{n}$, say $x$
and $y$, we have $d_{S^{2}}(x,y) > \delta$.
\end{lemma}

\begin{proof}
Let $L\ge N+2$ be the number in Lemma~\ref{fil}.  It is clear that
$$
P_{n} \cup \mathcal{Z}_{n} \cup \mathcal{P}_{n} \subset
G_{n}^{-1}(P_{n}  \cup X_{L}^{n}).
$$Consider the space
$$
Y_{n} = {\Bbb P}^{1} - G_{n}^{-1}(P_{n}  \cup X_{L}^{n}).
$$
Assume that Lemma~\ref{Z-P} were not true. Then we would have a
sequence of integers, say $\{n_{k}\}$, such that $n_{k} \to \infty$
as $k \to \infty$, and a sequence of short simple closed geodesics,
say  $\gamma_{n_{k}} \subset Y_{n_{k}}$, such that
 $\|\gamma_{n_{k}}\|_{Y_{n_{k}}} \to 0$. Then every
$G_{n_{k}}(\gamma_{n_{k}})$ covers a short simple closed geodesic
$\xi_{n_{k}} \subset {\Bbb P}^{1} - (P_{n_{k}}  \cup X_{L}^{n_{k}})$
whose length goes to $0$ as $k \to \infty$. This is a contradiction
with Lemma~\ref{fil}.
\end{proof}

\begin{lemma}\label{X-D}
There is a $\delta>0$ independent of $n$ such that for any point in
$P_{n} \cup \mathcal{Z}_{n} \cup \mathcal{P}_{n}$, say $x$, we have
$d_{S^{2}}(x,\partial \Delta) > \delta$.
\end{lemma}

\begin{proof}
Let $x^{*}$ be the symmetric image of $x$ about the unit circle. It
follows that $x^{*} \in P_{n} \cup \mathcal{Z}_{n} \cup
\mathcal{P}_{n}$. By Lemma~\ref{Z-P},  $d_{S^{2}}(x, \partial
\Delta) =  d_{S^{2}}(x, x^{*})/2$, and therefore has a positive
lower bound independent of $n$.
\end{proof}

Recall that $\mathcal{R}$$_{2d-1}$ denotes the space of all the
rational maps of degree $2d-1$. From Lemma~\ref{Z-P} and
Lemma~\ref{X-D}, we get
\begin{lemma}\label{com-F}
The sequence $\{G_{n}\}$ is contained in some compact subset of
$\mathcal{R}$$_{2d-1}$.
\end{lemma}

\subsubsection{Bounded geometry of $P_{G_{n}}$ on $\partial \Delta$}
 By passing to a convergent subsequence, we may now assume that $G_{n} \to G$.
It follows that  $G|\partial \Delta $ is an analytic critical circle
homeomorphism with rotation number $\theta$. It was proved by Herman
and Swiatek that such a critical circle homeomorphism is
quasi-symmetrically conjugate to the rigid rotation $R_{\theta}$ if
$\theta$ is of $\emph{bounded type}$ (see \cite{Pe2} for a detailed
proof). Let $h: \partial \Delta \to \partial \Delta$ be the
quasi-symmetric homeomorphism such that
$$
h(1) = 1 \hbox{ and } G|\partial \Delta = h \circ R_{\theta} \circ
h^{-1}.
$$

Since $G_{n}$ and $F_{n}$ are $\emph{combinatorially equivalent}$ to
each other rel $P_{F_{n}}'$, there exist  a pair of homeomorphisms
$\phi_{n}, \psi_{n}: S^{2} \to {\Bbb P}^{1}$ such that
$$
     \begin{CD}
           (S^{2},P_{F_{n}}')           @  >\psi_{n}   >  >({\Bbb P}^{1}, P_{G_{n}}')         \\
           @V F_{n} VV                         @VV G_{n} V\\
           (S^{2},P_{F_{n}}')           @  >\phi_{n}   >  >      ({\Bbb P}^{1}, P_{G_{n}}')
     \end{CD}
$$
and $\phi_{n}$ is isotopic to $\psi_{n}$ rel $P_{F_{n}}'$.

\begin{lemma}\label{boundary}
$\psi_{n}|\partial \Delta \to h, \phi_{n}|\partial \Delta \to h$
uniformly as $n \to \infty$.
\end{lemma}
\begin{proof} We need only to prove that $\psi_{n} \to h$ uniformly
as $n \to \infty$. The other one can be proved by the same argument.
Let $N$ be an integer such that the length of each interval
component of $\partial \Delta - \{G^{k}(1)\}_{0 \le k \le N}$ is
less than one-sixth of the whole circle. Since $G_{n} \to G$
uniformly as $n \to \infty$, it follows that when $n$ is large
enough, the length of each component of  $\partial \Delta -
\{G_{n}^{k}(1)\}_{0 \le k \le N}$ is less then one-fifth of the
whole circle. Let $$\delta_{1} = \min \{|I|/6 \:\big{|}\:I
\hbox{\:is a component of \:}\partial \Delta - \{e^{2 k \pi i
\theta}\}_{0 \le k \le N}\}.$$ It follows that for every $n$ large
enough, the image of an arc segment with length less than
$6\delta_{1}$ will be mapped by $\psi_{n}$  to some arc segment less
than one half of the whole circle. In fact, if $\psi_{n}(I)$
contains a half of the circle, then it contains at least two
components of $\partial \Delta - \{G_{n}^{k}(1)\}_{0 \le k \le N}$.
This implies that $I$ contains at least two components of $\partial
\Delta - \{e^{2 k \pi i \theta}\}_{0 \le k \le N}$. But this is a
contradiction with the definition of $\delta_{1}$.

Now for any given $\epsilon > 0$, since $h$ is uniformly continuous,
we have a $\delta_{2}>0$ such that for any $x, x' \in \partial
\Delta$ and $|x - x'| < 4 \delta_{2}$ ,
\begin{equation}\label{con-h}
|h(x)- h(x')| < \epsilon/5.
\end{equation}
Take $\delta = \min \{\delta_{1}, \delta_{2}\}$. For such $\delta$,
there is an integer $M  >0$ such that for any $x$ in the unit
circle, there are two integers $0< k_{1}, k_{2}<M$ such that
\begin{equation}\label{sita}
e^{2\pi ik_{1}\theta} \in (x+\delta, x+2\delta) \hbox{ and } e^{2\pi
ik_{2}\theta} \in (x-2\delta, x-\delta).
\end{equation}
 For such
$\epsilon, \delta$, and $M$, take $N$ large enough such that when $n
>N$,
\begin{itemize}
\item[i.] $|\theta_{n}- \theta| < \delta / 2\pi M $;
\item[ii.] $|G_{n}^{k}(x) - G^{k}(x)| < \epsilon /5$ for all
$1 \le k \le M$ and all $x \in \partial \Delta$;
\end{itemize}
From (\ref{sita}) and Property (i) above, it follows that we have
\begin{equation}\label{sitan}
e^{2\pi ik_{1}\theta_{n}} \in (x, x+3\delta)\quad   \hbox{and} \quad
e^{2\pi i k_{2}\theta_{n}} \in (x-3\delta, x).
\end{equation}
This implies that $e^{2 \pi ik_{1}\theta_{n}}, x$, and $e^{2 \pi
ik_{2}\theta_{n}}$ are contained in an arc segment with length less
then $6 \delta$, which is mapped by $\psi_{n}$  to some arc segment
less than one half of the  circle. It follows that
$$| \psi_{n}(x) - \psi_{n}(e^{2\pi ik_{1} \theta_{n}})| \le
|\psi_{n}(e^{2\pi ik_{1} \theta_{n}}) -\psi_{n}(e^{2\pi i k_{2}
\theta_{n}})|.$$ We thus have the following,
$$
|\psi_{n}(x) -h(x)| \leq | \psi_{n}(x) - \psi_{n}(e^{2\pi ik_{1}
\theta_{n}})| + |\psi_{n}(e^{2\pi ik_{1}\theta_{n}}) -h(e^{2\pi
k_{1}\theta})| $$$$ +|h (e^{2\pi k_{1}\theta})-h(x)|
$$
$$
\leq |\psi_{n}(e^{2\pi ik_{1} \theta_{n}}) -\psi_{n}(e^{2\pi i k_{2}
\theta_{n}})| + |\psi_{n}(e^{2\pi ik_{1}\theta_{n}}) -h(e^{2\pi i
k_{1}\theta})|$$$$ +|h(e^{2\pi k_{1}\theta})-h(x)|
$$
$$
=|G_{n}^{k_{1}}(1)-G_{n}^{k_{2}}(1)| +|G_{n}^{k_{1}}(1)
-G^{k_{1}}(1)| + |h (e^{2\pi i k_{1}\theta}) -h(x)|
$$
$$
\leq |G_{n}^{k_{1}}(1)-G^{k_{1}}(1)| +|G^{k_{1}}(1) -G^{k_{2}}(1)|+
|G^{k_{2}}(1)-G^{k_{2}}_{n}(1)|
$$
$$+|G^{k_{1}}_{n}(1) -G^{k_{1}}(1)|+ |h(e^{2\pi i k_{1}\theta}) -h(x)| \leq \epsilon
$$

Let us explain how the last inequality comes.  The inequalities
$|G_{n}^{k_{1}}(1)-G^{k_{1}}(1)| < \epsilon/5$,
$|G^{k_{2}}(1)-G^{k_{2}}_{n}(1)|< \epsilon/5$, and
$|G^{k_{1}}_{n}(1) -G^{k_{1}}(1)| < \epsilon/5$ come from the
property (ii) above. The inequality $|G^{k_{1}}(1) -G^{k_{2}}(1)| <
\epsilon/5$ comes from (\ref{con-h}) and (\ref{sita}).  The
inequality $|h(e^{2\pi i k_{1}\theta}) -h(x)|< \epsilon/5$ comes
from (\ref{con-h}) and (\ref{sitan}).

\end{proof}

\subsection{The Candidate Blaschke Product}
In this section, we will show that $G$ is the desired Blaschke
product by showing that there are homeomorphisms $\phi, \psi: S^{2}
\to S^{2}$ which fix $0$, $1$, and $\infty$, such that $G = \phi
\circ F \circ \psi^{-1}$ and $\phi, \psi$ are isotopic to each other
rel $P_{F}'$.  Recall that for every $n$ large enough, there is a
pair of homeomorphisms $\phi_{n}$ and $\psi_{n}$ such that $G_{n} =
\phi_{n} \circ F_{n} \circ \psi_{n}^{-1}$ and $\phi_{n}$ and
$\psi_{n}$ are isotopic to each other rel $P_{F_{n}}'$. The aim of
this section is to show that the homotopy classes of $\phi_{n}$ and
$\psi_{n}$ converge to the same one as $n \to \infty$.

First we will show that for every $n$ large enough, by deforming
$\phi_{n}$ and $\psi_{n}$ in their isotopy class, we can make
$\phi_{n}$ and $\psi_{n}$ satisfy some local properties around each
point in $\Omega_{F_{n}} \cup P_{F_{n}}'$ (Lemma~\ref{adjust}).
Secondly we will  prove that for every $\rho > 0$, provided that $n$
is large enough, the map  $\phi_{n}$ and $\psi_{n}$ can be perturbed
within their $\rho-$neighborhood into a pair of homeomorphisms
$\widehat{\phi}_{n}$ and $\widehat{\psi}_{n}$ such that  $G =
\widehat{\phi}_{n} \circ F \circ \widehat{\psi}_{n}^{-1}$
(Lemma~\ref{pair}). Finally we will prove that when $\rho$ is small,
the maps $\widehat{\phi}_{n}$ and $\widehat{\psi}_{n}$ are isotopic
to each other rel $P_{F}'$(Lemma~\ref{iso}).

\subsubsection{Deforming $\phi_{n}$ and $\psi_{n}$ in their isotopy class}

 Let $r> 0$ be a number  such that
 $$
 d_{S^{2}}(x, y) > r
 $$
for any two distinct points $x$ and $y$ in $\Omega_{F} \cup (P_{F}'
-\partial \Delta)$.  Since $F_{n} \to F$ uniformly, it follows that
for any $x \in \Omega_{F} \cup (P_{F}' -
\partial \Delta)$, and every large $n$,  $B_{r/3}(x)$ contains
exactly one point in  $\Omega_{F_{n}} \cup (P_{F_{n}}' - \partial
\Delta) $. Let us denote this point by $\tau_{n}(x)$. It is easy to
see that $\tau_{n}(x) \to x$ as $n \to \infty$. By passing to a
convergent subsequence, and by Lemma~\ref{com-F}, we may assume that
for any $x \in \Omega_{F} \cup (P_{F}' - \partial \Delta)$,
$\psi_{n}(\tau_{n}(x))$ and $\phi_{n}(F_{n}(\tau_{n}(x)))$ converge
as $n \to \infty$.

\begin{lemma}\label{adjust}
For any  $r, \delta > 0$ there exist $0< r_{0} < r$ and $0<
\delta_{0} < \delta$, such that for any $0< r' < r_{0}$ and  $0<
\delta' < \delta_{0}$, there exist $0 < r_{1} < r_{2}< r_{3}< r' $,
and $0< \delta_{1} < \delta_{2}< \delta_{3}< \delta'$  such that for
every $n$ large enough, there exist homeomorphisms $\phi_{n}$ and
$\psi_{n}: S^{2} \to S^{2}$ such that \begin{itemize} \item[1.]
$\psi_{n}$ and $\phi_{n}$ are isotopic to each other rel
$P_{F_{n}}'$, \item[2.]  $G_{n} = \phi_{n} \circ F_{n} \circ
\psi_{n}^{-1}$, \item[3.]  for any $x \in \Omega_{F} \cup (P_{F}' -
\partial \Delta)$, by taking a convergent subsequence, $\psi_{n}(\tau_{n}(x))$ and
$\phi_{n}(F_{n}(\tau_{n}(x)))$ converge as $n \to \infty$,
\item[4.]
for every $x \in P_{F_{n}}' \cup \Omega_{F_{n}}$,
$B_{\delta_{1}}(\phi_{n}(x)) \subset \phi_{n}(B_{r_{1}}(x)) \subset
\phi_{n}(B_{r_{2}}(x)) \subset B_{\delta_{2}}(\phi_{n}(x)) \subset
\phi_{n}(B_{r_{3}}(x)) \subset B_{\delta_{3}}(\phi_{n}(x))  \subset
B_{\delta_{0}}(\phi_{n}(x)) \subset \phi_{n}(B_{r_{0}}(x))$ and this
inclusion relation also holds if we replace $\phi_{n}$ by
$\psi_{n}$.
\end{itemize}
\end{lemma}

\begin{proof}From the previous sections, it follows that for every $n$
large enough, there exist a pair of homeomorphisms $\phi_{n}$ and
$\psi_{n}$ such that (1), (2) and (3) hold. For any $r, \delta> 0$,
since $P_{G_{n}}$ has uniform bounded geometry(see Lemma~\ref{Z-P},
\ref{X-D}, and \ref{boundary}), we can take $r_{0}' \ll r$ and
$\delta_{0}' \ll \delta$ such that for every $n$ large enough,
$\phi_{n}$ can be deformed in its isotopic class  so that it
satisfies
$$
B_{\delta_{0}'}(\phi_{n}(x)) \subset
\phi_{n}(B_{r_{0}'}(x))
$$ for every $x \in P_{F_{n}}' \cup
\Omega_{F_{n}}$. Then we lift $\phi_{n}$ by the equation $$\phi_{n}
\circ F_{n} = G_{n} \circ \psi_{n}$$ and get $\psi_{n}$. Since
$F_{n} \to F$ and $G_{n} \to G$ uniformly in the spherical metric,
it follows that there exist $r_{0}''$ and $\delta_{0}''$, which can
be taken arbitrarily small provided that  $r_{0}'$ and $\delta_{0}'$
are small, such that $$B_{\delta_{0}''}(\psi_{n}(x)) \subset
\psi_{n}(B_{r_{0}''}(x))$$ for every $x \in P_{F_{n}}' \cup
\Omega_{F_{n}}$. Now take $r_{0} = \max\{r_{0}', r_{0}''\}$ and
$\delta_{0} = \min\{\delta_{0}', \delta_{0}''\}$. By taking $r_{0}',
\delta_{0}'$ small enough, we can assure $r_{0} < r$ and $\delta_{0}
< \delta$. In particular,
$$
B_{\delta_{0}}(\phi_{n}(x)) \subset \phi_{n}(B_{r_{0}}(x)), \hbox{
and } B_{\delta_{0}}(\psi_{n}(x)) \subset \psi_{n}(B_{r_{0}}(x))
$$ for
every $x \in P_{F_{n}}' \cup \Omega_{F_{n}}$. Now let $r' < r_{0}$
and $\delta' < \delta_{0}$ be given. We may use the same process to
get $r_{3}, \delta_{3}$ as follows.

Deform $\phi_{n}$ in a smaller disk around each point $x \in
P_{F_{n}}' \cup \Omega_{F_{n}}$ so that
$$
\phi_{n}(B_{r_{3}'}(x)) \subset B_{\delta_{3}'}(\phi_{n}(x))\subset
B_{\delta_{0}}(\phi_{n}(x))
$$
for some  $r_{3}'\ll r', \delta_{3}'\ll \delta'$, and then get
$\psi_{n}$ by lifting $\phi_{n}$. As in the last step, By choosing
$r_{3}', \delta_{3}'$ small, we can  get  $0< r_{3}< r'$ and $0<
\delta_{3}< \delta'$, such that
$$
\phi_{n}(B_{r_{3}}(x)) \subset B_{\delta_{3}}(\phi_{n}(x))\subset
B_{\delta_{0}}(\phi_{n}(x))
$$ and
$$
\psi_{n}(B_{r_{3}}(x)) \subset B_{\delta_{3}}(\psi_{n}(x)) \subset
B_{\delta_{0}}(\psi_{n}(x)) $$ for all $x \in P_{F_{n}}' \cup
\Omega_{F_{n}}$. Since we deform $\phi_{n}$ only  in a smaller disk,
this step will not affect the relations obtained in the last step.
We may repeat this procedure and get $r_{1}, r_{2}, \delta_{1},
\delta_{2}$ so that the corresponding relations are also satisfied.
The proof of the lemma is completed.
\end{proof}

\subsubsection{Perturbing $\phi_{n}$ and $\psi_{n}$}
\begin{lemma}\label{pair}
Let $\rho > 0$ be an arbitrary number. Then there exists an $N
> 0$ such that for every $n > N$, there exist homeomorphisms
$\widehat{\phi}_{n}$, $\phi_{n}$,  $\widehat{\psi}_{n}$, and
$\psi_{n}$ of the sphere such that
\begin{itemize}
\item[1.]  $G_{n} = \phi_{n} \circ F_{n} \circ \psi_{n}^{-1}$,
and $G = \widehat{\phi}_{n} \circ F \circ
\widehat{\psi}_{n}^{-1}$,
\item[2.] $\phi_{n}$ and  $\psi_{n}$ are isotopic to each other rel
$P_{F_{n}}'$,
\item[3.] $\max_{z \in S^{2}}d_{S^{2}}(\widehat{\phi}_{n}(z),
\phi_{n}(z))< \rho$ and $\max_{z \in
S^{2}}d_{S^{2}}(\widehat{\psi}_{n}(z), \psi_{n}(z)) < \rho$,
\item[4.] $\widehat{\phi}_{n}(\Omega_{F}) =
\widehat{\psi}_{n}(\Omega_{F}) = \Omega_{G}$, and
$\widehat{\phi}_{n}(P_{F}') = \widehat{\psi}_{n}(P_{F}') = P_{G}'$.
\item[5.] $\widehat{\phi}_{n}|\partial \Delta  =
\widehat{\psi}_{n}|\partial \Delta = h$ where $h: \partial \Delta
\to \partial \Delta$ is the quasi-symmetric homeomorphism in
Lemma~\ref{boundary}.
\end{itemize}

\end{lemma}

\begin{proof}
Let  $$0< r_{1}< r_{2}<r_{3}< r'< r_{0}< r $$ and $$0< \delta_{1} <
\delta_{2}< \delta_{3}< \delta'< \delta_{0}< \delta,
$$ be a group of constants as in Lemma~\ref{adjust} such that  $$d_{S^{2}}(x, y) > \delta$$ for any two distinct
points $x$ and $y$  in $\Omega_{G} \cup (P_{G}' -
\partial \Delta)$.   Let $\phi_{n}$ and
$\psi_{n}$ be homeomorphisms which satisfy the conditions in Lemma
~\ref{adjust} with the constants given above.  We will adjust these
constants appropriately as the proof proceeds.

By taking  $r'$ small, and hence $r_{2}$ small, we may assume that
for each $c \in \Omega_{F}$, there is an open topological  disk
$B_{c}$ containing $c$ such that
\begin{itemize}
\item[1.] $F: B_{c} \to B_{r_{2}}(F(c))$ is a $d_{c}$-to-$1$ branched
covering map, where $d_{c} \ge 2$ is the local degree of $F$ at $c$,
\item[2.]   for $c \in \Omega_{F} -
\partial \Delta$, $(B_{c} - \{c\}) \cap  P_{F}' = \emptyset$,
\item[3.] all $B_{c},  c \in \Omega_{F}$, are disjoint.
\end{itemize}

Since $F_{n} \to F$ uniformly as $n \to \infty$, it follows that for
any $c \in \Omega_{F}$ and every $n$ large enough,
$$
B_{r_{1}}(F_{n}(\tau_{n}(c))) \subset  F_{n}(B_{c}) \subset
B_{r_{3}}(F_{n}(\tau_{n}(c))).
$$
This, together with (4) of Lemma~\ref{adjust}, implies that
$$
B_{\delta_{1}}(\phi_{n}(F_{n}(\tau_{n}(c)))) \subset
\phi_{n}(B_{r_{1}}(F_{n}(\tau_{n}(c))))\subset
\phi_{n}(F_{n}(B_{c})) = G_{n}\circ \psi_{n}(B_{c})
$$
and
$$
G_{n} \circ \psi_{n}(B_{c}) = \phi_{n}(F_{n}(B_{c})) \subset
\phi_{n}(B_{r_{3}}(F_{n}(\tau_{n}(c)))) \subset
B_{\delta_{3}}(\phi_{n}(F_{n}(\tau_{n}(c)))).
$$

From  $$B_{\delta_{1}}(\phi_{n}(F_{n}(\tau_{n}(c)))) \subset
G_{n}\circ \psi_{n}(B_{c})$$ and  $$G_{n} \circ \psi_{n}(B_{c})
\subset B_{\delta_{3}}(\phi_{n}(F_{n}(\tau_{n}(c)))),$$ it follows
that there exist $\nu , \mu> 0$ such that for every $n $ large
enough, and any $c \in \Omega_{F}$,
\begin{equation}\label{app-rad}
B_{\mu}(\psi_{n}(\tau_{n}(c)))  \subset \psi_{n}(B_{c}) \subset
B_{\nu}(\psi_{n}(\tau_{n}(c))).
\end{equation}
 Since as $\delta'
\to 0$, $\delta_{1}, \delta_{3} \to 0$, one can take $\mu$ and $\nu$
such that $\mu, \nu \to 0$ as $\delta' \to 0$. In particular, by
taking $\delta'$ small, we may assume that $\nu \le \delta_{0}/40$.

Set
$$
U = \bigcup_{c \in \Omega_{F}} B_{c}.
$$

From the first inclusion of (\ref{app-rad}) and the fact that
$\tau_{n}(c) \to c$ for any $c \in \Omega_{F}$, it follows that for
 all $n$ large enough,
\begin{equation}\label{dis-later}
B_{\mu/2 }(\psi_{n}(z)) \cap \Omega_{G} = \emptyset
\end{equation}
holds for any $ z \in S^{2} - U$. For such $\mu$, there
is an $0< \epsilon< \mu /10$ such that for any $z \in S^{2} -
\bigcup_{c \in \Omega_{G}}B_{\mu/2}(c)$, $G$ is injective on the
disk $B_{\epsilon}(z)$.

For any $\eta>0$, from the bounded geometry of $P_{G_{n}}'$, and
Lemma~\ref{boundary}, there is an $N$ large enough, such that for
every $n
> N$, there exists a homeomorphism $\widehat{\phi}_{n}:S^{2} \to
S^{2}$ such that
\begin{itemize}
\item[i.] $d(\widehat{\phi}_{n}, \phi_{n}) = \max_{z \in S^{2}}
d_{S^{2}}(\phi_{n}(z), \widehat{\phi}_{n}(z)) \le \eta$,
\item[ii.] $\widehat{\phi}_{n}(P_{F}') = P_{G}'$,
\item[iii.] $\widehat{\phi}_{n}(\Omega_{F}) = \Omega_{G}$,  and
\item[iv.] $\widehat{\phi}_{n}|\partial \Delta = h$,
where $h: \partial \Delta \to \partial \Delta$ is the
quasi-symmetric homeomorphism in Lemma~\ref{boundary}.
\end{itemize}

We now claim that by taking $\eta>0$ small enough, we can make sure
that  for every $n$ large enough, and  any  $x \in S^{2} - U$, there
is a unique point $y \in B_{\epsilon}(\psi_{n}(x))$, such that
\begin{equation}
G(y) = \widehat{\phi}_{n}(F(x))
\end{equation}
where $\widehat{\phi}_{n}$ is the map defined previously so that (i)
- (iv) are satisfied.

Let us prove the claim. Since $F_{n} \to F$ and $G_{n} \to G$
uniformly as $n \to \infty$, when $\eta$ is small and $n$ is large
enough, $\widehat{\phi}_{n}(F(x)) \in G(B_{\epsilon}(\psi_{n}(x))$.
This implies the existence of the point $y$. Since
$B_{\mu/2}(\psi_{n}(x)) \cap \Omega_{G} = \emptyset$ by
(\ref{dis-later}), from the choice of $\epsilon$ above, it follows
that $G$ is injective on $B_{\epsilon}(\psi_{n}(x))$. Therefore,
such $y$ must be unique and does not belong to $\Omega_{G}$. This
proves the claim. We define $\widehat{\psi}_{n}(x) = y$ for $x \in
S^{2} - U$. It follows that $\widehat{\psi}_{n}$ is continuous and
locally injective in $S^{2} - U$, and
\begin{equation}\label{no-cri}
\widehat{\psi}_{n}(S^{2} - U) \cap \Omega_{G} = \emptyset.
\end{equation}
Since $\widehat{\psi}_{n}(x) = y\in B_{\epsilon}(\psi_{n}(x))$, it
follows that $|\widehat{\psi}_{n}(x) - \psi_{n}(x)| < \epsilon$.

We now claim that for all $n$ large enough, and every $c \in
\Omega_{F}$, $\widehat{\psi}_{n}|\partial B_{c}$ is injective and
hence $\widehat{\psi}_{n}(\partial B_{c})$ is a Jordan curve. In
fact,  since for each $c \in \Omega_{F}$, by the definition of
$B_{c}$, $F: B_{c} \to B_{r_{2}}(F(c))$ is a $d_{c}$-to-$1$ branched
covering map, where $d_{c} \ge 2$ is the local degree of $F$  at
$c$, it follows that $F(\partial B_{c}) =
\partial B_{r_{2}}(F(c))$ is a Jordan curve, and hence
$\widehat{\phi}_{n}(F(\partial B_{c}))$ is a Jordan curve. From the
construction of $\widehat{\psi}_{n}$, we have
\begin{equation}\label{cover-lift}
G(\widehat{\psi}_{n}(\partial B_{c})) =
\widehat{\phi}_{n}(F(\partial B_{c})).
\end{equation}
Since $\widehat{\psi}_{n}(\partial B_{c}) \cap \Omega_{G} =
\emptyset$ by (\ref{no-cri}), it follows that
$\widehat{\psi}_{n}(\partial B_{c})$ does not intersect with itself,
and is therefore  a Jordan curve. Note that by the construction, we
have
\begin{itemize}
\item[1.]   $B_{\mu}(\psi_{n}(\tau_{n}(c))) \subset \psi_{n}(B_{c})$ by (\ref{app-rad}),
\item[2.]   $\epsilon < \mu/10$ by the choice of $\epsilon$,
\item[3.]   $\psi_{n}|\partial B_{c}: \partial B_{c} \to \psi_{n}(\partial B_{c})$ is a homeomorphism,
\item[4.]   $|\widehat{\psi}_{n}(z) - \psi_{n}(z)| < \epsilon$ for $z \in
\partial B_{c}$.
\end{itemize}
All the above implies that the topological degree of
$\widehat{\psi}_{n}:
\partial B_{c} \to \widehat{\psi}_{n}(\partial B_{c})$ must be  $1$.
Since $\widehat{\psi}_{n}$ is locally injective in $S^{2} - U$, in
particular, it is locally injective on $\partial B_{c}$. It follows
that $\widehat{\psi}_{n}$ is injective on $\partial B_{c}$. The
claim follows.

For any  $c \in \Omega_{F}$, by (3) of Lemma~\ref{adjust},
$\psi_{n}(\tau_{n}(c))$ converges. Let us denote it by $c'$. From
(\ref{app-rad}), and the fact that $0< \epsilon < \mu/10$, and that
$|\widehat{\psi}_{n}(z) - \psi_{n}(z)| < \epsilon$ for $z \in
\partial B_{c}$,  it follows that for every $n$ large enough,
$$
\widehat{\psi}_{n}(B_{c}) \cap \Omega_{G} = \{\c'\}.
$$ From
(\ref{cover-lift}), $G: \widehat{\psi}_{n}(\partial B_{c}) \to
\widehat{\phi}_{n}(F(\partial B_{c}))$ is a $d_{c}:1$ branched
covering map,  where $d_{c}$ is the local degree of $F$ at $c$,
which is equal to the local degree of $G$ at $c'$.  Let $D_{c}$ be
the component of $S^{2} - \widehat{\phi}_{n}(F(\partial B_{c}))$
which contains $G(c')$. It follows that $G:
\widehat{\psi}_{n}(\partial B_{c}) \to \partial D_{c}$ is a
$d_{c}:1$ branched covering map. This allows us to continuously
extend $\widehat{\psi}_{n}$ to the inside of  every $B_{c}$, such
that $\widehat{\psi}_{n}(c) = c'$ and moreover,  $$G(z) =
\widehat{\phi}_{n} \circ F \circ \widehat{\psi}_{n}^{-1}(z)$$ holds
on the whole sphere. It is also easy to see that
$\widehat{\psi}_{n}: S^{2} \to S^{2}$ is a homeomorphism.

From the construction of $\widehat{\phi}_{n}$ and
$\widehat{\psi}_{n}$, it follows  that $d(\widehat{\phi}_{n},
\phi_{n}) < \eta$ and $d(\widehat{\psi}_{n}, \psi_{n}) < \epsilon +
2\nu < \delta_{0} /15$. This completes the proof since $\eta$ and
$\delta_{0}$ can be taken arbitrarily small.

\end{proof}
\begin{figure}
\bigskip
\begin{center}
\centertexdraw { \drawdim cm \linewd 0.02 \move(0 1)

\move(-3 -1) \lcir r:1.4 \move(3 -1) \lcir r:1.4

\move(-3 -1) \lvec(-2.3 0.2) \lvec(2.3 0.2) \lvec(3 -1) \move(-3 -1)
\lvec(-2.3 -2.2) \lvec(2.3 -2.2) \lvec(3 -1)

\move(-2.4 -1) \lvec(-2.2 0.1) \lvec(2.2 0.1) \lvec(2.4 -1)
\move(-2.4 -1) \lvec(-2.2 -2.1) \lvec(2.2 -2.1) \lvec(2.4 -1)
\move(-3.5 -1.1) \htext{$a_{n}'$} \move(-2.2 -1.1) \htext{$a'$}
\move(3.2 -1.1) \htext{$b_{n}'$} \move(2 -1.1) \htext{$b'$}
\move(-2.4 -1) \fcir f:0.5 r:0.07 \move(-3 -1) \fcir f:0.5 r:0.07
\move(3 -1) \fcir f:0.5 r:0.07 \move(2.4 -1) \fcir f:0.5 r:0.07

 }
\end{center}
\vspace{0.2cm} \caption{The resulted curves after the first step}
\end{figure}
\begin{lemma}\label{iso}
There is a $\rho_{0} > 0$ small such that for all $0< \rho <
\rho_{0}$, the maps $\widehat{\phi}_{n}$ and $\widehat{\psi}_{n}$
obtained in Lemma~\ref{pair} are isotopic to each other rel
$P_{F}'$.
\end{lemma}
\begin{proof}
Since $\widehat{\phi}_{n}|\partial \Delta  =
\widehat{\psi}_{n}|\partial \Delta  = h$, it is sufficient to prove
that the restrictions of $\widehat{\phi}_{n}$ and
$\widehat{\psi}_{n}$ on the unit disk are isotopic to each other rel
$P_{F}' \cap \overline{\Delta}$. This is then equivalent to show
that for any curve segment $\gamma \subset \overline{\Delta}$ which
connects two distinct points $a$ and $b$ in $P_{F}'$, the image
curve segments $\widehat{\phi}_{n}(\gamma)$ and
$\widehat{\psi}_{n}(\gamma)$ are homotopic to each other rel $\{a',
b'\}$ where $a' = \widehat{\phi}_{n}(a) = \widehat{\psi}_{n}(a)$ and
$b' = \widehat{\phi}_{n}(b) = \widehat{\psi}_{n}(b)$.

It is sufficient to consider two cases. In the first case, neither
$a$ nor $b$ is on the unit circle. In the second case, $a$ is on the
unit circle, but $b$ is not on the unit circle.  The proofs are
quite direct. Let us explain the idea only and the reader shall have
no difficulty to supply the details.

Suppose that we are in the first case. Let $\delta > 0$ be such that
for any $z \in P_{F}' - \partial \Delta$,
$$
B_{3\delta}(z) \cap P_{G}' = \{z\}.
$$
Let $a_{n}$ and $b_{n}$ be the two points in $P_{F_{n}}'$ which are
correspond to $a$ and $b$, respectively. Let $\gamma_{n}$ be a curve
segment which connects $a_{n}$ and $b_{n}$ and is close to $\gamma$.
Let $a_{n}' = \phi_{n}(a_{n}) = \psi_{n}(a_{n})$ and $b_{n}' =
\phi_{n}(b_{n}) = \psi_{n}(b_{n})$. By deforming $\gamma$ in its
homotopy class rel $\{a, b\}$ and changing $\gamma_{n}$
correspondingly, we may assume that each of $\phi_{n}(\gamma_{n})$
and $\psi_{n}(\gamma_{n})$ is the union of three curve segments
described as follows. The first piece is a straight segment which
connects $a_{n}'$ and $\partial B_{\delta}(a_{n}')$. The second
piece is some curve segment which does not intersect the
$\delta-$neighborhood of $P_{G}'$ and connects $\partial
B_{\delta}(a_{n}')$ and $\partial B_{\delta}(b_{n}')$. The third
piece is a straight segment connects $\partial B_{\delta}(b_{n}')$
and $b_{n}'$.
\begin{figure}
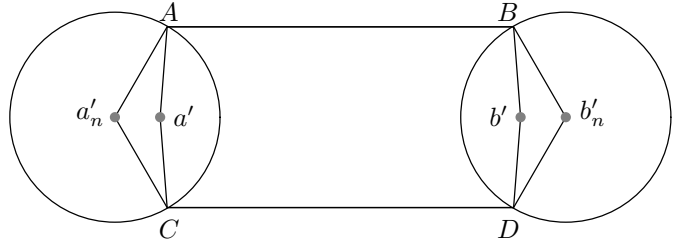

\bigskip
\begin{center}
\centertexdraw { \drawdim cm \linewd 0.02 \move(0 1)

\move(-3 -1) \lcir r:1.4 \move(3 -1) \lcir r:1.4

\move(-3 -1) \lvec(-2.3 0.2) \lvec(2.3 0.2) \lvec(3 -1) \move(-3 -1)
\lvec(-2.3 -2.2) \lvec(2.3 -2.2) \lvec(3 -1)

\move(-2.4 -1) \lvec(-2.3 0.2) \lvec(2.3 0.2) \lvec(2.4 -1)
\move(-2.4 -1) \lvec(-2.3 -2.2) \lvec(2.3 -2.2) \lvec(2.4 -1)
\move(-3.5 -1.1) \htext{$a_{n}'$} \move(-2.2 -1.1) \htext{$a'$}
\move(3.2 -1.1) \htext{$b_{n}'$} \move(2 -1.1) \htext{$b'$}
\move(-2.4 -1) \fcir f:0.5 r:0.07 \move(-3 -1) \fcir f:0.5 r:0.07
\move(3 -1) \fcir f:0.5 r:0.07 \move(2.4 -1) \fcir f:0.5 r:0.07

\move(-2.4 0.3) \htext{$A$} \move(2.1 0.3) \htext{$B$}

\move(-2.4 -2.6) \htext{$C$} \move(2.1 -2.6) \htext{$D$}
 }
\end{center}
\vspace{0.2cm} \caption{The resulted curves after the first step}
\end{figure}

Now in Lemma~\ref{pair}, by taking $0< \rho \ll \delta$ small and
thus $n$ large, we may assume that each of
$\widehat{\phi}_{n}(\gamma)$ and $\widehat{\psi}_{n}(\gamma)$ is the
union of three curve segments described as follows. The first piece
is a  curve segment which connects $a'$ and $\partial
B_{\delta}(a_{n}')$  and is contained in $B_{2 \delta}(a_{n}')$. The
second piece is some curve segment which  connects $\partial
B_{\delta}(a_{n}')$ and $\partial B_{\delta}(b_{n}')$, and is
contained in the $2 \rho$-neighborhood of the corresponding second
piece described as above. The third piece is a curve segment which
connects $\partial B_{\delta}(b_{n}')$ and $b'$ and is contained in
$B_{2\delta}(b_{n}')$.

For an illustration of these curves, see Figure 6.

Now the homotopy is realized by two steps. In the first step, we can
deform $\widehat{\phi}_{n}(\gamma)$ in $\Delta - P_{G}'$ so that the
first and the third pieces are still contained in $B_{2
\delta}(a_{n}')$ and $B_{2 \delta}(b_{n}')$, respectively, but the
second piece coincide with the second piece of $\phi_{n}$.   Then do
the same thing for $\widehat{\psi}_{n}(\gamma)$. For an illustration
of the resulted curves, see Figure 7.

Let us use $[A, B]$ and $[C, D]$ to denote the second pieces of
$\phi_{n}(\gamma)$ and $\psi_{n}(\gamma)$, respectively. Since
$\phi_{n}(\gamma)$ is homotopic to $\psi_{n}(\gamma)$, the curve
segment $[C, D]$ can be deformed to $[A, B]$ in $\Delta -P_{G_{n}}'$
so that $C$ moves to $A$ along $\partial B_{\delta}(a_{n}')$ and $D$
moves to $B$ along $\partial B_{\delta}(b_{n}')$, and moreover, the
deformation can be taken such that it does not intersect the
$\delta/2$-neighborhood of $P_{G_{n}}'$. Since $P_{G_{n}}'$ is close
to $P_{G}'$ as $n$ is large, this deformation does not intersect
$P_{G}'$.   Since
$$B_{2\delta}(a_{n}') \cap (P_{G}'-\{a'\}) =
 B_{2 \delta}(b_{n}') \cap (P_{G}' - \{b'\}) =
\emptyset, $$  the first piece and the third piece of
$\widehat{\phi}_{n}(\gamma)$ can be deformed to the corresponding
piece of $\widehat{\psi}_{n}(\gamma)$ in $B_{2 \delta}(a_{n}')$ and
$B_{2 \delta}(b_{n}')$, respectively.  It is not difficult to see
that  these deformations can be taken carefully so that they can be
glued into a homotopy between $\widehat{\phi}_{n}(\gamma)$ and
$\widehat{\psi}_{n}(\gamma)$ in $\Delta - P_{G}'$.

The second case can be treated in a similar way. We leave it to the
reader.
\end{proof}

\subsection{Proof of Theorem A}
\begin{figure}
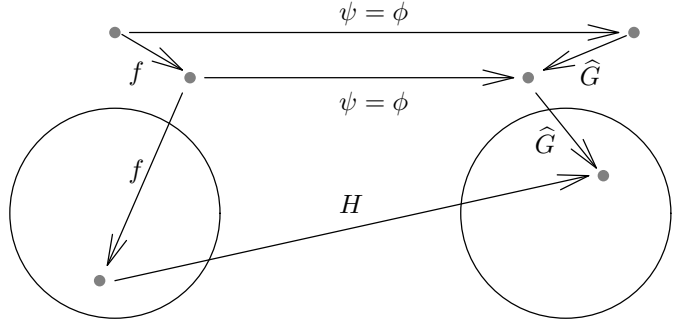

\bigskip
\begin{center}
\centertexdraw { \drawdim cm \linewd 0.02 \move(0 1)

\move(-3 -2.8) \lcir r:1.4 \move(3 -2.8) \lcir r:1.4

\move(-2 -1) \fcir f:0.5 r:0.08 \move(-3 -0.4) \fcir f:0.5 r:0.08

\move(2.5 -1) \fcir f:0.5 r:0.08 \move(3.9 -0.4) \fcir f:0.5 r:0.08

\move(-3.2 -3.7) \fcir f:0.5 r:0.08 \move(3.5 -2.3) \fcir f:0.5
r:0.08

\move(-1.8 -1) \arrowheadtype t:V \avec(2.3 -1)

\move(-2.8 -0.4) \arrowheadtype t:V \avec(3.7 -0.4)

\move(-2.9 -0.42) \arrowheadtype t:V \avec(-2.1  -0.9)

\move(3.8 -0.44) \arrowheadtype t:V \avec(2.7  -0.9)

\move(-2.1 -1.2) \arrowheadtype t:V \avec(-3.1 -3.5)

\move(2.6 -1.2) \arrowheadtype t:V \avec(3.4 -2.2)

\move(-3 -3.7) \arrowheadtype t:V \avec(3.3 -2.3)

\move(0 -0.3)\htext{$\psi = \phi$}

\move(0 -1.5)\htext{$\psi = \phi$} \move(0 -2.8)\htext{$H$}

\move(-2.8 -2.4)\htext{$f$} \move(-2.8 -1.1)\htext{$f$}

\move(3.2 -1.1)\htext{$\widehat{G}$} \move(2.6
-2)\htext{$\widehat{G}$} }
\end{center}
\vspace{0.2cm} \caption{Realize the combinatorial structure in
$\Delta$}
\end{figure}
\subsubsection{Realizing the combinatorics  in the rotation disk}
Let $G$ be the Blaschke product obtained in the last section.  Since
$G|\partial \Delta$ is an analytic critical circle homeomorphism
with $\emph{bounded type}$ rotation number, by Herman-Swiatek's
theorem, $G|\partial \Delta = h \circ R_{\theta}\circ h^{-1}$ where
$h:
\partial \Delta \to \Delta$ is a quasi-symmetric homeomorphism with
$h(1) = 1$. All we need to do now is to follow the standard
procedure to do the quasiconformal surgery on $G$. There are many
places where a detailed description of this surgery can be found(see
for example, \cite{Pe1}, \cite{YZ} and \cite{Z1}).

The first thing we need to take care of is the combinatorial
structure of $f$ in the inside of the $\emph{rotation disk}$, which
is not reflected by the Blaschke product $G$ (see
Remark~\ref{additional}). Recall that
$$
X = \{z \in \Omega_{f} \;\big{|}\;f^{i}(z) \in \Delta -\{0\} \hbox{
for some } i > 0\}.
$$
We may assume  that $X \ne \emptyset$, for otherwise, we just skip
this step.  For $z \in X$, let $i_{z} > 0$ be the least integer such
that $f^{i_{z}}(z) \in \Delta$.  Now we can extend $h: \partial
\Delta \to \partial \Delta$ to a quasiconformal homeomorphism $H:
\Delta \to \Delta$ by using Douady-Earle's extension
theorem\cite{DE}.  By composing $H$ with an appropriate
quasiconformal homeomorphism $\tau: \Delta \to \Delta$ with
$\tau|\partial \Delta = id$, which is still denoted by $H$, we may
assume that $H(0) = 0$ and $$H(f^{i_{z}}(z)) = G^{i_{z}}((\psi(z))$$
 for each $z \in X$(see Figure 8).

\subsubsection{Quasiconformal surgery}
Define a $\emph{modified}$ Blaschke product as follows.
\begin{equation}\label{modi-B}
\widehat{G}(z) =
\begin{cases}
 G(z) & \text{ for $|z| \ge 1$}, \\
 H \circ R_{\theta} \circ H^{-1}(z)& \text{ for $z \in \Delta$}.
\end{cases}
\end{equation}

\begin{lemma}\label{com-equ}
$\widehat{G}$ is $\emph{combinatorially equivalent}$ to $f$ rel
$P_{f} \cap \{\infty\}$.
\end{lemma}
\begin{proof}
Let $\widehat{\phi}_{n}$ and $\widehat{\psi}_{n}$ be the
homeomorphisms obtained in $\S2.4$. Let $\phi = \widehat{\phi}_{n}$
and $\psi = \widehat{\psi}_{n}$. By Lemma~\ref{pair} and \ref{iso},
$\phi$ and $\psi$ are isotopic to each other rel $P_{F}'$, and $G =
\phi \circ F \circ \psi^{-1}$.
 Define
\begin{equation}
\omega_{0}(z) =
\begin{cases}
 \phi(z) & \text{ for $|z| \ge 1$}, \\
 H(z) & \text{ for $z \in \Delta$}.
\end{cases}
\end{equation}
Since $\widehat{G}$ and $f$ have the same combinatorial structure on
the outside of the unit disk, for $k = 1, 2, \cdots$, we can lift
$\omega_{k-1}$ by the equation
$$
\widehat{G} \circ \omega_{k} =  \omega_{k-1} \circ f
$$
and get a sequence of quasiconformal homeomorphisms $\omega_{n}$.
Note that $\omega_{1} = \psi$ on the outside of $f^{-1}(\Delta)$. It
follows that up to a homotopy, the only possible places where
$\omega_{k-1}$ and $\omega_{k}$ are different are the components of
$\bigcup_{l=1}^{\infty}f^{-l}(\Delta)$ which intersect $P_{f}$. Let
$N = |P_{f} - \partial \Delta|$. It follows that for each $x \in
P_{f}$, either the forward orbit of $x$ under $f$ is eventually
finite, or $f^{N+1}(x) \in \overline{\Delta}$. This implies if a
component of $\bigcup_{l=1}^{\infty}f^{-l}(\Delta)$ intersects
$P_{f}$, it must be one of the components of $f^{-N-1}(\Delta)$. On
the other hand, it is easy to see that $$\omega_{N+1} =
\omega_{N+2}$$ on all the components of $f^{-N-1}(\Delta)$. It
follows that $\omega_{N+1}$ and $\omega_{N+2}$ are
$\emph{combinatorially equivalent}$ to each other  rel $P_{f} \cup
\{\infty\}$. The lemma follows.

\end{proof}

Now let us  define a $\widehat{G}$-invariant complex structure $\mu$
as follows. Let
$$
\mu(z) = \frac{(H^{-1})_{\bar{z}}}{(H^{-1})_{z}}
$$
for $z \in \Delta$. For $z \notin \Delta$, there are two cases. In
the first case, the forward orbit of $z$ falls into the inside of
the unit disk. Let $k > 0$ be the least integer such that
 $G^{k}(z) \in \Delta$.  We define $\mu(z) = (G^{k})^{*}(\mu(G^{k}(z)))$,
 that is, we pull back by $G^{k}$ the complex structure of $H^{-1}$ at
 $G^{k}(z)$ to $z$. In the second case, the forward orbit of $z$ is
 contained in the outside of the unit disk. In this case,  we define
$\mu(z)=0$. By this way we get a $\widehat{G}$-invariant complex
structure $\mu(z)$ on the whole Riemann sphere. Since $\widehat{G}$
is holomorphic outside the unit disk, it follows that
$$
\|\mu\|_{\infty} = \sup_{z \in
\Delta}\bigg{|}\frac{(H^{-1})_{\bar{z}}}{(H^{-1})_{z}}\bigg{|} <1.
$$ By Ahlfors-Bers theorem, there is a quasiconformal homeomorphism
$\Phi: S^{2} \to S^{2}$ such that $\mu_{\Phi} =\mu$ and $\Phi$ fixes
$0$, $1$, and the infinity.   Now let $$g = \Phi \circ \widehat{G}
\circ\Phi^{-1}.$$  It follows that $g$ is a rational map which has a
Siegel disk centered at the origin. Let us denote the Siegel disk by
$D_{g}$. It follows that $\partial D_{g}$ is the image of the unit
circle under $\Phi$, hence  is a quasi-circle which passes through
the critical point $1$ of $g$. This implies that $g \in
R_{\theta}^{geom}$. By Lemma~\ref{com-equ}, we have $$g = \Phi \circ
h_{N+1} \circ f \circ h_{N+2}^{-1} \circ \Phi^{-1}.$$  Note that $
h_{N+2}^{-1} \circ \Phi^{-1}|D_{g}: D_{g} \to \Delta$ is a
holomorphic homeomorphism. Therefore, $g$ $\emph{realizes}$ the
topological branched covering map $f$ in the sense of
Definition~\ref{realize}.   This completes the proof of Theorem A.

\begin{center}
\section{ Combinatorial Rigidity of the Maps in $R_{\theta}^{geom}$}
\end{center}

\subsection{Blaschke Models for Maps in $R_{\theta}^{geom}$}

Let $G$ be a Blaschke product such that $G|\partial \Delta = h \circ
R_{\theta} \circ h^{-1}$ where $h: \partial \Delta \to \partial
\Delta$ is a quasi-symmetric homeomorphism with $h(1) = 1$. Let $H:
\Delta \to \Delta$ be a quasiconformal extension of $h$ to the unit
disk. Let $\widehat{G}$ be the $\emph{modified}$ Blaschke product
defined by (\ref{modi-B}). We say a Siegel rational map $g$ is
$\emph{modeled}$ by the Blaschke product $G$  if there is a
quasiconformal homeomorphism $\phi:S^{2} \to S^{2}$ such that $g =
\phi^{-1} \circ \widehat{G} \circ \phi$. We have
\begin{lemma}\label{3-fir}
For each $g \in R_{\theta}^{geom}$, there is a Blaschke product $G$
which models $g$.
\end{lemma}

\begin{proof} Let $D_{g}$ be the Siegel disk of $g$
and $\psi: D_{g} \to \Delta$ be the holomorphic homeomorphism which
conjugates $g|D_{g}$ to the rigid rotation $R_{\theta}: \Delta \to \Delta$.
Since $\partial D_{g}$ is a quasi-circle, we can extend $\psi$ to be
a quasiconformal homeomorphism of the Riemann sphere: $S^{2} \to S^{2}$.
Then $f = \psi \circ g \circ \psi^{-1} \in R_{\theta}^{top}$
is $\emph{realized}$ by $g$. Since $g$ has no Thurston obstructions
outside the Siegel disk, $f$ has no Thurston obstructions outside its
$\emph{rotation disk}$. By Theorem A, there is a
$\widehat{g} \in R_{\theta}^{geom}$ which realizes $f$, and which
can be modeled by some Blaschke product $G$. That is to say,
$\widehat{g} = \phi_{1} \circ \widehat{G} \circ \phi_{1}^{-1}$
where $\phi_{1}:S^{2} \to S^{2}$ is a quasiconformal homeomorphism
and $\widehat{G}$ is the $\emph{modified}$ Blaschke product.

On the other hand, since $g$ and $\widehat{g}$ both $\emph{realize}$
the topological branched covering map $f$, it follows that $g$ and
$\widehat{g}$ are $\emph{combinatorially equivalent}$. Since the
boundaries of the Siegel disks $D_{g}$ and $D_{\widehat{g}}$   are
both quasi-circles, and since $P_{g} - \overline{D_{g}}$ and
$P_{\widehat{g}} - \overline{D_{\widehat{g}}}$ are both finite sets,
it follows that $g$ and $\widehat{g}$ are quasiconformally
equivalent. By a theorem of McMullen (see also \cite{C2}), $g$ and
$\widehat{g}$ are quasiconformally conjugate to each other. Let $
\phi_{2}:S^{2} \to S^{2}$ be a quasiconformal homeomorphism such
that $g  = \phi_{2} \circ \widehat{g} \circ \phi_{2}^{-1}$. We thus
get  $g = \phi_{2} \circ  \phi_{1} \circ \widehat{G} \circ
\phi_{1}^{-1}\circ \phi_{2}^{-1}$.  The lemma follows.

\end{proof}

Let $g \in R_{\theta}^{geom}$ and $D_{g}$ be the Siegel disk. Assume
that  $J_{g}$ has positive measure.   Since $\partial D_{g}$ is a
quasi-circle, it follows that $\bigcup_{k=0}^{\infty}g^{-k}(\partial
D_{g})$ is a zero measure set.  Let $x_{0}$ be a Lebesgue point of
$J_{g} - \bigcup_{k=0}^{\infty}g^{-k}(\partial D_{g})$.  By
Proposition 1.14\cite{L}, $w(x_{0}) \subset P_{g}'$ where $w(x_{0})$
is the $w$-limit set of $x_{0}$ and $P_{g}'$ is the derived set of
$P_{g}$. Since $P_{g} - \partial D_{g}$ is a finite set, it follows
that $g^{k}(x_{0}) \to \partial D_{g}$ as $k \to \infty$. By
Lemma~\ref{3-fir}, $g$ is $\emph{modeled}$ by a Blaschke product
$G$. That is to say, there is a quasiconformal homeomorphism
$\phi:S^{2} \to S^{2}$ such that $g = \phi \circ \widehat{G} \circ
\phi^{-1}$. Let $z_{0} = \phi^{-1}(x_{0})$, and
\begin{equation}
J_{\widehat{G}} = J_{G} - \bigcup_{k=0}^{\infty}G^{-k}(\Delta).
\end{equation}
It follows that $J_{\widehat{G}} = \phi^{-1}(J_{g})$ is a set of
positive measure. Since quasiconformal maps preserve zero-measure
sets, we have
\begin{lemma}\label{Lyu}
$z_{0}$ is Lebesgue point of $J_{\widehat{G}} -
\bigcup_{k=0}^{\infty}G^{-k}(\partial \Delta)$, and $G^{k}(z_{0})
\to \partial \Delta$ as $k \to \infty$.
\end{lemma}

\subsection{Contraction Regions of $G^{-1}$}
Let  $c \in \partial \Delta \cap \Omega_{G}$ and $v = G(c)$. Suppose
that the local degree of $G$ at $c$ is $2m+1$ for some integer $m\ge
1$.  For $\delta > 0$ small, denote $U_{\delta}(v) = B_{\delta}(v)
\cap \{z\:\big{|}\:|z| > 1\}$. Then there are exactly $m+1$ inverse
branches of $G$ which map $U_{\delta}(c)$ to $m+1$ domains which are
attached to $c$ from the outside of the unit disk. In this section,
we will show that for each $c \in \Omega_{G} \cap \partial \Delta$,
there exists a region $W_{v} \subset U_{\delta}(v)$ which is
attached to the critical value $v$,  such that when restricted on
$W_{v}$, all these $m+1$ inverse branches of $G$ strictly contract
the hyperbolic metric on some appropriate Riemann surface.

Let $\Omega_{*} = {\Bbb P}^{1} - (\Delta \cup P_{G})$ and
$\Omega^{*} = {\Bbb P}^{1} - (\Delta \cup (G^{-1}(\Delta \cup
P_{G})))$.  Note that $\Omega^{*}$ may not be connected, and in that
case, each component of $\Omega^{*}$ is a hyperbolic Riemann
surface. We use $d\rho_{*} = \lambda_{\Omega_{*}}|dz|$ to denote the
hyperbolic metric of $\Omega_{*}$. To save the symbols, we use the
same notation $\Omega^{*}$ to denote the component with which we are
concerned,  and  $d \rho^{*} = \lambda_{\Omega^{*}}|dz|$ to denote
the hyperbolic metric on that component.  It follows that  $G:
\Omega^{*} \to \Omega_{*}$ is a holomorphic covering map.

Let $r> 0$ be small enough and $B_{r}(c)$ be the disk centered at
$c$ with radius $r$. Then there are exactly $m+1$ domains which are
contained in $$B_{r}(c) \cap \{z\:\big{|}\: |z| > 1\},$$ and which
are mapped to the outside of the unit disk. Each of these domains is
attached to $c$. Moreover, for each of such domains, the boundary of
the domain has an inner angle $\pi/(2m+1)$ at $c$.  Take $0 <
\epsilon < 1/(4m+2)$. Let $R$ and $L$ be the two rays starting from
$c$ such that the angles between $\partial \Delta$ and $R$,
$\partial \Delta$ and $L$, are both equal to $\epsilon \pi$. Let
$S_{\epsilon}^{c}$ be the cone spanned by $R$ and $L$ which is
attached to $c$ from the outside of the unit disk(see Figure 9,
where $m  =2$). Set
$$
\Omega_{\epsilon,r}^{c} = S_{\epsilon}^{c} \cap \Omega^{*} \cap
B_{r}(c).
$$
\begin{figure}
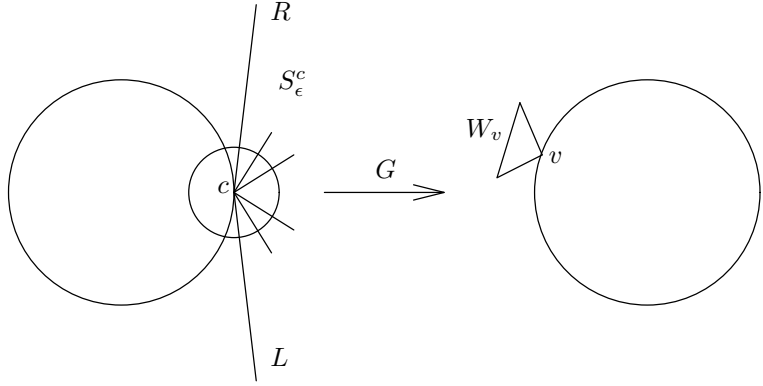

\bigskip
\begin{center}
\centertexdraw { \drawdim cm \linewd 0.02 \move(0 1)

\move(-3.5 -2.5) \lcir r:1.5

\move(3.5 -2.5) \lcir r:1.5

\move(-2 -2.5) \lcir r:0.6

\move(-2 -2.5) \lvec(-1.7 0) \move(-2 -2.5) \lvec(-1.7 -5) \move(-2
-2.5) \lvec(-1.5 -1.7) \move(-2 -2.5) \lvec(-1.2 -2)

\move(-2 -2.5) \lvec(-1.5 -3.3) \move(-2 -2.5) \lvec(-1.2 -3)

\move(2.1  -2) \lvec(1.8 -1.3)\move(2.1 -2) \lvec(1.5 -2.3)\lvec(1.8
-1.3)

\move(-2.2 -2.5) \htext{$c$} \move(2.2 -2.1) \htext{$v$}

\move(1.1 -1.8) \htext{$W_{v}$}

\move(-1.5 -0.2) \htext{$R$} \move(-1.5 -4.8) \htext{$L$}

\move(-1.4 -1.2)\htext{$S_{\epsilon}^{c}$}

\move(-0.8 -2.5) \arrowheadtype t:V \avec(0.8  -2.5) \move(-0.1
-2.3) \htext{$G$}

 }
\end{center}
\vspace{0.2cm} \caption{The contraction region of $G^{-1}$}
\end{figure}

The following lemma says that on $\Omega_{\epsilon,r}^{c}$, $G$
strictly increases the hyperbolic metric in $\Omega_{*}$. The lemma
is a  general version of Lemma 1.11 in [Pe],
\begin{lemma}\label{contra}
There is a $\delta>0$ which depends only on $\epsilon$ such that for
all $r> 0$ small enough and any $c \in \partial \Delta \cap
\Omega_{G}$,   we have
$$\lambda_{\Omega_{*}}(G(x))|G'(x)| \ge (1 +
\delta)\lambda_{\Omega_{*}}(x)$$ where  $x$ is an arbitrary point in
$\Omega_{\epsilon,r}^{c}$.
 \end{lemma}

\begin{proof}
Assume that $r > 0$ is small.  Take any point $x \in
\Omega_{\epsilon ,r}^{c}$. Note that $\Omega_{\epsilon ,r}^{c}$ may
not be connected. We need only to consider the case that $x$ lies in
a component which has part of its boundary on $R$ or $L$, for in the
other cases, $G^{-1}(\partial \Delta)$ does much more contributions
to the hyperbolic density function $\lambda_{\Omega^{*}}$, and
therefore the value $\delta$ can actually be made bigger. This will
be clear from the following proof.

Since $G: \Omega^{*} \to \Omega_{*}$ is a holomorphic covering map,
we have
$$
\lambda_{\Omega_{*}}(G(x))|G'(x)| = \lambda_{\Omega^{*}}(x).
$$
So it is sufficient to prove that
$$
\lambda_{\Omega^{*}}(x)/\lambda_{\Omega _{*}}(x) \ge 1+ \delta.
$$

Since $r$ is small, when viewed from the point $x$, $\Omega^{*}$ is
approximately an angle domain near the vertex $c$, with
 angle $\alpha \pi$, where $\alpha = 1 /(2m +1)$.
 By taking an appropriate coordinate system, we may write
 $x = c+  \eta e^{i \lambda \pi}$ where $\epsilon < \lambda < \alpha <1$,
 and $0< \eta < r$. Thus we get
$$
\lambda_{\Omega^{*}}(x) \approx(\frac{1}{\alpha})\eta
^{\frac{1}{\alpha}-1} \frac{1}{\eta^{\frac{1}{\alpha}}\sin
\frac{1}{\alpha} \lambda \pi}=\frac{1}{\eta  \alpha \sin
\frac{1}{\alpha} \lambda \pi}.
$$

On the other hand, when viewed from $x$, $\Omega_{*}$ is
approximately the half plane, therefore,
$$
\lambda_{\Omega_{*}}(x) \approx \frac{1}{\eta \sin \lambda \pi}.
$$
This gives us
$$
\lambda_{\Omega^{*}}(x)/\lambda_{\Omega_{*}}(x) \approx \frac{ \sin
\lambda \pi} {\alpha \sin \frac{\lambda\pi}{\alpha}} > \frac{ \sin
\epsilon \pi} {\alpha \sin \frac{\epsilon \pi}{\alpha}} > 1.
$$

\end{proof}

\subsection{ Closed Half Hyperbolic Neighborhood}

Let $I = [a, b] \subset {\Bbb R}$ be an interval segment. Denote
${\Bbb C}_{I} = {\Bbb C}-({\Bbb R} - I)$. For a given $d>0$, the
hyperbolic neighborhood of the interval $I$ in the slit plane ${\Bbb
C}_{I}$ is defined to be the set which consists of all the points
$x$ such that $d_{{\Bbb C}_{I}}(x, I) < d$, where $d_{{\Bbb C}_{I}}$
denotes the hyperbolic distance in  ${\Bbb C}_{I}$. Let us use
$U_{d}(I)$ to denote this hyperbolic neighborhood. It is known that
the set $U_{d}(I)$ is a domain bounded by two Euclidean arcs which
are symmetric about the real line. The exterior angle between the
Euclidean arc and the interval $I$ is uniquely determined by $d$,
and let us denote this angle by $\alpha(d)$ ( for an explicit
formula of $\alpha(d)$, see \cite{Ya}). Such an object was first
introduced by Sullivan to complex dynamics and now becomes a popular
tool in this area.

Now let us adapt this object so that it is suitable for our
situation.  For each arc segment $I \subset \partial \Delta$, Let
$$
\Omega_{I} = {\Bbb P}^{1}- (P_{G} - I).
$$
For any two points $x, y \in \Omega_{I}$, let $d_{\Omega_{I}}(x, y)$
denote the distance between $x$ and $y$ with respect to the
hyperbolic metric on $\Omega_{I}$.  Let
\begin{equation}
H_{d}(I) = \{z \in \Omega_{I}\: | \: d_{\Omega_{I}}(z, I) \le d,
\hbox{and} \: |z| \ge 1\}.
\end{equation}
where $d_{\Omega_{I}}(x, y)$ is the hyperbolic distance between $x$
and $y$ in $\Omega_{I}$.

Let  $A_{\alpha}(I) \subset \{z: |z| > 1\}$  denote the arc segment
of some Euclidean circle such that it has the same ending points as
$I$ and such that the exterior angle between $A_{\alpha}(I)$ and $I$
is equal to $\alpha$. Let $$\gamma_{d}(I) =
\partial H_{d}(I) - I.$$  Note that  $\gamma_{d}(I)$ may not be an
 arc segment of some Euclidean circle.
 But since $P_{G} - \partial \Delta$ is a finite set,
it follows that when $|I|$ is small enough, the set $P_{G} -
\partial \Delta$ will do very little contribution to the hyperbolic
density of the points near the arc $I$, and thus $\gamma_{d}(I)$ is
like the Euclidean arc $A_{\alpha(d)}(I)$. Let us formulate this as
the next lemma and leave the proof to the reader.
\begin{lemma}\label{approx}
For any $\delta >0$, there is an $\epsilon > 0$, such that when $|I|
< \epsilon$, $\gamma_{d}(I)$ lies in between the two Euclidean arcs
$A_{\alpha(d)+\delta}(I)$ and $A_{\alpha(d)-\delta}(I)$.
\end{lemma}

Let us fix $d> 0 $ throughout the following discussions.  Let
$$
\Lambda_{G} = \{c \in \Omega_{G} - \partial \Delta \big{|}\: \exists
\: k \ge 1 \hbox{ such that }G^{k}(c) \in \partial \Delta\}.
$$
For any $c \in \Lambda_{G}$, let $k(c) \ge 1$ be the least integer
such that  $G^{k(c)} \in \partial \Delta$. Let
$$
X_{G} = \{G^{k(c)}(c)\:\big{|}\: c \in \Lambda_{G}\}.
$$

\begin{lemma}\label{sch}
Let $J \subset \partial \Delta$ with $J \cap \Omega_{G} =
\emptyset$. Let $I = G(J)$.  Suppose that $I \cap X_{G} =
\emptyset$. Then $V \subset H_{d}(J)$ where $V$ is the connected
component of $G^{-1}(H_{d}(I))$ which is attached to $J$ from the
outside of the unit disk.
\end{lemma}
\begin{figure}
\bigskip
\begin{center}
\centertexdraw { \drawdim cm \linewd 0.02 \move(0 1)

\move(0 -2.5) \larc r:1.7 sd:0 ed:85 \larc r:1.7 sd:95 ed:175 \larc
r:1.7 sd:185 ed:360

\move(-1.7 -2.5) \lcir r:0.5 \move(0 -0.8) \lcir r:0.5 \move(0 -4)

\arrowheadtype t:V
 \move(-0.3 -1.3) \avec(-1.2 -2.2)

 \move(-0.6 -2) \htext{$\Phi$}

 \move(-2.1 -2.65) \htext{$\Sigma^{*}$}

 \move(-0.14 -0.75) \htext{$\widetilde{\Sigma}$}

 }
\end{center}
\vspace{0.2cm} \caption{The map $\Phi: \widehat{\Sigma} \to
\Sigma^{*}$}
\end{figure}
\begin{proof}
Let $\widehat{\Omega}_{J}=  {\Bbb P}^{1} - G^{-1}(P_{G}- I) $. By
the assumption,  it follows that  $G: \widehat{\Omega}_{J} \to
\Omega_{I} $ is a holomorphic covering map.  For any two points $x,
y \in \widehat{\Omega}_{J}$, let  $d_{\widehat{\Omega}_{J}}(x, y)$
be the distance between $x$ and $y$ with respect to the hyperbolic
metric on $\widehat{\Omega}_{J}$. It follows that
$$
V \subset \{z \: \big{|}\: d_{\widehat{\Omega}_{J}}(z, J) \le d, |z|
\ge 1\}.
$$
Since $I \cap X_{G} = \emptyset$, we have $\widehat{\Omega}_{J}
\subset \Omega_{J}$, and therefore $d_{\Omega_{J}}(x, y) <
d_{\widehat{\Omega}_{J}}(x, y)$ for any two points in
$\widehat{\Omega}_{J}$.  This implies that
$$
\{z \: \big{|}\: d_{\widehat{\Omega}_{J}}(z, J) \le d, |z| \ge 1\}
\subset H_{d}(J).
$$
The lemma follows.
\end{proof}

\begin{lemma}\label{sch-a}
Let $d' > d$. Then there is a $\ell > 0$ such that for every $J
\subset \partial \Delta$ satisfying \begin{itemize} \item[1.] $|J|<
\ell$, \item[2.]  $J \cap \Omega_{G} = \emptyset$, \item[3.] $I \cap
X_{G} \ne \emptyset$ where $I = G(J)$, \end{itemize} we have  $V
\subset H_{d'}(J)$ where $V$ is the connected component of
$G^{-1}(H_{d}(I))$ which is attached to $J$ from the outside of the
unit disk.
\end{lemma}
\begin{proof}
Since $P_{G} - \partial \Delta$ is a finite set, there is a $\delta
> 0$ such that
$$
(B_{\delta}(x)-\partial \Delta) \cap P_{G} = \emptyset.
$$
where $x$ is the mid-point of $I$.  Let
$$
\widetilde{\Sigma} = B_{\delta}(x) - (\partial \Delta - I).
$$
It is clear  that $\widetilde{\Sigma}$ is a simply connected domain.
Since $J \cap \Omega_{G} = \emptyset$ and $\widetilde{\Sigma} - I$
contains no critical value of $G$, there is an inverse branch of
$G$, say, $\Phi$, defined on $\widetilde{\Sigma}$ which maps
$\widetilde{\Sigma}$ to some domain containing $J$. Let
$$\Sigma^{*} = \Phi(\widetilde{\Sigma}).$$ See Figure 10 for an
illustration. Let $d_{\widetilde{\Sigma}}(,)$ and
$d_{\Sigma^{*}}(,)$ denote the hyperbolic distance in
$\widetilde{\Sigma}$ and $\Sigma^{*}$, respectively. Define
$$
\widetilde{H}_{d'}(I) = \{z\in \widetilde{\Sigma}\:\big{|}\:
d_{\widetilde{\Sigma}}(z, I) \le d' \hbox{ and } |z| \ge 1\}
$$
and
$$
H_{d}^{*}(I) = \{z\in \Sigma^{*}\:\big{|}\: d_{\Sigma^{*}}(z, I) \le
d \hbox{ and } |z| \ge 1\}.
$$
Since for $I$ small, when viewed from the points near $I$, the
difference between $\Omega_{I}$ and $\widetilde{\Sigma}$ is small,
it follows that
$$
H_{d}(I) \subset \widetilde{H}_{d'}(I)
$$
provided that $|I|$ is small enough. Since $\Phi: \widetilde{\Sigma}
\to \Sigma^{*}$ is a holomorphic isomorphism and $\Sigma^{*} \subset
\Omega_{J}$,  we have
$$
V  = \Phi(H_{d}(I)) \subset  \Phi(\widetilde{H}_{d'}(I)) =
H_{d'}^{*}(J) \subset H_{d'}(J).
$$
The lemma follows.

\end{proof}

\subsection{ Minimal Neighborhoods}
Let $z_{0}$ be the point in Lemma~\ref{Lyu}. Let $z_{k} = G^{k}
(z_{0})$ for $k \ge 1$.  In $\S3.2$, we show that there exist
regions which are attached to the critical values on the unit
circle, such that in these regions, $G^{-1}$ strictly contracts the
hyperbolic metric in $\Omega_{*}$. Our next step is to show that,
there will be some infinite subsequence of $\{z_{k}\}$ which passes
through these contraction regions. To prove the existence of such
infinite subsequence, we will introduce an object, called
$\emph{minimal neighborhood}$.

Recall that  $G|\partial \Delta = h \circ R_{\theta }\circ h^{-1}$,
where $h: \partial \Delta \to \partial \Delta$ is a quasi-symmetric
homeomorphism with $h(1) = 1$.   Now for each arc $I \subset
\partial \Delta$, we define $\sigma(I) = |(h^{-1}(I)|$.
It follows from the definition that  $\sigma$ is $G$-invariant.

\begin{lemma}\label{geom}
Let $\delta > 0$ be small. Then there exists a $\tau > 0$ such that
for any two arcs $I, J \subset\partial \Delta$ with $I \cap J \ne
\emptyset$ and $|J| < \tau|I|$, we have $\sigma(J) < \delta
\sigma(I)$.
\end{lemma}
The proof is easy and we shall leave the details to the reader.

\begin{figure}
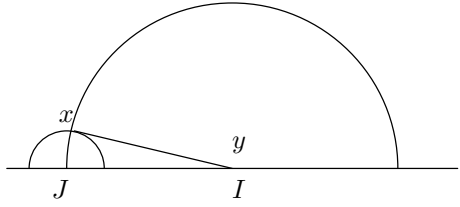

\bigskip
\begin{center}
\centertexdraw { \drawdim cm \linewd 0.02 \move(0 1)

\move(-3 -2.5) \lvec(3 -2.5)

\move(0 -2.5) \larc r:2.2 sd:0 ed:180

\move(-2.2 -2.5) \larc r:0.5 sd:0 ed:180

\move(0 -2.5) \lvec(-2.1 -2)

\move(-2.4 -2.9) \htext{$J$}

\move(0 -2.9) \htext{$I$}

\move(0 -2.3) \htext{$y$}

\move(-2.3 -1.9) \htext{$x$}

 }
\end{center}
\vspace{0.2cm} \caption{$\sigma(J) < \delta\sigma(I)$ when the angle
between $[x, y]$ and $I$ is small enough}
\end{figure}

\begin{lemma}\label{min-arc}
Let $\delta > 0$ be small. Then there is a $\rho > 0$ and an
$\epsilon
> 0$ such that for any $I \subset \partial \Delta$ with $|I|<
\epsilon$ and any $x \in H_{d}(I)$ and  $y \in I$, if the angle
between the segment $[x,y]$ and $\partial \Delta$ is less than
$\rho$, then there is an arc $J \subset \partial \Delta$ such that
\begin{itemize}
\item [1.] $x \in H_{d}(J)$, and
\item [2.] $\sigma(J) < \delta \sigma(I)$.
\end{itemize}
\end{lemma}

\begin{proof}
We may consider the worst case, that is, $x \in \gamma_{d}(I)$. See
Figure 11 for an illustration. If $\epsilon$ is small, by
Lemma~\ref{approx}, $\gamma_{d}(I)$ lies in between two Euclidean
arcs which have the same ending points as the arc $I$. So if $\rho$
is small enough, $x$ must be close to one of the end points of $I$.
On the other hand, by Lemma~\ref{approx} again, if $x \in
\gamma_{d}(I)$ is close to one of the end points of $I$, say $a$, it
must be contained in $H_{d}(J)$ for some $J$ with $$|J| = O(d(x,
a))$$ and $a$ being the middle point of $J$. Clearly, as $\rho \to
0$,
$$d(x, a)/|I| \to 0.$$ It follows that by taking $\rho$ small,
$|J|/|I|$ can be as small as wanted, and hence by Lemma~\ref{geom},
$\sigma(J) < \delta \sigma (I)$. Lemma~\ref{min-arc} follows.
\end{proof}

Let $k \ge 0$ be an integer. Define
\begin{equation}
\Phi_{k} = \{I \subset \partial \Delta\big{|}\: z_{k} \in
H_{d}(I)\},
\end{equation}
and
\begin{equation}\label{leng}
l_{k} = \inf \{\sigma(I)\big{|}\: I \in \Phi_{k}\}.
\end{equation}
\begin{remark}\label{obt-min}
Note that by  taking a limit of a convergent subsequence of the
intervals, the value $l_{k}$ in (\ref{leng}) can be obtained by some
interval $I \in \Phi_{k}$.
\end{remark}

Since $z_{k} \to \partial \Delta$, we have
\begin{lemma}\label{tozero}
$l_{k} \to 0$ as $k \to \infty$.
\end{lemma}

\begin{definition}\label{key-s}
For each $n$, we define $0 \le m(n) \le n$ to be the least integer
such that
\begin{equation}
l_{m(n)} = \min \{l_{k} \big{|} \: 0 \le k \le n \}.
\end{equation}
\end{definition}

The following two lemmas follow directly from  the definition of
$m(n)$ and Lemma~\ref{tozero}.
\begin{lemma}
$m(n) \le m(n+1)$, and  $m(n) \to \infty$ as $n \to \infty$.
\end{lemma}

\begin{lemma}\label{min-neig}
For each $m(n)$, there is an open arc, say $I_{m(n)} \subset
\partial \Delta$, which may not be unique,   such that
$\sigma(I_{m(n)}) = l_{m(n)}$ and $z_{m(n)} \in H_{d}(I_{m(n)})$.
\end{lemma}
\begin{proof}
As mentioned in Remark~\ref{obt-min}, by taking a  convergent
subsequence of the intervals, we can get an interval $I\subset
\partial \Delta$ such that $\sigma(I) = l_{m(n)}$ and $I \in
\Phi_{m(n)}$. Let us denote this interval by $I_{m(n)}$. By the
minimal property of $I_{m(n)}$, it follows that
$d_{\Omega_{I_{m(n)}}}(z_{m(n)}, I_{m(n)}) = d$ and therefore,
$z_{m(n)} \in \gamma_{d}( I_{m(n)}) \subset H_{d}( I_{m(n)})$.
 \end{proof}

We call the region  $H_{d}(I_{m(n)})$ in Lemma~\ref{min-neig}  a
$\emph{minimal neighborhood} $ associated to the number $m(n)$. From
the proof of Lemma~\ref{min-neig},  $z_{m(n)} \in
\gamma_{d}(I_{m(n)})$.

\begin{lemma}\label{fin-lem}
There exist  $\epsilon > 0$ and $r > 0$ and an increasing sequence
of integers $\{\tau(n)\}$ such that for all $n$ large enough,
$z_{\tau(n)} \in \Omega_{\epsilon, r}^{c}$ for some $c \in
\Omega_{G} \cap \partial \Delta$.
\end{lemma}

The idea of the proof is based on the following fact: for $r>0$
small and $c \in \Omega_{G} \cap \partial \Delta$, if $y, y' \in
B_{r}(c)$ such that $G(y) = G(y')$, then $\alpha$ and $\alpha'$ can
not be both small, where $\alpha$ is the angle between $\partial
\Delta$ and the straight segment $[c, y]$, and $\alpha'$ is the
angle between $\partial \Delta$ and the straight segment $[c, y']$.
\begin{proof}
To fixed the ideas, let $\epsilon > 0$ and  $r> 0$  be two small
numbers and  $N$ be a large integer.  By taking $N$ large, we may
assume that when $n \ge N$, the interior of the $\emph{minimal
neighborhood}$ $H_{d}(I_{m(n)})$ does not intersect $P_{G}$. We may
also assume that $I_{m(n)}$ does not contain any critical point of
$G$. Since otherwise, by Lemma~\ref{min-arc} and the minimal
property of $m(n)$, it follows that the angle between the straight
segment $[c, z_{m(n)}]$ and $\partial \Delta$ has a uniform positive
lower bound, and this will imply the lemma. Take $n \ge N$. Let
$$
M = \min\{k\ge 1\:\big{|}\: \exists\:\: c \in \Omega_{G} \cap
\partial \Delta \hbox{ such that } G^{k}(c) \in I_{m(n)}\}.
$$
For $0 \le l \le M$, let $J_{l}\subset \partial \Delta $ be the arc
segment  such that
$$
G^{l}(J_{l}) = I_{m(n)}.
$$

Since $\theta$ is of bounded type, there exists a number $K$ which
depends only on $\theta$ and $|X_{G}|$ such that the number of the
intervals $J_{l}$, $1 \le l \le M$, which intersect $X_{G}$ is not
more than $K$(Note that $M$ can be arbitrarily large if $I_{m(n)}$
is small). Moreover, by taking $n$ large enough and thus $I_{m(n)}$
is small, we may also assume that each $J_{l}$, $0 \le l \le M$,
contains at most one of the points in $X_{G}$. For $0 \le l \le
M-1$, let $V_{l}$ denote the pull back of $H_{d}(I_{m(n)})$ by
$G^{l}$ along the orbit $\{J_{l}\}_{0 \le l \le M-1}$. It follows
that for any $d' > d$, there is an $\eta> 0$, such that
$$V_{l} \subset H_{d'}(J_{l})$$ for all $0 \le l \le M-1$ provided
that $|I_{m(n)}| < \eta$. This is because the number of $J_{l}$, $0
\le l \le M$, which intersect $X_{G}$, is not more than $K$, and
thus we need only apply Lemma~\ref{sch-a} at most $K$ times, and for
all other $J_{l}$, we can apply Lemma~\ref{sch}, in which case $d$
is not increased. From now on, let us fix a $d' > d$ and suppose
that $n$ is large enough such that $|I_{m(n)}| < \eta$.
\begin{figure}
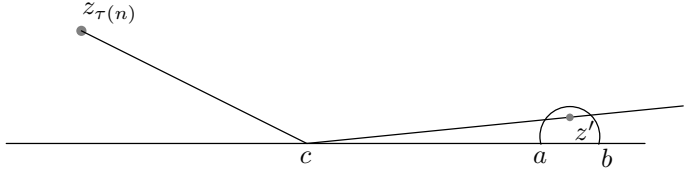

\bigskip
\begin{center}
\centertexdraw { \drawdim cm \linewd 0.02 \move(0 1)

\move(-4 -4) \lvec(4.5 -4) \move(0 -4) \lvec(5 -3.5) \move(3.55
-3.95)\htext{$z'$} \move(3.5 -3.65) \fcir f:0.5 r:0.05

\move(3.5 -3.9) \larc r:0.4 sd:-15 ed:195 \move(0 -4) \lvec(-3 -2.5)
\fcir f: 0.5 r:0.07  \move(-3 -2.4) \htext{$z_{\tau(n)}$} \move(-0.1
-4.25) \htext{$c$}

\move(3 -4.25) \htext{$a$}\move(3.9 -4.3) \htext{$b$}

}

\end{center}
\vspace{0.2cm} \caption{$|a - b| \ll |c - b|$}
\end{figure}

Now there are two cases.  In the first  case,  there is some $1 \le
k(n)\le M-1$ such that
$$
z_{m(n) - l} \in H_{d'}(J_{l})
$$ holds for all $0 \le l \le k(n)-1$,  but
$$
z_{m(n) - k(n)} \notin H_{d'}(J_{k(n)}).
$$

In the second case,
$$
z_{m(n) - l} \in H_{d'}(J_{l})
$$
for all $0 \le l \le M-1$.

In the first case, let $J_{k(n)} = [a, b]$ where $b$ is such that
$|a - c| < |b - c|$.  Let $z' \in H_{d'}(J_{k(n)})$ be such that
\begin{equation}\label{bro}
G(z') = G(z_{m(n) - k(n)}) = z_{m(n) - k(n) + 1}\in
H_{d'}(J_{k(n)-1}).
\end{equation}
Since $$\sigma(J_{k(n) -1}) = \sigma(I_{m(n)}) \hbox{ and }z_{m(n) -
k(n) +1} \in H_{d'}(J_{k(n)- 1}),$$ it follows that when $n$ is
large, $z_{m(n) - k(n) + 1}$ is near $\partial \Delta$ and thus
$m(n) - k(n) + 1$ is large. In particular, $z_{m(n) - k(n)}$ is near
$\partial \Delta$. Since the restriction of $G$ on $\partial \Delta$
is a homeomorphism, it follows from (\ref{bro})that there is some $c
\in \Omega_{G} \cap
\partial \Delta$ such that both $z'$ and $z_{m(n) - k(n)}$ belong to
a small neighborhood of $c$. Let $\tau(n) = m(n) - k(n)$. The first
case is now separated into two subcases (i) and (ii).

In  subcase (i), $|a - b|$ is small compared with $|b - c|$. Then
the angle between the straight segment $[c, z']$ and the unit circle
is small. It follows that the angle between the  straight segment
$[c, z_{\tau(n)}]$ and the unit circle can not be small(in this
case, it is at least about $\pi/(2m+1)$ where $2m+1 \ge 3$ is the
degree of $G$ at $c$, see Figure 12).

\begin{figure}
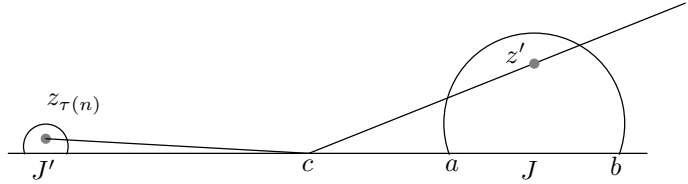

\bigskip
\begin{center}
\centertexdraw { \drawdim cm \linewd 0.02 \move(0 1)

\move(-4 -4) \lvec(4.5 -4) \move(0 -4) \lvec(5 -2) \move(2.6
-2.8)\htext{$z'$} \move(3 -2.8) \fcir f:0.5 r:0.07

\move(3 -3.6) \larc r:1.2 sd:-20 ed:200 \move(0 -4) \lvec(-3.5 -3.8)
\fcir f: 0.5 r:0.07  \move(-3.5 -3.5) \htext{$z_{\tau(n)}$}
\move(-3.5 -3.9) \larc r:0.3 sd:-20 ed:200

\move(-0.1 -4.25) \htext{$c$}

\move(2.8 -4.35) \htext{$J$} \move(-3.7 -4.35) \htext{$J'$}
\move(1.8 -4.25) \htext{$a$}\move(4 -4.3) \htext{$b$} }

\end{center}
\vspace{0.2cm} \caption{$\sigma(J') \ll \sigma(J)$}
\end{figure}
In subcase (ii), there is a uniform $ 0 < k < 1$ such that $|a - b|
> k|c - b|$. See Figure 13 for an illustration.
Since $G(z)$ is like $G(c) + \mu (z - c)^{2 m +1}$ in $B_{r}(c)$,
where $\mu \ne 0$ is some constant, it follows that when $r$ is
small,
$$
|c - z_{\tau(n)}| \asymp |c - z'|.
$$
Note that if the angle between the straight segment $[c,
z_{\tau(n)}]$ and the unit circle were small, there would be an arc
segment $J' \subset
\partial \Delta$ such that $z_{\tau(n)} \in H_{d}(J')$, and
\begin{equation}\label{Eu-len}
|J'| \ll |c - z_{\tau(n)}| \asymp |c - z'| \preceq |c - b| \preceq
|J|.
\end{equation}
From (\ref{Eu-len}) and  Lemma~\ref{geom}, we have
$$
\sigma(J') \ll  \sigma(J) = \sigma(I_{m(n)})
$$
provided that $|J'|$ is small enough. But this contradicts with the
minimal property of $m(n)$.

In the second case, $J_{M-1}$ contains a critical value $v = G(c)$
with $c \in \Omega_{G} \cap \partial \Delta$,
 and $z_{m(n)-M+1} \in H_{d'}(J_{M-1})$. As in the first case,
 as $n$ is large, $\sigma(J_{M-1}) = \sigma(I_{m(n)})$ is small, and
 therefore, $z_{m(n)-M+1}$ is close to $\partial \Delta$. It follows
 that $m(n) - M +1$ is large provided $n$ is large. Let $\tau(n) =
 m(n) - M$.
Now we claim that the angle between the segment of $[z_{\tau(n)+1},
v]$ and the unit circle $\partial \Delta$ must have a positive lower
bound, which is independent of $n$. In fact, if this were not true,
 by Lemma~\ref{min-arc}, we would have a $J \subset
\partial \Delta$ small such that
$$
z_{\tau(n)+1} \in H_{d}(J)
$$
but
$$
\sigma(J) < \sigma(J_{M-1}) = \sigma(I_{m(n)}).
$$
But this  contradicts with the minimal property of $m(n)$ and the
claim is proved.  From the claim, it follows that  $[z_{\tau(n)},
c]$ and the unit circle has a uniform lower bound $\epsilon> 0$,
which is independent of $n$.

Now we get a sequence of integers $\tau(n)$ such that the angle
between $\partial \Delta$ and $[c, z_{\tau(n)}]$  has a positive
lower bound independent of $n$. Since $\{\tau(n)\}$ is unbounded by
the proof above, we may assume that $\tau(n)$ is an increasing
sequence by taking a subsequence. This completes the proof of the
lemma.

\end{proof}

\subsection{Pull back argument}

For a subset $X \subset \Omega_{*}$, we use
$\hbox{Diam}_{\Omega_{*}}(X)$ to denote the diameter of $X$ with
respect to the hyperbolic metric $d \rho_{*}$ of $\Omega_{*}$.  For
any subset $E$ of the complex plane, we use $\hbox{area}(E)$, and
$\hbox{Diam}(E)$ to denote the area and the diameter of $E$ with
respect to the Euclidean metric respectively. Let $\{\tau(n)\}$ be
the sequence obtained in Lemma~\ref{fin-lem}.

\begin{lemma}\label{pull-b}
There exist $K_{1}, K_{2}, K_{3} >0$ independent of $n$, such that
for every $n$ large enough,  there are open simply connected domains
$C^{n} \subset B^{n} \subset A^{n} \subset   \Omega_{*}$ satisfying
\begin{itemize}
\item [1.] $z_{\tau(n)} \in B^{n}$,
\item [2.] $G(C^{n}) \subset \Delta$,
\item [3.] $\mod(A^{n}- B^{n}) \ge K_{1}$,
\item [4.] $\hbox{area}(C^{n})/\hbox{Diam}(B^{n})^{2} \ge K_{2}$,
\item [5.] $d_{\Omega_{*}}(A^{n})  \le K_{3}$.
\end{itemize}
\end{lemma}

\begin{figure}
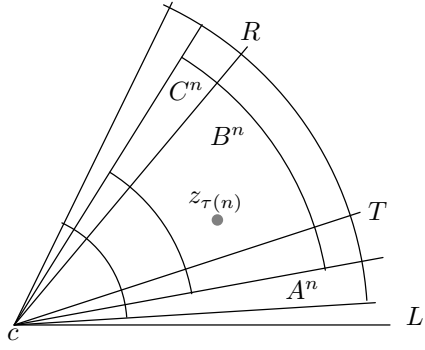

\bigskip
\begin{center}
\centertexdraw { \drawdim cm \linewd 0.02 \move(0 1)

\move(-1 -4) \lvec(4 -4) \move(-1 -4) \lvec(1.1 0.3) \move(-1 -4)
\lvec(3.8 -3.7)

\move(-1 -4)\lvec(3.9 -3.1)

\move(-1 -4) \lvec(3.6 -2.5)

\move(-1 -4) \lvec(1.5 0)

\move(-1 -4) \lvec(2.1 -0.3)

\move(1.3 -2.5)\htext{$z_{\tau(n)}$}

\move(1.7 -2.6) \fcir f:0.5 r:0.08

\move(1.6 -1.6) \htext{$B^{n}$}

\move(1.05 -1) \htext{$C^{n}$}

\move(-1 -4) \larc r:4.7 sd:4 ed:65

\move(-1 -4) \larc r:4.2 sd:10 ed:58

\move(-1 -4) \larc r:2.4 sd:10 ed:58

\move(-1 -4) \larc r:1.5 sd:4 ed:65

\move(2.6 -3.65) \htext{$A^{n}$}

\move(2 -0.2) \htext{$R$} \move(3.7 -2.6) \htext{$T$}

\move(4.2 -4) \htext{$L$}

\move(-1.1 -4.2) \htext{$c$}

}

\end{center}
\vspace{0.2cm} \caption{Construction of $(A_{n}, B_{n}, C_{n},
z_{\tau(n)})$}
\end{figure}

\begin{proof}
 Assume   $d(z_{\tau(n)}, \partial \Delta)$ is  small
enough.  Let $2m+1$ be the local degree of $G$ at $c$ where $m\ge 1$
is some integer. As in the proof of Lemma~\ref{fin-lem}, for $r>0$
small, there are $m+1$ domains which are attached to $c$ and
contained in $B_{r}(c)$ and which are mapped into the outside of the
unit disk. There are two of such domains which are tangent with the
unit disk at $c$. To fix the discussions, let us assume that
$z_{\tau(n)}$ lies in one of these two domains, say $U$. All the
other cases can be treated in the same way.  We also know that there
are $m$ domains, which are contained in $B_{r}(c)$ and attached to
$c$ from the outside of the unit disk, and which are mapped into the
inside of the unit disk. Let $V$ be one of these domains such that
$V$ is adjacent to $U$.  Let $L$ and $R$ be the two half rays which
are tangent with $U$ at $c$. In a small neighborhood of
$z_{\tau(n)}$, $\partial U$ is approximately the union of  two
straight segments starting from $c$ and which lie on $R$ and $L$,
respectively. To simplify the notation, we still use $R$ and $L$ to
denote them. Suppose that the angle between $R$ and $L$ is $\alpha
\pi$. Let $T$ be the straight segment between $R$ and $L$ and which
is on the boundary of $\Omega_{\epsilon,r}^{c}$(see Figure 14).  By
assumption, the angle between $T$ and $L$ is $\epsilon \pi$ where
$0< \epsilon < \alpha < \frac{1}{2}$. For convenience, we use the
polar coordinate system formed by $(c, L)$. by Lemma~\ref{fin-lem},
$z_{\tau(n)} \in \Omega_{\epsilon, r}^{c}$, therefore, we have
$$
z_{\tau(n)} = r_{0} e^{\lambda \pi}.
$$
for some $\epsilon < \lambda < \alpha$ and $0< r_{0} < r$. Now let
$A^{n}$ be the region bounded by
$$
\frac{1}{4}\epsilon \pi \le \theta \le (\alpha+2\epsilon)\pi,
$$ and
$$
r_{0}/2 \le r \le 3r_{0}/2.
$$

Let $B^{n}$ be the region bounded by
$$
\frac{1}{2}\epsilon \pi \le \theta \le (\alpha+\epsilon)\pi,
$$ and
$$
3r_{0}/4 \le r \le 5r_{0}/4.
$$
Let $C^{n} = B \cap V$.
It is not difficult to check that for the domains defined above,
there are constants $K_{i}>0, 1 \le i \le 3$ such that the
conditions in the Lemma are all satisfied. We leave the details to
the reader.
\end{proof}

Let us prove Theorem B now. By taking $n$ large enough, we may
assume that $A^{n} \cap P_{G} = \emptyset$. Now let us  consider the
pull back of $(A^{n}, B^{n}, C^{n}, z_{\tau(n)})$ along the orbit
$\{z_{k}\}$. For $0 \le l < \tau(n)$, let us  denote the connected
component of $G^{l-\tau(n)}(A^{n})$ containing $z_{l}$ by
$A^{n}_{l}$. Then $A_{0}^{n}$ is the connected component of
$G^{-\tau(n)}(A^{n})$ which  contains $z_{0}$, and $A_{\tau(k)}^{n}$
is the connected component of $G^{\tau(k)-\tau(n)}(A^{n})$ which
contains $z_{\tau(k)}$ for $1  \le k < n$. We use $B_{0}^{n}$ and
$C_{0}^{n}$ to denote the subdomains of $A_{0}^{n}$ which are the
pull backs of $B^{n}$ and $C^{n}$ by $G^{-\tau(n)}$. It follows that
$C_{0}^{n} \subset B_{0}^{n} \subset A_{0}^{n}$.

Since $G^{-1}$ contracts the hyperbolic metric  in $\Omega_{*}$,  we
have for all $0\le l < \tau(n)$,
\begin{equation}\label{f-1}
Diam_{\Omega_{*}}( A^{n}_{l}) \le K_{3}
\end{equation}
where $K_{3}$ is the constant in (5) of Lemma 3.12.

By Lemma~\ref{fin-lem}, there is an $N_{0} >0$ such that when $k >
N_{0}$, $z_{\tau(k)} \in \Omega_{\epsilon, r}^{c}$ for some $c \in
\Omega_{G}\cap \partial \Delta$. Since $z_{k} \to \partial \Delta$
and $\tau(k) \to \infty$, by (\ref{f-1}), there is an $N_{1}$ and an
$0 < \eta < 1$ such that  for all $k \ge N_{1}$,
\begin{equation}\label{s-2}
A^{n}_{\tau(k)} \subset \Omega_{\eta \epsilon, r}^{c}.
\end{equation}
From Lemma~\ref{contra} and (\ref{s-2}), it follows that  there is a
$\delta
>0$ independent of $n$ such that for every $k$ with
$\max\{N_{0},N_{1}\} \le k\le n$,
\begin{equation}\label{t-3}
\hspace{1.5cm}{Diam_{\Omega_{*}}(A^{n}_{\tau(k)}) \le (1-\delta)
Diam_{\Omega_{*}}(A^{n}_{\tau(k) + 1}).}
\end{equation}
Since $\{\tau(k)\}$ is an infinite sequence, by (\ref{t-3}), it
follows that as $n \to \infty$, $Diam (A_{0}^{n}) \to 0$ and hence
$Diam(B_{0}^{n}) \to 0$ as $n\to \infty$. On the other hand, by (3),
(4) of Lemma~\ref{pull-b} and Koebe's distortion theorem, we get a
constant $0< C < \infty$ such that for all $n$ large enough,  the
following conditions hold:
\begin{itemize}
\item[1.] $z_{0} \in B_{0}^{n}$, and
\item[2.] $C^{n}_{0} \subset B^{n}_{0}$, and
\item[3.] $area C^{n}_{0} \ge C diam(B^{n}_{0})^{2}$
\end{itemize}

By (2) of Lemma~\ref{pull-b}, $C^{n}_{0} \subset
\bigcup_{k=0}^{\infty}G^{-k}(\overline{\Delta})$. This implies that
$z_{0}$ is not a Lebesgue point of $J_{\widehat{G}}$, which is a
contradiction. The proof of the zero measure statement of Theorem B
is completed.

Now let us prove the rigidity statement of Theorem B.
\begin{lemma}
Let $f \in R_{\theta}^{top}$ and suppose that $f$ has no Thurston
obstructions outside the $\emph{rotation disk}$, and is $realized$
by two maps $g, h \in R_{\theta}^{geom}$.  Then there exist two
quasiconformal homeomorphisms of the sphere $\phi_{1}$ and
$\phi_{2}$ such that
\begin{itemize}
\item[1.] $\phi_{1}$ and  $\phi_{2}$ are
$\emph{combinatorially equivalent}$ to each other rel $P_{g}$, and
\item [2.] $\phi_{1}|D_{g}  = \phi_{2}|D_{g}$ are holomorphic
on the Siegel disk, and
\item[3.] For each super-attracting periodic point $x$ of $g$,
there is a neighborhood of $x$, say $U_{x}$, such that
$\phi_{1}|U_{x} = \phi_{2}|U_{x}$ are holomorphic, and
\item [4.] $g = \phi^{-1}_{1} \circ h \circ \phi_{2}$.
\end{itemize}
\end{lemma}
The proof is easy and we leave the details to the reader.

Now for $k \ge 2$, since $g$ and $h$ are $\emph{combinatorially
equivalent}$, we can lift $\phi_{k}$ by  the equation
$$
g =  \phi_{k}^{-1} \circ h \circ \phi_{k+1}.
$$
and get $\phi_{k+1}$. In this way we get a sequence of
quasiconformal homeomorphisms $\{\phi_{k}\}$ of the sphere such that
$\phi_{k}|P_{g} = \phi_{k+1}|P_{g}$ for all $k \ge 1$. Let $\mu_{k}$
be the dilation of $\phi_{k}$. Since both $g$ and $h$ are rational
maps, it follows that $\|\mu_{k}\|_{\infty} < K < 1$ where $K$ is
some constant independent of $k$. Since any periodic Fatou component
of $g$ must either be the Siegel disk, or a super-attracting
periodic Fatou component, it follows that $\mu_{k} \to 0$ on the
Fatou set of $g$. Since the Julia set of $g$ has zero Lebesgue
measure, and $\phi_{k}|P_{g} = \phi_{k+1}|P_{g}$ for all $k \ge 1$,
it follows that $\phi_{k}$ converges to the same M\"{o}bius map.  We
complete the proof of Theorem B.


\begin{center}
\section{Quadratic Rational Maps with Bounded Type Siegel Disks}
\end{center}

\subsection{Quadratic Siegel Rational Maps}

Let $g$ be a quadratic rational map which has a $\emph{bounded
type}$ Siegel disk. Up to a M\"{o}bius conjugation, we may assume
that the center of the Siegel disk is at the origin and $g(\infty) =
\infty$. Then $g$ has the following normalized form,
\begin{equation}\label{norml}
g(z) = \frac{az^{2} + e ^{2 \pi  i \theta }z } {bz + 1}
\end{equation}

From Riemann-Huiwitz formula, it follows that any quadratic rational
map has exactly two distinct critical points. Through a M\"{o}bius
conjugation, we may   further assume that $1$ is one of the critical
points of $g$, that is, $g'(1) = 0$.   By a simple calculation, this
is equivalent to
\begin{equation}
a(b+2) +e^{2 \pi i \theta} = 0
\end{equation}
Let us denote the other critical point of $g$, which is different
from $1$,  by $c_{g}$.

\begin{lemma}
Let $\Sigma$ be the space of all the normalized quadratic Siegel
rational maps $g$ such that $g(0) = 0$, $g(\infty) = \infty$, and
$g'(1) = 0$. Then the map $\rho: g \to c_{g}$ is a homeomorphism
between $\Sigma$ and $\widehat{\Bbb C}-\{0, 1, -1\}$.
\end{lemma}

\begin{proof}
 Since $g'(0) = e^{2 \pi i\theta}$ and the two critical points of $g$
 must be distinct from each other, it follows that $c_{g} \ne 0 , 1$.
 By a simple calculation, we get
\begin{equation}\label{g-dir}
g'(z) = \frac{abz^{2} + 2az + e^{2 \pi i \theta}}{(bz+1)^{2}}.
\end{equation}

From (\ref{g-dir}) and $g'(1) = g'(c_{g}) = 0$, it follows that $1$
and $c_{g}$ are the two roots of the quadratic polynomial equation
\begin{equation}\label{w-d}
abz^{2} + 2az + e^{2 \pi i \theta} = 0.
\end{equation}
This implies that $c_{g} \ne -1$. In  fact, if $c_{g} = -1$, we have
$2/b  = -(1 + c_{g}) = 0$, and hence $b = \infty$.  This is a
contradiction.

Now for $c_{g} \ne 0, 1$, and $-1$, we can solve $a = -e^{2 \pi i
\theta}(1+ c_{g})/2c_{g}$ and $b = -2/(1+ c_{g})$. Therefore, $g$ is
uniquely determined by $c_{g}$,  and we have
\begin{equation}\label{g-c}
g(z) = \frac{-e^{2 \pi i \theta}(1 +c_{g})^{2} z^{2} + 2 e^{2 \pi i
\theta}c_{g}(1 + c_{g})z}{-4c_{g}z + 2c_{g}(1+ c_{g})}.
\end{equation}
In particular, as $c_{g} \to \infty$, $a \to  -e^{2 \pi i
\theta}/2$, $b \to 0$ and hence $g(z) \to -e^{2 \pi i \theta}
z^{2}/2 + e^{2 \pi i \theta} z  = g_{\infty}(z) \in \Sigma$. The
lemma follows.

\end{proof}

Now for each $c \in \widehat{\Bbb C} - \{0, 1, -1\}$,  we use
$g_{c}$ to denote the normalized quadratic Siegel rational map which
has $1$ and $c$ as its critical points.

\begin{figure}
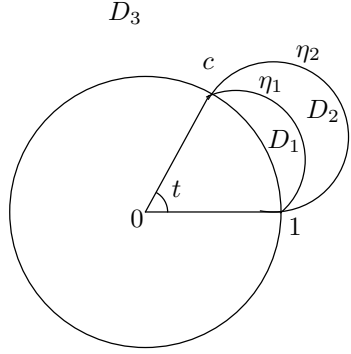

\bigskip
\begin{center}
\centertexdraw { \drawdim cm \linewd 0.02 \move(0 1)

\move(0 -3) \lcir r:1.8 \lvec(1.8 -3) \move(0 -3) \lvec(0.85 -1.45)

\move(1.2 -2.3) \larc r:0.92 sd:-50 ed:115

\move(0.75 -1.1) \htext{$c$}

 \move(1.7 -2) \larc r:1 sd:-100 ed:150

\move(2 -1) \htext{$\eta_{2}$}

\move(1.62 -2.2) \htext{$D_{1}$}

\move(2.13 -1.8) \htext{$D_{2}$}

\move(1.5 -1.42) \htext{$\eta_{1}$}

\move(-0.5  -0.5) \htext{$D_{3}$}

\move(-0.2 -3.2) \htext{$0$}

\move(1.9  -3.3) \htext{$1$}

\move(0 -3) \larc r:0.3 sd:0 ed:60 \move(0.35 -2.8) \htext{$t$}

 }
\end{center}
\vspace{0.2cm} \caption{Combinatorics of $f_{t} \in
R_{\theta}^{geom}$ for $0 < t < 2 \pi$}
\end{figure}

\subsection{Branched Covering Maps $f_{t} \in R_{\theta}^{top}$}
\subsubsection{Branched covering maps $f_{t}$ and Siegel rational
maps $g_{c(t)}$}
In this section, we will construct a family of topological branched
covering maps $f_{t} \in R_{\theta}^{top}, 0 < t < 2 \pi$. This
family of topological branched covering maps will provide models of
a continuous family of quadratic Siegel rational maps $g_{c(t)} \in
R_{\theta}^{geom}$, $0 < t < 2 \pi$, where $c(t), 0< t < 2\pi$ is a
continuous curve in the critical parameter plane. Later we will see
that this curve plays a fundamental role in the proof of Theorem C.

\begin{df}
For each $0< t < 2\pi$, let $c \in \partial \Delta$ such that the
angle spanned by $1$ and $c$ is $t$. Let $\eta_{1}$, $\eta_{2}$ be
two curve segments connecting $1$ and $c$ as indicated in Figure 15.
Let $D_{1}$ be the domain bounded by $\eta_{1}$ and the arc from $1$
to $c$, anticlockwise, and $D_{2}$ denote the domain bounded by
$\eta_{1}$ and $\eta_{2}$. Let $D_{3}$ denote the domain which
contains the infinity and which is bounded by $\eta_{2}$ and the arc
from $c$ to $1$, anticlockwise. Let $f_{t} \in R_{\theta}^{top}$ be
a topological branched covering map defined as follows:
$(f_{t}|\Delta)(z)  =  e^{2 \pi i\theta}z$, and $f_{t}: D_{1} \to
S^{2} - \Delta, D_{2} \to \Delta,  D_{3} \to S^{2} - \Delta$ are all
homeomorphisms. Moreover,  $f_{t}(\infty) = \infty$.
\end{df}

\begin{dff}
For each $0< t < 2\pi$, let $c \in \partial \Delta$ such that the
angle spanned by $1$ and $c$ is $t$. Let $\eta_{1}$, $\eta_{2}$ be
two curve segments connecting $1$ and $c$ as indicated in Figure 16.
Let $D_{1}$ denote the domain bounded by $\eta_{1}$ and the arc from
$1$ to $c$ anticlockwise. Let $D_{2}$ denote the domain bounded by
$\eta_{1}$ and $\eta_{2}$. Let $D_{3}$ denote the domain bounded by
$\eta_{2}$ and the arc from $c$ to $1$ anticlockwise. Define
$\tilde{f}_{2 \pi - t}$ as follows: $(\tilde{f}_{2\pi - t}|\Delta(z)
=   e^{2 \pi i\theta}z$, and $\tilde{f}_{2\pi - t}: D_{2} \to
\Delta, D_{1} \to S^{2} - \Delta,  D_{3} \to S^{2} - \Delta$ are all
homeomorphisms. Moreover, $\tilde{f}_{2 \pi - t}(\infty) = \infty$.
\end{dff}
\begin{figure}
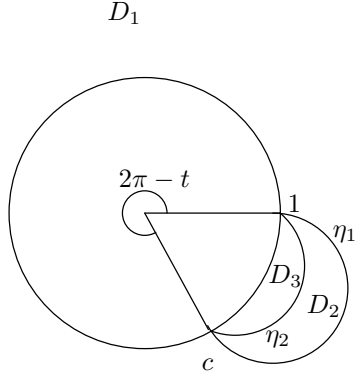

\bigskip
\begin{center}
\centertexdraw { \drawdim cm \linewd 0.02 \move(0 1)

\move(0.75 -5.1) \htext{$c$}

 \move(0 -3) \lcir r:1.8 \lvec(1.8 -3)
\move(0 -3) \lvec(0.85 -4.55)

\move(1.2 -3.7) \larc r:0.92 sd:-110 ed:50

\move(1.9  -3) \htext{$1$} \move(1.7 -4) \larc r:1 sd:-150 ed:90

\move(1.6 -4.8) \htext{$\eta_{2}$}

\move(1.64 -4) \htext{$D_{3}$}

\move(2.15 -4.4) \htext{$D_{2}$}

\move(2.5 -3.4) \htext{$\eta_{1}$}

\move(-0.5  -0.5) \htext{$D_{1}$} \move(-0.35 -2.7) \htext{$2\pi -
t$}

\move(0 -3) \larc r:0.3 sd:0 ed:300

 }
\end{center}
\vspace{0.2cm} \caption{Combinatorics of $\widetilde{f}_{2\pi -t}
\in R_{\theta}^{geom}$ for $0 < t < 2 \pi$}
\end{figure}

Since any simple closed curve $\gamma \subset S^{2} - \Delta$ is
$\emph{peripheral}$, it follows that $f_{t}$($\widetilde{f}_{t}$)
has no Thurston obstructions outside the $\emph{rotation disk}$
$\Delta$ for all $0< t < 2\pi$. By Theorem A and Theorem B, we have
\begin{lemma}\label{rel-t}
For each $0< t< 2\pi$, there is a unique $c(t)(\widetilde{c}(t)) \in
{\Bbb C}-\{0, 1, -1\}$ such that $g_{c(t)}(g_{\widetilde{c}(t)})$
realizes $f_{t}(\widetilde{f}_{t})$ in the sense that $$f_{t} =
\phi^{-1} \circ g_{c(t)} \circ \psi(\widetilde{f}_{t} = \phi^{-1}
\circ g_{\widetilde{c}(t)} \circ \psi)$$ where $\phi$ and $\psi:
S^{2} \to S^{2}$ are homeomorphisms which fix $0$, $1$, and the
infinity, and are isotopic to each other rel $P_{f}$.
\end{lemma}

\noindent $\bold{Inner \:angle\:between \: the
\:two\:critical\:points.\:}$Let $g_{c} \in R_{\theta}^{geom}$ and
$D$ be the Siegel disk of $g_{c}$ centered at the origin such that
$\partial D$ passes through both the two critical points $1$ and
$c$.   Let $\phi: D \to \Delta$ be the holomorphic map which
conjugates $g_{c}|D$ to the rigid rotation $R_{\theta}$ on $\Delta$.
Since $\partial D$ is a quasi-circle, it follows that $\phi$ can be
homeomorphically extended to $\partial D \to \partial \Delta$.   We
use $A_{c}$ to denote the angle from $\phi(1)$ to $\phi(c)$
anticlockwise. We call it the $\emph{inner angle}$ between $1$ and
$c$.

\begin{remark}\label{exp}
Let $g_{c} \in R_{\theta}^{geom}$ such that both of the critical
points of $g_{c}$ are on the boundary of the Siegel disk. Suppose
that the inner angle between $1$ and $c$ is $t$. Let $D_{c}$ be the
Siegel disk of $g_{c}$.  Then  $g_{c}$ realizes $f_{t}$ for some $0<
t< 2\pi$ in the sense of Lemma~\ref{rel-t}, if and only if the
boundary of the bounded component of $g_{c}^{-1}(S^{2} -
\overline{D_{c}})$ contains the part of the boundary  of $D_{c}$,
which connects $1$ to $c$ anticlockwise. By contrary, $g_{c}$
realizes $\tilde{f}_{t}$ for some $0< t< 2\pi$ in the sense of
Lemma~\ref{rel-t}, if and only if  the boundary of the bounded
component of $g_{c}^{-1}(S^{2} - \overline{D_{c}})$ contains the
boundary arc of the Siegel disk, which connects  $1$ to $c$
clockwise.
\end{remark}

\subsubsection{Some basic facts about $f_{t}$}
It is useful to find the M\"{o}bius transformations which conjugate
a normalized quadratic Siegel rational map to an another normalized
one. Let $g_{c}$ be a normalized quadratic Siegel rational map given
by (\ref{norml}). There are two cases.

In the first case, $g_{c}$ has exactly two fixed points $0$ and
$\infty$. By a simple calculation, this is equivalent to that
$c_{g}$ is one of the two roots of the following equation,
$$
c^{2} + (4e^{-2 \pi i \theta} + 2)c + 1 = 0.
$$
It follows that in this case, there are exactly two normalized
quadratic Siegel rational maps which have exactly two fixed points
$0$ and $\infty$, and which are conjugate to each other by $z \to z
/c_{g}$.

In the second case, $g_{c}$ has exactly three distinct fixed points
$0, \infty$, and some complex value $p$. Let $\phi$ be a M\"{o}bius
transformation such that $\phi \circ g_{c} \circ \phi^{-1}$ has the
normalized form.   Then $\phi$ is determined by one of the following
four conditions,
\begin{itemize}
\item [1.] $\phi = id$.
\item [2.] $\phi(0) = 0$, $\phi(1) = 1$, and $\phi(p)= \infty$,
\item [3.] $\phi(z)  =  z/c_{g}$,
\item [4.] $\phi(0) = 0$, $\phi(c_{g}) = 1$, and $\phi(p) = \infty$.
\end{itemize}

Let us collect some basic facts about the topological branched
covering maps $f_{t}$,  $0< t < 2 \pi$,  which can be easily seen
from Figure 15 and 16. The rigorous proofs of these facts are not
difficult and shall be left to the reader.

Fact 1. Let $0 < t < 2\pi$. Suppose that  $g_{c(t)} \in
R_{\theta}^{geom}$ $\emph{realizes}$ $f_{t}$.   Then $g_{c(t)}$ has
exactly  three distinct fixed points, $0, \infty$, and $p$.

Fact 2. Let $g_{\tilde{c}(t)} \in R_{\theta}^{geom}$
$\emph{realizes}$ the topological branched covering map
$\tilde{f}_{t}$ for $0< t< 2\pi$.  Then $g_{\tilde{c}(t)}$ is
conjugate to $g_{c(t)}$ by the M\"{o}bius map determined by (2)
above.

Fact 3. By just exchanging the positions of $1$ and $c$ in Figure
15, with all the other topological data being fixed, we will get an
another new topological branched covering map in $R_{\theta}^{top}$
indicated by Figure 16. This new topological branched covering map
models the Siegel rational map $g_{\tilde{c}(2 \pi -t)} \in
R_{\theta}^{geom}$. It is clear that the maps $g_{\tilde{c}(2 \pi
-t)}$ and $g_{c(t)}$ are conjugate to each other by the M\"{o}bius
map determined by the condition (3).

Fact 4. If we compose the last two conjugations in either order(that
is, we may first change $f_{t}$ to $\tilde{f}_{t}$, and then
exchange the positions of $1$ and $c$ in $\tilde{f}_{t}$ and finally
get $f_{2 \pi - t}$, or we first exchange the positions of $1$ and
$c$ in $f_{t}$ and get $\tilde{f}_{2 \pi - t}$ and then change it to
$f_{2 \pi - t}$), we will get the same topological branched covering
map $f_{2 \pi - t} \in R_{\theta}^{top}$ which models the Siegel
rational map $g_{c(2\pi - t)} \in R_{\theta}^{geom}$. It follows
that  $g_{c(2 \pi -t)}$ is conjugate to $g_{c(t)}$ by the M\"{o}bius
map determined by the condition (4).


\subsection{A Distortion Lemma}
For each $0< t < 2\pi$, suppose that $f_{t}$ is $\emph{realized}$ by
a Siegel rational map  $g_{c(t)} \in R_{\theta}^{geom}$. By
Lemma~\ref{3-fir}, there is a Blaschke product, say $G_{t}$, which
$\emph{models}$ $g_{c(t)}$.  The main purpose of this section is to
prove
\begin{lemma}\label{distortion}
There is a constant $1 < K < \infty$ which depends only on $\theta$
such that  for every Siegel rational map in $R_{\theta}^{geom}$
which is modeled by $f_{t}$ for some $0< t < 2 \pi$, the boundary of
the Siegel disk is a $K-$quasi-circle.
\end{lemma}

In the procedure of the quasiconformal surgery in $\S2.5.2$, if we
 just take $H$ to be the Douady-Earle extension of $h$ and do
not require that $H(0) = 0$, then by the conformal natural property
of Douady-Earle extension, we can reduce Lemma~\ref{distortion} to
the following lemma. For $0< t < 2\pi$, let $h_{t}: \partial \Delta
\to
\partial \Delta$ be the quasisymmetric homeomorphism such that
$h_{t}(1) = 1$ and $$G_{t}|\partial \Delta = h_{t} \circ R_{\theta}
\circ h_{t}^{-1}.$$
\begin{lemma}\label{reduced}
There is a uniform $1 <K < \infty$, such that for every $0< t < 2
\pi$, there is a M\"{o}bius map $\sigma_{t}$ which fixes $1$ and
 maps the unit circle to itself with orientation preserved,
such that the map $\sigma_{t} \circ h_{t}$ is a $K-$quasisymmetric
homeomorphism.
\end{lemma}

\begin{remark}
Let $d \ge 3$ be an integer and $0< \theta < 1$ be a bounded type
irrational number. Let $B_{d}^{\theta}$ denote the family of all the
Blaschke products such that the restriction of every $B \in
B_{d}^{\theta}$ to the unit circle is a critical circle
homeomorphism of rotation number $\theta$.  By using Buff-Cheritat's
Relative Schwartz lemma, it was recently proved that the above bound
$K$ actually exists for all the maps in $B_{d}^{\theta}$  and $K$
depends only on $\theta$ and $d$\cite{Zh3}.
\end{remark}

\begin{sub1}
There exist $0< \delta_{0} < 2 \pi$ and $0< \epsilon_{0} < 2 \pi$
such that for any $0< t < 2\pi$, there exist four distinct points
$x_{1}, x_{2}, x_{3}, x_{4} \in \partial \Delta$ and a M\"{o}bius
map $\sigma_{t}$ which maps the unit circle to itself and preserves
the orientation, such that the arc length of each component of
$\partial \Delta - \{x_{1}, x_{2},x_{3}, x_{4}\}$ is $\ge
\delta_{0}$,  and the arc length of each component of $\partial
\Delta  - \{\tau_{t}^{-1}(x_{1}), \tau^{-1}_{t}(x_{2}),
\tau^{-1}_{t}(x_{3}), \tau^{-1}_{t}(x_{4})\}$ is  $\ge
\epsilon_{0}$, where $\tau_{t} =  \sigma_{t} \circ h_{t}$.
\end{sub1}

\begin{proof}
Take $l$ large enough such that $\{l \theta\} < \pi/2$, where
$\{\cdot\}$ is used to denote the fraction part of a number. Let $I$
be an arc segment with minimal arc length such that $|h_{t}^{-1}(I)|
= \{l \theta\}$. Let $L$ and $R$ be the two adjacent arc segments of
$I$ on $\Bbb T$ such that $$|h_{t}^{-1}(L)| = |h_{t}^{-1}(R)| = \{l
\theta\}.$$ We now claim that there exists an $1 < M < \infty$ which
does not depend on $t$ such that one of the following two
inequalities hold:
$$
|L| \le M |I| \hbox{ or } |R| \le M |I|.
$$
Let us prove the claim now. Assume that it is not true. Then there
is a sequence $t_{n} \in (0, 2 \pi)$ such that for each $n$, there
exist three adjacent intervals $L_{n}, I_{n}$ and $R_{n}$ in $\Bbb
T$ so that
\begin{equation}\label{ass-equ}
|h_{t_{n}}^{-1}(L_{n})|  = |h_{t_{n}}^{-1}(I_{n})| =
|h_{t_{n}}^{-1}(R_{n})| = \{l \theta\},
\end{equation}
but both of the above two inequalities do not hold. By passing to a
subsequence, we may assume that $|R_{n}|/|I_{n}| \to \infty$ and
$|L_{n}|/|I_{n}| \to \infty$. Take $n$ large enough. Let
$\Pi_{t_{n}}$ be the set of the critical points of $G_{t_{n}}$. Let
$$
X_{n} =  \widehat{\Bbb C}-\big{(} (\partial \Delta  - (R_{n} \cup
L_{n})) \cup \bigcup_{1 \le i \le
l}G_{t_{n}}^{i}(\Pi_{t_{n}})\big{)},
$$
and
$$
Y_{n} = G_{t_{n}}^{-l}(X_{n}).
$$
It follows that $$G_{t_{n}}^{l}: Y_{n} \to X_{n}$$ is a holomorphic
covering map.

Since $I_{n}$ has a large space around it in $L_{n}\cup I_{n} \cup
R_{n}$, it follows that there is a short simple closed geodesic
$\gamma_{n} \subset X_{n}$ which separates $I_{n}$ and $\partial
\Delta - L_{n}\cup I_{n} \cup R_{n}$. We thus get that
$\|\gamma_{n}\|_{X_{n}} \to 0$ as $n \to \infty$. Let $\xi_{n}$
denote the component of $G_{t_{n}}^{-l}(\gamma_{n})$ which
intersects the unit circle. It follows that $\xi_{n}$ is also a
short simple closed geodesic, which is symmetric about the unit
circle. Moreover, $\|\xi_{n}\|_{Y_{n}} \to 0$ as $n \to \infty$.
Most importantly, by (\ref{ass-equ}), it follows that
$$
G_{t_{n}}^{l}(I_{n}) = L_{n}
$$
and therefore  the geodesic $\xi_{n}$ separates $L_{n}$ and $R_{n}$.
But since $|R_{n}|/|I_{n}| \to \infty$ and $|L_{n}|/|I_{n}| \to
\infty$, it follows that  the length of any simple closed geodesic
which separates $L_{n}$ and $R_{n}$ has a positive lower bound. This
is a contradiction and the claim has been proved.

Now we may assume that $|L| \le M |I|$(the case that $|R|< M|I|$ can
be treated in the same way).  Let $$S = \partial \Delta  - L\cup I
\cup R.$$ By the choice of $l$ and $I$, it follows that
$|h_{t}^{-1}(S)|
> \{l\theta\}$ and hence $|S| > |I|$. Let $z \in \Delta$ be the
point which lies in the straight line which passes through the
origin and the middle point of $I$ such that $d(z, I) = |I|$. Define
the M\"{o}bius map $\sigma_{t}$ such that $\sigma_{t}(1) = 1$,
$\sigma_{t}(z) = 0$ and $\sigma_{t}({\Bbb T}) = {\Bbb T}$.  Let
$t_{1}, t_{2}, t_{3}$ and $t_{4}$ be the end points of the interval
of $L$, $I$ and $R$. Let $x_{1}, x_{2}, x_{3}$ and $x_{4}$ be the
images of $t_{1}, t_{2}, t_{3}$ and $t_{4}$ under the map
$\sigma_{t}$. It follows that there is a uniform $\delta_{0}> 0$
such that each component of $\partial \Delta - \{x_{1}, x_{2},
x_{3}, x_{4}\}$ has arc length $\ge \delta_{0}$(To get this, one can
consider the cross ratio of the four end points of   the intervals
$L$, $I$, $R$, and $S$. Use the fact that $|I| \le |L| \le M |I|,
|I| \le |R|$, and $|I| \le |S|$  and that M\"{o}bius maps preserve
cross ratios). Let $\tau_{t} = \sigma_{t} \circ h_{t}$. Then the arc
length of each component of $$\partial \Delta  -
\{\tau_{t}^{-1}(x_{1}), \tau^{-1}_{t}(x_{2}), \tau^{-1}_{t}(x_{3}),
\tau^{-1}_{t}(x_{4})\}$$ is  $\ge \epsilon_{0} = \{l\theta\}$. The
proof of Sublemma 1 is completed.

\end{proof}


To simplify the notations, in the following we use $G_{t}$ and
$h_{t}$ instead of $\sigma_{t} \circ G_{t} \circ \sigma_{t}^{-1}$
and $\sigma_{t}\circ h_{t}$, and assume that there exist $0<
\delta_{0} < 2 \pi$ and $0< \epsilon_{0} < 2 \pi$ such that for any
$0< t < 2\pi$, there exist four distinct points $x_{1}, x_{2},
x_{3}, x_{4} \in
\partial \Delta$ such that the arc length of each component of
$\partial \Delta - \{x_{1}, x_{2},x_{3}, x_{4}\}$ is $\ge
\delta_{0}$,  and the arc length of each component of $\partial
\Delta  - \{h_{t}^{-1}(x_{1}), h^{-1}_{t}(x_{2}), h^{-1}_{t}(x_{3}),
h^{-1}_{t}(x_{4})\}$ is  $\ge \epsilon_{0}$, where $h_{t}:
\partial \Delta \to \partial \Delta$ is the quasisymmetric
homeomorphism such that $h_{t}(1) = 1$ and $G_{t}|\partial \Delta =
h_{t} \circ R_{\theta} \circ h_{t}^{-1}$.

Let $J \subset I \subset \Bbb T$ such that both the components of $I
- J$, say $R$ and $L$, are  non-trivial arc segments. Define
$$
C(I, J) = \frac{|I||J|}{|R||L|}.
$$
The value $C(I, J)$ measures the space around $J$ in $I$. Let $$X =
\widehat{\Bbb C} - (\partial \Delta - R \cup L).$$ Let
$\gamma\subset X$ be the simple closed geodesic which separates $J$
and $\partial \Delta - I$. The proof of the following lemma is
direct, and we shall leave the details to the reader:
\begin{sub2}
Let $\delta, C > 0$. Then there exists a $\lambda(\delta, C) > 0$
dependent only on $\delta$ and $C$ such that if $|\partial \Delta -
I|> \delta$ and $\|\gamma\|_{X} \le C$, then  $C(I, J) \le
\lambda(\delta, C)$. Moreover,  if $|\partial \Delta - I|> \delta$
and $C(I, J) \le C$, then $\|\gamma\|_{X} \le \lambda(\delta, C)$.
\end{sub2}

\begin{remark}\label{trans-re}
Sublemma 2 implies that the existence of the upper bound of the
length of the simple closed geodesic which separates $J$ and
$\partial \Delta - I$ is equivalent to the existence of some
definite space around $J$ inside $I$ provided that $\partial \Delta
- I$ is not too small.
\end{remark}

Given a collection of arc segments $$\mathcal{I} = \{I^{k} \subset
\partial \Delta, k \in \Lambda\},$$ the $\emph{intersection
multiplicity}$ of $\mathcal{I}$ is defined to be the largest integer
$n \ge 0$ such that there exist $n$ distinct arc segments in
$\mathcal{I}$ whose intersection is not empty.

For an arc segment $I \subset
\partial \Delta$, we use $I^{k}_{t} \subset \partial \Delta$ to
denote the component of $G_{t}^{-k}(I)$ which lies in the unit
circle. In particular, $I_{t}^{0} = I$.

\begin{lemma}\label{chang-zhou}
For each $K>0$, $l \ge 1$ and $\rho > 0$, there is a constant
$\lambda(K, l, \rho)>0$ , which is independent of $t$,  such that
for any arc segments $M \subset T \subset \partial \Delta$, if the
following three conditions are satisfied,
\begin{itemize}\item[1.]
$C(T,M) < K$, \item[2.] the intersection multiplicity of
$\{T^{i}_{t}, i=0,1, \cdots, N\}$ is  less than $l$, \item[3.]
$|\partial \Delta - T_{t}^{i}| > \rho$ for $0 \le i \le N$,
\end{itemize}
then $C(T^{N}_{t}, M^{N}_{t}) < \lambda(K,l, \rho)$.
\end{lemma}

\begin{proof}
 Let $M \subset T \subset \partial \Delta$. Let $c_{t}^{i}, i=1,2$
 be the two critical points and $v_{t}^{i}, i=1,2$  the two critical
 values of $G_{t}$. For a given  $ 0 \le k \le n$, there are two cases.

 In the first case, $T_{t}^{k}$ contains some critical value of
 $G_{t}$. Set
$$
A_{k} = (\partial \Delta - T_{t}^{k})  \cup M_{t}^{k} \cup
(T_{t}^{k} \cap \{ v_{t}^{1}, v_{t}^{2}\}),
$$
and
$$
B_{k} = (\partial \Delta - T_{t}^{k})  \cup M_{t}^{k}.
$$

Now let us consider the following three hyperbolic Riemann surfaces,
\begin{equation}
X_{k} = {\Bbb P}^{1} - A_{k},
\end{equation}

\begin{equation}
Y_{k} = {\Bbb P}^{1} - B_{k},
\end{equation}
and
\begin{equation}
Z_{k} = {\Bbb P}^{1} - G_{t}^{-1}(A_{k}).
\end{equation}

By the assumption that $C(T, M) <K$ and $|\partial \Delta - T|
> \rho$, it follows from Sublemma 2 that there is a simple closed geodesic in
$Y_{0}$ which separates $M$ and $\partial \Delta - T$ whose
hyperbolic length has an upper bound which depends only on $K$.

Since $ Y_{k}-X_{k} \subset \{v_{t}^{1}, v_{t}^{2}\}$ is a finite
set,  it follows  that there is a uniform constant $1 < C < \infty$
such that for the simple  closed geodesic $\xi \subset Y_{k}$, there
is a simple closed geodesic $\xi' \subset X_{k}$ which is homotopy
to $\xi$ in $Y_{k}$, such that $l_{X_{k}}(\xi') < Cl_{Y_{k}}(\xi)$.

Let  $\eta \subset Y_{k}$ be the simple closed geodesic  which
separates $M_{t}^{k}$ and $\partial \Delta - T_{t}^{k}$. Take a
simple closed geodesic $\eta' \subset X_{k}$ such that $\eta'$ is
homotopy to $\eta$ in  $Y_{k}$ and such that $l_{X_{k}}(\eta') <
Cl_{Y_{k}}(\eta)$ where $C > 0$ is the uniform constant above. Let
$\eta''$ be the simple closed geodesic in $Z_{k}$  which separates
$\partial \Delta - T_{t}^{k+1}$ and $M _{t}^{k+1}$  such that the
image of $\eta''$ under $G_{t}$ covers $\eta'$. Since $$G_{t}: Z_{k}
\to X_{k}$$ is a holomorphic covering map of degree $3$,  it follows
that $l_{Z_{k}}(\eta'')  \le 3l_{X_{k}}(\eta')$.  Therefore, we have
\begin{equation}
l_{Y_{k+1}}(\eta'') < l_{Z_{k}}(\eta'') \le 3 l_{X_{k}}(\eta') <
3Cl_{Y_{k}}(\eta).
\end{equation}

In the second case, $T_{t}^{k}$ does not contain any critical value
of $G_{t}$. Let  $\eta \subset Y_{k}$ be a simple closed geodesic
which separates $M_{t}^{k}$ and $\partial \Delta - T_{t}^{k}$. It
follows that  there is a simple closed geodesic $\eta' \subset
Z_{k}$  which separates $\partial \Delta - T_{t}^{k+1}$ and
$M_{t}^{k+1}$  such that the image of $\eta'$ under $G_{t}$ covers
$\eta$ exactly one time. It follows that
\begin{equation}
l_{Y_{k+1}}(\eta') < l_{Z_{k}}(\eta') = l_{Y_{k}}(\eta).
\end{equation}

 Since the $\emph{intersection multiplicity}$ of $\{T_{t}^{k}\}$ is $l$,
and $G_{t}$ has only two critical values,  it follows that, when $k$
runs through $0, 1, \cdots, N-1$, case 1 can happen at most $2l$
times. Therefore, there is a simple closed geodesic which separates
$\partial \Delta - T_{t}^{N}$ and $M_{t}^{N}$ whose length has an
upper bound dependent only on $K$ and $l$. Note that $|\partial
\Delta - T_{t}^{N}| > \rho$. The lemma then follows from Sublemma 2
and Remark~\ref{trans-re}.
\end{proof}


Let $I = [a, b] \subset \partial \Delta$.  We use $|a-b|$ or $|I|$
to denote the Euclidean length of the arc $I$.  For $K > 1$, we say
two intervals $I, J \subset \partial \Delta$ are $K-$comparable if
$K^{-1} < |I|/|J| < K$. Let $p_{n}/q_{n}, n=1, 2, \cdots$ be the
convergents of $\theta$.
\begin{lemma}\label{metric-dis}
There is a constant $K>1 $ which is  only dependent on $\theta$ such
that for all $0 < t < 2 \pi$, $z \in \partial\Delta$  and $n \ge 1$,
the following two inequalities hold,
\begin{equation}\label{first-ine}
1/K \le \frac{|G_{t}^{-q_{n}}(z) - z|}{|G_{t}^{q_{n}}(z) - z|} \le K
\end{equation}
and
\begin{equation}\label{second-ine}
1/K \le \frac{|G_{t}^{q_{n+1}}(z) - z|}{|G_{t}^{q_{n}}(z) - z|} \le
K.
\end{equation}
\end{lemma}
The idea of the proof is taken from $\S3$ of \cite{dFdM}.
\begin{proof}
 Let $M$  be an integer  such that
$$
|h_{t}^{-1}(x) - h_{t}^{-1}(G_{t}^{q_{n}}(x))| < \epsilon_{0}/3
$$
holds for all $n \ge M$  and $0< t < 2 \pi$ where $\epsilon_{0}$ is
the number in Sublemma 1. It is sufficient to prove that there is a
$K > 1$ such that the above two inequality hold for all $n \ge M$
and $0< t < 2 \pi$. The case for $n < M$ then follows by noting the
fact that $\theta$ is of $\emph{bounded type}$.

Take $x \in \partial \Delta$ such that it attains the minimum of
$|G_{t}^{q_{n}}(y) - y|$. Then $[x, G_{t}^{q_{n}}(x)]$ has a
definite space around it inside $[G_{t}^{-q_{n}}(x),
G_{t}^{2q_{n}}(x)]$.  Let $M = [x, G_{t}^{q_{n}}(x)]$ and $T =
[G_{t}^{-q_{n}}(x), G_{t}^{2q_{n}}(x)]$. Since $\theta$ is of
bounded type, the $\emph{intersection multiplicity}$ for
$\{T_{t}^{k}, 0 \le k \le 5q_{n}\}$ has a uniform upper bound
dependent only on $\theta$. Applying Lemma~\ref{chang-zhou} to the
intervals $M \subset T$ and $N = q_{n}, 2q_{n}, 3q_{n}, 4q_{n}$, and
$5q_{n}$, respectively. Note that the $\emph{multiplicity}$ of the
corresponding collection of intervals is bounded above by some
constant dependent only on $\theta$. It follows that the six
intervals
 $[G_{t}^{-5q_{n}}(x), G_{t}^{-4q_{n}}(x)]$,
 $[G_{t}^{-4q_{n}}(x), G_{t}^{-3q_{n}}(x)]$,
$[G_{t}^{-3q_{n}}(x), G_{t}^{-2q_{n}}(x)]$, $[G_{t}^{-2q_{n}}(x),
G_{t}^{-q_{n}}(x)]$, $[G_{t}^{-q_{n}}(x), x]$ and $[x,
G_{t}^{q_{n}}(x)]$ are $L-$comparable with each other, where $L$ is
a constant dependent only on $\theta$.  Let $l$ be the minimum of
the length of these six intervals.

For any $z \in \partial \Delta $, it follows from the property of
the closed returns that there is an $0 \le i < 2q_{n+1}$ such that
$G_{t}^{i}(z)
 \in [G_{t}^{-5q_{n}}(x), G_{t}^{-4q_{n}}(x)]$.

Let us prove (\ref{first-ine}) first.  There are two cases. In the
first case, there is some $1 \le j \le 3$ such that $[
G_{t}^{i+jq_{n}}(z), G_{t}^{i+(j+1)q_{n}}(z)]$ has length less than
$l/2$.  Then
$$
[G_{t}^{i+jq_{n}}(z), G_{t}^{i+(j+1)q_{n}}(z)]
$$ has a definite space around it inside $[ G_{t}^{i+(j-1)q_{n}}(z),
G_{t}^{i+(j+2)q_{n}}(z)]$.  Let $$M = [G_{t}^{i+jq_{n}}(z),
G_{t}^{i+(j+1)q_{n}}(z)] \hbox{ and } T = [ G_{t}^{i+(j-1)q_{n}}(z),
G_{t}^{i+(j+2)q_{n}}(z)].$$  Apply Lemma~\ref{chang-zhou} to the
intervals $M \subset T$ and $N = i+ jq_{n}$.  Again note that the
$\emph{multiplicity}$ of the corresponding collection of intervals
is bounded above by some constant dependent only on $\theta$. We
thus get a definite space around $[z, G_{t}^{q_{n}}(z)]$ inside $[
G_{t}^{-q_{n}}(z), G_{t}^{2q_{n}}(z)]$. This proves
(\ref{first-ine}) in the first case.

In the second case, for each $j =1, 2, 3$, $[ G_{t}^{i+jq_{n}}(z),
G_{t}^{i+(j+1)q_{n}}(z)]$ has length not less than $l/2$. It follows
that the interval $[ G_{t}^{i+2q_{n}}(z), G_{t}^{i+3q_{n}}(z)]$ has
definite space around it inside the interval $[ G_{t}^{i+q_{n}}(z),
G_{t}^{i+4q_{n}}(z)]$.  As before, by applying
Lemma~\ref{chang-zhou}   we get a definite space around $[z,
G_{t}^{q_{n}}(z)]$ inside $[ G_{t}^{-q_{n}}(z), G_{t}^{2q_{n}}(z)]$.
This proves (\ref{first-ine}) in the second case.

Now let us prove (\ref{second-ine}). Let $b = \sup\{a_{k}\} <
\infty$ where $[a_{1}, \cdots, a_{n}]$ is the continued fraction of
$\theta$. Note that $[G_{t}^{-q_{n+1}}(z), z] \subset
[G_{t}^{q_{n}}(z), z]$, so from (\ref{first-ine}), we have
$$
|G_{t}^{q_{n+1}}(z)- z| \le K|G_{t}^{-q_{n+1}}(z)- z| <
K|G_{t}^{q_{n}}(z)- z|,
$$
and this implies the right hand of (\ref{second-ine}). To prove the
left hand, Note that
$$
[G_{t}^{q_{n}}(z), z] \subset \bigcup_{0 \le i \le
b}[G_{t}^{-iq_{n+1}}(z), G_{t}^{-(i+1)q_{n+1}}(z)],
$$
This implies that
$$
|G_{t}^{q_{n}}(z)- z|  <  \sum_{0 \le i\le b}|G_{t}^{-iq_{n+1}}(z)-
G_{t}^{-(i+1)q_{n+1}}(z)|.
$$
Applying (\ref{first-ine}) again, we have
$$
|G_{t}^{-iq_{n+1}}(z)- G_{t}^{-(i+1)q_{n+1}}(z)| \le
K^{i+1}|G_{t}^{q_{n+1}}(z)- z|
$$
for each $0 \le i \le b$. Therefore, we get
$$
|G_{t}^{q_{n}}(z)- z|   <  \sum_{0 \le i\le
b}K^{i+1}|G_{t}^{q_{n+1}}(z)- z|.
$$
By  modifying the value $K$, (\ref{second-ine})  follows.

\end{proof}

It is the time to prove Lemma~\ref{reduced}.
\begin{proof} We need
only to prove that there is an $M
>1$ dependent only on $\theta$ such that for any $x \in \partial
\Delta$ and $0< \delta < 2\pi$, the following inequality hold for
all $0< t < 2 \pi$,
$$
\frac{1}{M} <
\frac{|h_{t}(x+\delta)-h_{t}(x)|}{|h_{t}(x-\delta)-h_{t}(x)|} < M.
$$

Now for given $\delta$ and $x$, let us take $k \ge 1$ to be the
least integer such that one of the intervals $[x-\delta, x]$ and
$[x, x+\delta]$ contains $[G_{t}^{-q_{k}}(x),x]$ or $[x,
G_{t}^{q_{k}}(x)]$. Without loss of generality, Let us suppose
$[G_{t}^{-q_{k}}(x), x] \subset [x-\delta, x]$.  From the definition
of $k$, $[x-\delta , x] \subset [G_{t}^{-q_{k-2}}(x), x]$. Since
$\theta$ is of bounded type, by Lemma~\ref{metric-dis}, it follows
that $[ G_{t}^{-q_{k}}(x), x]$ and $[x-\delta, x]$ are
$L-$comparable where $1< L < \infty$ is some constant dependent only
on $\theta$. On the other hand, by the definition of $k$, we have
$[x, x+\delta] \subset [x, G_{t}^{q_{k-1}}(x)]$. By
Lemma~\ref{metric-dis}, $[x, G_{t}^{q_{k-1}}(x)]$ and $[
G_{t}^{-q_{k}}(x), x]$ are $K-$comparable for some $1 < K < \infty$
dependent only on $\theta$.  Therefore, $[x, x+\delta]$ and $[x,
G_{t}^{q_{k-1}}(x)]$ are $KL-$comparable. So we have
$$
|G_{t}^{q_{k-1}}(x) - x| < KL \delta.
$$ By
Lemma~\ref{metric-dis} again, there is an $\epsilon  > 0$ dependent
only on $\theta$ such that
$$
|G_{t}^{q_{k}}(x) -x| > (1+ \epsilon)|G_{t}^{q_{k+2}}(x) - x|,
$$
holds for all $x \in \partial \Delta$. Take $l > 1$ to be the least
integer such that $KL < (1+ \epsilon)^{l}$. It follows that $l$
depends only on $\theta$ and
$$
|G^{q_{k-1}}(x) - x| > ( 1 + \epsilon)^{l}|x - G^{q_{k+2l -1}}(x)|.
$$
It follows that $[x, G_{t}^{q_{k+2l-1}}(x)] \subset [x, x+\delta]$.
We then get
\begin{equation}\label{inc-1}
[x, G_{t}^{q_{k+2l-1}}(x)] \subset [x, x+\delta] \subset [x,
G_{t}^{q_{k-1}}(x)],
\end{equation}
and
\begin{equation}\label{inc-2}
[G_{t}^{-q_{k}}(x), x] \subset [x-\delta, x] \subset
[G_{t}^{-q_{k-2}}(x), x].
\end{equation}
Now for $x \in {\Bbb R}$, let $\{x\} \in (-1/2, 1/2)$ be the number
such that $x - \{x\} \in {\Bbb Z}$. From (\ref{inc-1}) and
(\ref{inc-2}), we have

$$
|\{q_{k+2l-1}\theta\}| \le |h_{t}(x+\delta)-h(x)| <
|\{q_{k-1}\theta\}|,
$$
and
$$
|\{q_{k}\theta\}| \le |h_{t}(x-\delta)-h(x)| < |\{q_{k-2}\theta\}|,
$$
Now the lemma follows from the assumption that $\theta$ is of
$\emph{bounded type}$.

\end{proof}

\subsection{Quadratic Siegel Rational Maps Modeled by $f_{\alpha}$}

In this section we will determine all the critical parameters $c$
such that $g_{c} \in R_{\theta}^{geom}$ and  the boundary of the
Siegel disk  of $g_{c}$ passes through both of the critical points.

\begin{lemma}\label{c-curve}
For $0< t< 2\pi$, let $g_{c(t)}$ be the Siegel rational map which
realizes $f_{t}$ in the sense of Lemma~\ref{rel-t}. Then $c(t)$ is
continuous in $(0, 2 \pi)$.
\end{lemma}
\begin{proof}
Let $t_{k} \to t$ for some $0< t< 2\pi$. We first claim  that the
sequence $\{c(t_{k})\}$ is contained in some compact set of ${\Bbb
C} - \{0, 1, -1\}$. Let us prove the claim now.

 Note that for each $0< t_{k}< 2\pi$,
$P_{f_{t_{k}}} =
\partial \Delta$ does not contain the infinity, and that the
infinity is fixed by $f_{t_{k}}$.  Following the same steps in the
proof of Theorem A, we can construct a Blaschke product, say
$G_{k}$, to model $f_{t_{k}}$. Write
\begin{equation}
G_{k}(z) = \lambda_{k} z \frac{z - p_{k}}{1-
\overline{p_{k}}z}\frac{z-q_{k}}{1 - \overline{q_{k}}z},
\end{equation}
where $|\lambda_{k}| = 1$ is some constant and $|p_{k}| >1, |q_{k}|<
1$.  In particular, by the construction, $G_{k}'(1) = 0$ for all $k
\ge 0$.

Let $h_{k}: \partial \Delta \to \partial \Delta$ be the
quasi-symmetric homeomorphism  such that $h_{k}(1) = 1$ and
$$
G_{k}|\partial \Delta = h_{k} \circ R_{\theta} \circ h_{k}.
$$
Let $H_{k}: \Delta \to \Delta$ be the Douady-Earle extension of
$h_{k}$. By Lemma~\ref{reduced} and the conformal natural property
of Douady-Earle extension,  there exists a uniform $0< \delta < 1$
which depends only on $M$ such that
$$
\sup_{z \in
\Delta}\bigg{|}\frac{(H_{k}^{-1})_{\bar{z}}}{(H_{k}^{-1})_{z}}\bigg{|}
\le \delta.
$$
Define
$$
\widehat{G}_{k}(z) =
\begin{cases}
 G_{k}(z) & \text{ for $|z| \ge 1$}, \\
 H_{k} \circ R_{\theta} \circ H_{k}^{-1}(z)& \text{ for $z \in \Delta$}.
\end{cases}
$$

Now as in the proof of Theorem A, we can pull back the complex
structure of $H_{k}^{-1}$ by $G_{k}$ and get a
$\widehat{G}_{k}-$invariant complex structure $\mu_{k}$ on the whole
sphere. Let $\phi_{k}$ be the quasiconformal homeomorphism of the
sphere which solves the Beltrami equation given by $\mu_{k}$ such
that $\phi_{k}(1) = 1$, $\phi_{k}(\infty) = \infty$, and
$\phi_{k}(0) = H(0)$. Then $\phi_{k}^{-1} \circ \widehat{G}_{k}
\circ \phi_{k}$ is a Siegel rational map in $R_{\theta}^{geom}$
which $\emph{realizes}$ $f_{t_{k}}$ in the sense of
Lemma~\ref{rel-t}.  We thus have
$$
g_{c(t_{k})} = \phi_{k}^{-1}
\circ \widehat{G}_{k} \circ \phi_{k}.
$$

Since when $|c|$ is large enough, $g_{c}$ has an attracting fixed
point at the infinity, and when $|c|$ is small enough, $g_{c}$ has
an attracting fixed point at the origin,
 by passing to a convergent subsequence, we may assume that
either $c(t_{k}) \to 1$ or $c(t_{k}) \to -1$.

First let us assume that $c(t_{k}) \to 1$.  From (\ref{norml}) and a
direct calculation, it follows that $g_{c(t_{k})} \to e^{2 \pi i
\theta} z$ uniformly in any compact set of the complex plane which
does not contain $1$. Let $D_{k}$ denote the Siegel disk of
$g_{c(t_{k})}$. By Lemma~\ref{distortion},  $\partial D_{k}$ is a
$K-$quasi-circle for some uniform $K$.  Therefore,  $D_{k} \to
\Delta$ in the Carath\'{e}odory sense. This implies that as $k \to
\infty$, the inner angle between $1$ and $c(t_{k})$ either converges
to $0$ or converges to $2\pi$. This contradicts with the assumption
that $t_{k} \to t$ for some $0< t< 2\pi$.

Now let us assume that $c(t_{k}) \to -1$.  Let
$$
g_{c(t_{k})} = \frac{a_{k} z^{2} + e ^{2 \pi i \theta} z}{b_{k}z +
1}.
$$
From (\ref{w-d}), we  get
$$
a_{k} \to 0 \hbox{ and } b_{k} \to \infty
$$
as $c(t_{k}) \to -1$.  Let $p_{k}$ be the fixed point of
$g_{c(t_{k})}$ which is distinct from $0$ and the infinity. By a
direct calculation, we have
$$
p_{k} = \frac{1 - e^{2 \pi i \theta}}{a_{k} - b_{k}}.
$$
We thus have  $p_{k} \to 0$ as $c(t_{k}) \to -1$.

 Let $\psi_{k}$ be
the M\"{o}bius transformation which maps $0$ to $0$, $1$ to $1$, and
$p_{k}$ to the infinity.  It follows that
$$
\psi_{k}(z) = \frac{z(1 - p_{k})}{z - p_{k}}.
$$
Consider the map
$$
g_{\tilde{c}(t_{k})} = \psi_{k} \circ g_{c(t_{k})} \circ
\psi_{k}^{-1}
$$
where
$$
\tilde{c}(t_{k})  = \psi_{k}(c(t_{k})) = \frac{c(t_{k})(1 -
p_{k})}{c(t_{k}) - p_{k}}.
$$
Since $p_{k} \to 0$ as $c(t_{k}) \to -1$, it follows that
$$
\tilde{c}(t_{k}) \to 1
$$ as $c(t_{k}) \to -1$.  Since the conjugation map $\psi_{k}$
preserves the inner angle, that is, the inner angle between $1$ and
$\tilde{c}(t_{k})$ is the same as that between $1$ and
$c(t_{k})$(Compare with Fact 2 in $\S4.2.2$), from the conclusion we
just obtained above, it follows that the inner angle between $1$ and
$\tilde{c}(t_{k})$ either converges to $0$ or converges to $2\pi$
also. We get a contradiction again. The claim has been proved.

Now let us prove that the sequence $\{c(t_{k})\}$ is convergent. By
passing to a subsequence, we may assume that $c(t_{k}) \to c \in
{\Bbb C} - \{0, 1, -1\}$. Since $\partial D_{k}$ is a
$K-$quasi-circle for every $k \ge 1$, it follows that the boundary
of the Siegel disk of $g_{c}$ is a quasi-circle also,  and moreover,
the inner angle between $1$ and $c$ is equal to $$\lim_{k \to
\infty} t_{k} = t.$$ Since $g_{c_{k}} \to g_{c}$ uniformly in any
compact set of the complex plane, by Remark~\ref{exp}, it follows
that $g_{c}$ $\emph{realizes}$ $f_{t}$ in the sense of
Lemma~\ref{rel-t}. Since such $c$ must be unique by Theorem B, it
follows that any convergent subsequence of $c(t_{k})$ converges to
the same limit. The lemma follows.
\end{proof}

\begin{lemma}\label{dic-c}
$\lim_{t\to 0} c(t)= 1$ or $-1$.
\end{lemma}
\begin{proof}
Let us prove it by contradiction. Since when $|c|$ is large enough,
$g_{c}$ has an attracting fixed point at the infinity, and when
$|c|$ is small enough, $g_{c}$ has an attracting fixed point at the
origin,  we may assume that there is a sequence $t_{k} \to 0$ such
that $c(t_{k}) \to c$ for some $c \in {\Bbb C} - \{0, 1, -1\}$. Let
$D_{c}$ and $D_{c(t_{k})}$ denote respectively the Siegel disks of
$g_{c}$ and $g_{c(t_{k})}$, which are centered at the origin.  Since
every $\partial D_{c(t_{k})}$ is a $K-$quasi-circle passing through
$1$ and $c(t_{k})$ and $c(t_{k}) \to c$, it follows that there is a
$\delta > 0$ such that $$B_{\delta}(0) \subset D_{c(t_{k})}$$ for
all $t_{k}$.  Let $p \ne 0$ be such that $g_{c}(p) = 0$. Then there
is a  $r > 0$ such that
$$
B_{r}(p) \cap D_{c} = B_{r}(p) \cap D_{c(t_{k})} = \emptyset
$$
for all $t_{k}$. Let $\phi(z) = 1/(z - p)$. Set
$$
T_{k}(z) = \phi \circ g_{c(t_{k})} \circ \phi^{-1}  \hbox{ and }
T(z) = \phi \circ g_{c} \circ \phi^{-1}.
$$
Denote the corresponding Siegel disks of $T_{k}$ and $T$  by
$D_{T_{k}}$ and $D_{T}$, respectively. Clearly, as $k \to \infty$,
$T_{k} \to T$ uniformly with respect to the spherical metric, and
moreover, there is a compact set $E$ of the complex plane such that
$D_{T} \subset E$ and $D_{T_{k}} \subset E$ for all $k \ge 1$. Since
every $\partial D_{c(t_{k})}$ is a $K-$quasi-circle for some uniform
$1< K < \infty$ by Lemma~\ref{distortion},  it follows that every
$\partial D_{T_{k}}$ is a $K-$quasi-circle.  Let $h_{k}:\Delta \to
D_{T{k}}$  be the univalent map such that $h_{k}'(0) > 0$ and
$h_{k}^{-1} \circ T_{k} \circ h_{k} = R_{\theta}$.  Since $\partial
D_{T{k}}$ is a uniform $K-$quasi-circle, by passing to a convergent
subsequence, we may assume that $h_{k}$ uniformly converges to $h$
on $\overline{\Delta}$ such that $h^{-1} \circ T \circ h =
R_{\theta}$. This implies that $\partial D_{T}$ is a quasi-circle
also and  passes through both the two critical points of $T$. In
particular, the $\emph{inner angle}$ of the two critical points of
$T$  must be $0$ or $2\pi$, and therefore, the two critical points
of $T$ coincide.  It follows that $1 = c$. This is a contradiction
and the lemma follows.

\end{proof}

\begin{lemma}\label{two-ends}
$\{\lim_{t\to 0} c(t), \lim_{t\to 2 \pi} c(t)\} = \{1, -1\}$.
\end{lemma}

\begin{proof}
For $0< t <2\pi$, let $\tilde{c}(t)$ be the critical parameter such
that $g_{\tilde{c}(t)}$ realizes $\tilde{f}_{t}$ in the sense of
Lemma~\ref{rel-t}. Let $p_{t}$ be the fixed point of $g_{c(t)}$
which is distinct from $0$ and the infinity. From Fact 2 in
$\S4.2.2$, it follows that the M\"{o}bius transformation
$$
\phi_{t}(z)  =  \frac{(1-p_{t})z}{z - p_{t}}
$$
conjugates $g_{t}$ to $g_{\tilde{c}(t)}$. By a direct calculation,
we get
$$
\tilde{c}(t) = \phi_{t}(c(t)) = \frac{(e^{2 \pi  i\theta} -2)c(t) +
e^{2 \pi i\theta}}{- e^{2 \pi i\theta} c(t) + 2 -e^{2 \pi i\theta}}.
$$
From Fact 3 in $\S4.2.2$, it follows that the M\"{o}bius
transformation
$$
\psi_{t}(z) =  z / \tilde{c}(t)
$$
conjugates $g_{\tilde{c}(t)}$ to $g_{c(2 \pi -t)}$. In particular,
\begin{equation}\label{inv}
c(2\pi - t) =  1/\tilde{c}(t).
\end{equation}
By Lemma~\ref{dic-c}, either $\lim_{t \to 0 }c(t) = 1$,  or $\lim_{t
\to 0 }c(t) = -1$. If $\lim_{t \to 0 }c(t) = 1$, then $$\lim_{t \to
2 \pi}c(t) = \lim_{t \to 0}c(2\pi - t) = \lim_{t \to 0}
1/\tilde{c}(t) = \lim_{t \to 0}\frac{-e^{2 \pi i \theta} c(t) + 2
-e^{2 \pi i \theta}}{ (e^{2 \pi i \theta} -2)c(t) + e^{2 \pi i
\theta}}  = -1.$$

If $\lim_{t \to 0 }c(t) = -1$, then $$\lim_{t \to 2 \pi}c(t) =
\lim_{t \to 0}c(2\pi - t) = \lim_{t \to 0} 1/\tilde{c}(t) = \lim_{t
\to 0}\frac{-e^{2 \pi i \theta} c(t) + 2 -e^{2 \pi i \theta}}{ (e^{2
\pi i \theta} -2)c(t) + e^{2 \pi i \theta}}  = 1.$$
Lemma~\ref{two-ends} follows.
\end{proof}

From Lemma~\ref{c-curve} and Lemma~\ref{two-ends}, it follows that
$c(t), 0 < t< 2\pi$ is a continuous curve segment which does not
intersect with itself and which connects $1$ and $-1$. By using the
same argument, the same conclusion can be derived for the curve
$\tilde{c}(t), 0< t < 2 \pi$.  Let  $\gamma = \{c(t)\big{|}0< t <
2\pi\}$ and $\gamma' = \{\tilde{c}(t)\big{|}0< t< 2\pi\} =
\{1/c(t)\big{|}0< t < 2 \pi\}$. It is clear that except the two end
points, $\gamma$ does not intersect $\gamma'$(This is simply because
for $0< t, t' <2 \pi$, $g_{c(t)}$ and $g_{\tilde{c}(t')}$
$\emph{realize}$ different topological models, which are indicated
by Figure 15 and Figure 16, respectively).  It follows that $$\xi =
\gamma \cup \gamma' \cup \{1, -1\}$$ is a simple closed curve. From
(\ref{inv}),  the map $c \to 1/c$ preserves the curve $\xi$ but
reverses its orientation.  It follows that $\xi$ separates $0$ and
the infinity. We summarize these as follows:
\begin{lemma}\label{par-c}
The curve $\xi= \gamma \cup \gamma' \cup \{1, -1\}$ is a simple
closed curve which separates $0$ and the infinity. Moreover, $\xi$
is invariant under the map $z \to 1/z$.
\end{lemma}

\begin{figure}
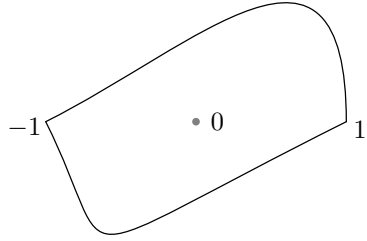

\bigskip
\begin{center}
\centertexdraw { \drawdim cm \linewd 0.02 \move(0 1)

\move(-2 -1) \clvec(0 0)(2 2)(2 -1)

\move(-2 -1) \clvec(-1 -3)(-2 -3)(2 -1) \move(-2.5 -1.2)
\htext{$-1$}

\move(2.1 -1.2) \htext{$1$}

\move(0 -1) \fcir f:0.5 r:0.05 \move(0.2 -1.1) \htext{$0$}

 \move(2 -2.5)
 }
\end{center}
\vspace{0.2cm} \caption{Critical parameters determined by $f_{t}$
and $\tilde{f}_{t}$ for $0< t < 2 \pi$}
\end{figure}

\subsection{Quadratic Siegel Rational Maps with One Finite Critical Orbit }
In this section, we consider all those quadratic rational maps which
have a fixed Siegel disk of rotation number $\theta$ and a critical
point with finite forward orbit. The aim of this section is to show
that such Siegel rational maps belong to $R_{\theta}^{geom}$. That
is, for any such map, the another critical point must lie in the
boundary of the Siegel disk which is a quasi-circle. Before we state
the result, let us introduce some notations first.

Let $0\le m < n$ be integers and $t \in {\Bbb C}$.  Let us define
$Z_{m, n}^{s}$ to be the set of all the quadratic rational maps $g$
such that
\begin{itemize}
\item[1.]  $g'(1) = g'(c) = 0$,
\item[2.] $g^{m}(c)= g^{n}(c)$,
\item[3.] $g$ fixes $0$ and the infinity,
\item[4.] $g'(0) = s$.
\end{itemize}
Recall that $\lambda = e^{2 \pi  i\theta}$. Define $R_{m,n}^{\theta}
= Z_{m, n}^{\lambda} \cap R_{\theta}^{geom}$. The main result of
this section is as follows.
\begin{lemma}\label{main-lem}
$Z_{m,n}^{\lambda} = R_{m,n}^{\theta}$.
\end{lemma}

Before the proof of Lemma~\ref{main-lem}, let us prove a few lemmas.
\begin{lemma}\label{hol-two}
Let  $ 0\le m < n$ be two integers. Then for any $\epsilon>0$, there
is some $0< |s| < 1$, such that for any quadratic rational map
$g_{c} \in Z_{m,n}^{\lambda}$, there is a quadratic rational map $g
\in Z_{m,n}^{s}$ such that $d(g, g_{c}) < \epsilon$.
\end{lemma}

\begin{proof}
  For $s \ne 0$, and $t \ne 0, 1, -1$, consider the function
\begin{equation}\label{FST}
F_{s, t}(z) = \frac{a(s,t)z^{2} + s z}{b(s, t)z + 1}
\end{equation}
where $a(s, t) = -s(1+ t)/2t$ and $b(s,t) = -2/(1+ t)$. It follows
that $F_{s,t}'(0) =s$, and $F_{s, t}'(1) = F_{s, t}'(t) = 0$. There
are three cases.

Case 1. $c \ne \infty$, and $g_{c}^{m}(c) = g^{n}_{c}(c) \ne
\infty$.  It is clear that $F_{\lambda, c}(z) = g_{c}(z)$. It
follows that there is an open neighborhood of $\lambda$, say $U$,
and an open neighborhood of $c$, say $V$, such that both the
functions $F_{s, t}^{n}(t)$ and $ F_{s, t}^{m}(t)$ are holomorphic
for $(s, t) \in U \times V$. In particular, by taking $V$ smaller,
we can assume that as $s \to \lambda$, $F_{s, t}^{n}(t) \to
F_{\lambda, t}^{n}(t)$ and $F_{s, t}^{m}(t) \to F_{\lambda,
t}^{m}(t)$ uniformly for $t \in V$. Since $F_{\lambda,
t}^{n}(t)-F_{\lambda, t}^{m}(t)$ has a zero at $c \in V$, it follows
from Rouch\'{e} theorem that for every small $r > 0$, there is a
$\delta > 0$, such that for every $s\in B_{\delta}(\lambda)$, there
is a point $c_{s} \in B_{r}(c)$ such that $F_{s, c_{s}}^{n}(c_{s}) -
F_{s, c_{s}}^{m}(c_{s}) = 0$. Since $|\lambda| = 1$,  for any $s$
close to $\lambda$ with $|s| < 1$, one can take $c_{s}$ close to $c$
such that $F_{s, c_{s}}^{n}(c_{s}) - F_{s, c_{s}}^{m}(c_{s}) = 0$.
The lemma in this case follows by taking  $g(z) = F_{s, c_{s}}(z)$.

Case 2. $c \ne \infty$, $g_{c}^{l}(c) = \infty$ for some $1 \le l
\le m$. We may assume that  $l$ is the least positive integer such
that $g_{c}^{l}(c) = \infty$. From (\ref{FST}), it follows that
$$b(\lambda, c)g_{c}^{l-1}(c) + 1  = 0.$$ Then  instead of
considering  the function $F_{s, t}^{n}(t) -  F_{s, t}^{m}(t)$, this
time we consider the function $b(s, t)F_{s, t}^{l-1}(t) + 1$. The
lemma in this case then follows by  using the same argument as in
the proof of the first case. The reader shall have no difficulty to
supply the details.

Case 3. $c = \infty$. In this case, just take $g = sg_{c}$ where $s$
is any number close enough to $\lambda$ with $|s| < 1$.

\end{proof}

\begin{lemma}\label{tem}
Let $0< |s| < 1$. Then every   $f \in Z_{m,n}^{s}$ has exactly three
distinct fixed points $0, \infty$ and some complex value $p$.
\end{lemma}
\begin{proof}In fact, if this were not true, then the infinity
would be a double root of $f(z) - z$ and hence a parabolic fixed
point of $f$. Therefore, one of the forward critical orbit
approaches to the infinity and the other one approaches to the
origin.   This is a contradiction with the assumption that $f^{m}(c)
= f^{n}(c)$ for some $0 \le m < n$.  The lemma follows.
\end{proof}

\begin{lemma}\label{Count-Z}
Let $0< |s| < 1$. Then  for every $f \in Z_{m, n}^{s}$, the
conformal equivalent class $[f]$ of $f$ contains exactly two
elements in $Z_{m,n}^{s}$.
\end{lemma}
\begin{proof}
 Let $g \in [f]$ such that $g \ne f$.  Assume that
 $g  = \phi \circ f \circ \phi^{-1}$ for some M\"{o}bius map $\phi$.
 Let $p\ne 0, \infty $ be the fixed point of $f$. Since $g^{m}(c)= g^{n}(c)$,
 it follows that  $g$  has exactly one non-repelling fixed point
 which is the origin. This implies that  the forward orbit $\{g^{k}(1)\}$
 is the only infinite critical orbit of $g$.   Therefore, we get
 $\phi(0) = 0$ and $\phi(1)= 1$. Now let $0, \infty, q$ be the three
 fixed points of $g$. It follows that $\{\phi(\infty), \phi(p)\} = \{\infty, q\}$.
 Note that  $\phi(\infty) \ne \infty$, for otherwise $\phi = id$ and hence $g = f$,
 which contradicts with  the assumption that $g \ne f$.   It follows that
 $\phi(p) = \infty$.  This implies  that $\phi$ is uniquely
 determined by $f$ and the lemma follows.
\end{proof}
By using the same argument as in the proof of Lemma~\ref{tem}, one
can show that  every $f \in R_{m,n}^{\theta}$ also has  three
distinct fixed points. Then using the same argument as in the proof
of Lemma~\ref{Count-Z},  one has
\begin{lemma}\label{Count-R}
For every $f \in R_{m, n}^{\theta}$, the conformal equivalent class
$[f]$ of $f$ contains exactly two elements in $R_{m,n}^{\theta}$.
\end{lemma}

For $|s| < 1$, let $R_{s}$ be the set which consists of all the
quadratic rational maps $f$ such that $f(0) = 0$ and $f'(0) = s$.
  For each $f \in R_{s}$, the map $f$ restricted to a
suitable neighborhood of its Julia set is polynomial-like of
quadratic with connected Julia set and hence is hybrid equivalent to
a unique quadratic polynomial $z^{2} + c$ for some $c \in M$ where
$M$ is the Mandelbrot set.  This induces a homeomorphism between the
set of the conformal equivalent classes of $R_{s}$, say $M_{s}$, and
the Mandelbrot set $M$(see \cite{GK}, or the proof of Lemma 8.5,
\cite{Mi}).  Let $Q_{m,n}$ be the set of all the quadratics
$q_{c}(z) = z^{2} + c$ such  that $q_{c}^{m}(0) = q_{c}^{n}(0)$.  It
follows from Lemma~\ref{Count-Z},  that
\begin{lemma}\label{Count-M}
For $0< |s| < 1$ and any integers $0 \le m < n$,  $|Z_{m,n}^{s}| = 2
|Q_{m,n}|$.
\end{lemma}

Now let us prove Lemma~\ref{main-lem}.
\begin{proof}
It suffices to show that $|Z_{m,n}^{\lambda}| \le
|R_{m,n}^{\theta}|$. By Lemma~\ref{hol-two}, we have
$|Z_{m,n}^{\lambda}| \le |Z_{m,n}^{s}|$ for some $0< |s|<1$. Note
that each element $f$ in $Q_{m,n}$ induces a topological branched
covering map $\tilde{f}$ in $R_{\theta}^{top}$ by topologically
mating itself with $z^{2} + \lambda z$(for one way of the
construction of such mating, see $\S7$ of \cite{Mi}). Clearly, the
resulted map $\tilde{f}$ has no Thurston obstructions outside the
$\emph{rotation disk}$. Moreover, if $f_{1}, f_{2}$ are two
different elements in $Q_{m,n}$, then the two maps $\tilde{f}_{1},
\tilde{f}_{2}$ in $R_{\theta}^{top}$ induced by $f_{1}$ and $f_{2}$
belong to different combinatorial classes. This, together with
Theorem A and Lemma~\ref{Count-R}, implies that  $  2 |Q_{m,n}|\le
|R_{m,n}^{\theta}|$. It follows from Lemma~\ref{hol-two} and
\ref{Count-M} that $|Z_{m,n}^{\lambda}| \le |Z_{m,n}^{s}| = 2
|Q_{m,n}| \le |R_{m,n}^{\theta}|$. The lemma follows.

\end{proof}


\subsection{Critical Parameterization}

In this section, we will give a critical parameterization of the
space of all the Blaschke products in the following form,
\begin{equation}\label{def-B}
B_{p,q}(z) = z\frac{z-p}{1-\overline{p}z}\frac{z-q}{1-\overline{q}z}
\end{equation}
where $|p| > 1$, $|q|<1$.

\begin{lemma}\label{p-q}
For any compact set $K \subset {\Bbb C}-\partial \Delta$, there is a
$\delta > 0$, such that in either of the following two cases,
\begin{itemize}
\item[1.] $p \in K, |p| > 1$, and dist$(q, \partial \Delta) < \delta$, or
\item[2.] $q \in K, |q| < 1$, and dist$(p, \partial \Delta) < \delta$,
\end{itemize}
$B_{p, q}$ has at least two distinct critical points in $\partial
\Delta$.
\end{lemma}

\begin{proof}
Let $T_{p,q}(\alpha) = - i \log B_{p, q}(e^{ i\alpha})$ for $0 \le
\alpha \le 2 \pi$. Then
\begin{equation}\label{dir-for}
T_{p,q}'(\alpha) = 1 + \frac{1 - |q|^{2}}{|1 - \bar{q}e^{i
\alpha}|^{2}} + \frac{ 1-|p|^{2}}{|1 - \bar{p}e^{ i \alpha}|^{2}}.
\end{equation}

Let us assume that we are in the first case and the second case can
be proved in the same way. Suppose that the lemma were not true.  By
passing to a convergent subsequence, we may assume that there exist
a sequence $p_{k} \to p \in K$ and a sequence $q_{k} \to e^{i\alpha}
\in
\partial \Delta$ such that $B_{p_{k}, q_{k}}$ has at most one
critical point in the unit circle.  since
$$\int_{0}^{2 \pi} \frac{ |p|^{2}-1}{|1 - \bar{p}e^{ i
\alpha}|^{2}}d\alpha = 2 \pi,$$ it follows that there exist
$\beta_{1}< \alpha < \beta_{2}$ such that
\begin{equation}\label{two-ines}
\frac{|p|^{2}-1}{|1 - \bar{p}e^{i
\beta_{1}}|^{2}}> 1, \hbox{\:and\:} \frac{|p|^{2}-1}{|1 -
\bar{p}e^{i \beta_{2}}|^{2}}< 1.
\end{equation}

Note that as $q_{k} \to e^{i \alpha}$, $\alpha_{k} \to \alpha$ where
$a_{k} = \arg (q_{k})$. By a simple calculation, it is easy to see
that
$$
\frac{1 - |q_{k}|^{2}}{|1 - \bar{q_{k}}e^{ i t}|^{2}} \to 0
$$ uniformly on any closed sub-interval of $[0, 2 \pi]$ which
does not contain $\alpha$, and
$$
\frac{1 - |q_{k}|^{2}}{|1 - \bar{q_{k}}e^{ i \alpha_{k}}|^{2}} \to
\infty.
$$
It follows from (\ref{dir-for}) and (\ref{two-ines}) that for all
$k$ large enough, we have
$$
T'_{p_{k}, q_{k}}(\alpha_{k})
> 0, T'_{p_{k}, q_{k}}(\beta_{1}) <  0, \hbox{ and }T'_{p_{k},
q_{k}}(\beta_{2}) <  0.
$$
Since $\beta_{1} < \alpha_{k} < \beta_{2}$ for all $k$ large enough,
the proof of the first case can thus be completed by using the
Immediate Value Theorem.

\end{proof}
\begin{lemma}\label{p-q-t}
For any compact set $K \subset \Bbb C - \partial \Delta$, there is a
$\delta > 0$ such that if $B_{p, q}$ has a critical point in $K$,
then $d(p, \partial \Delta) L \delta$ and  $d(q, \partial \Delta) >
\delta$.
\end{lemma}
\begin{proof}
This is because as $p$ and $q$ approach $\partial \Delta$,  by
passing to a subsequence, $B_{p, q}$ converges to a rigid rotation
uniformly in any compact set $K \subset \Bbb C - \partial \Delta$.
\end{proof}
\begin{definition}\label{the set B}
Let $\mathcal{B}$  be the set which consists of all the Blaschke
products $B_{p,q}$ satisfying the following three properties:
\begin{itemize}
\item[1.]  $B_{p, q}$ has a double critical point at $1$,
\item[2.] the other two critical points $c$ and $\frac{1}{\overline{c}}$
are symmetric about the unit circle such that  $c \in \widehat{\Bbb
C} - \overline{\Delta}$,
\item[3.]  $|p| > 1$ and $|q| < 1$.
\end{itemize}
\end{definition}

\begin{lemma}
Let $B_{p, q} \in \mathcal{B}$. Then $B_{p, q}|\partial \Delta:
\partial \Delta \to \partial \Delta$ is a homeomorphism which
preserves the orientation.
\end{lemma}
\begin{proof}
Since $$\int_{0}^{2\pi}T_{p,q}'(\alpha)d\alpha = 2\pi,$$ it follows
that the topological degree of $B|\partial \Delta: \partial \Delta
\to \partial \Delta$ is $1$.   If $B|\partial \Delta$ is not a
homeomorphism, then $B|\partial \Delta$ would have two distinct
critical points.  This is a contradiction with the definition of
$\mathcal{B}$. The fact that $B|\partial \Delta$ preserves the
orientation also follows. The proof of the lemma is completed.
\end{proof}

\noindent $\bold{Critical\:Parameterization\:of\:}$$\mathcal{B}$.
Let $B_{p, q} \in \mathcal{B}$ and let $w = p+q, v = pq$. Assume
that $c \ne \infty$. Therefore, $q \ne 0$ and hence $v \ne 0$.  By a
direct calculation, we get
$$
B_{p,q}'(z) = \frac{\overline{v}z^{4} -2 \overline{w}z^{3} + (3 +
|w|^{2}-|v|^{2})z^{2} -2wz + v}{(\overline{v}z^{2} - \overline{w}z +
1)^{2}}
$$

The numerator of $B_{p,q}'(z)$ can be written into
$$
\overline{v}(z^{4} -\frac{2\overline{w}}{\overline{v}}z^{3} +
(\frac{3}{\overline{v}} +
\frac{|w|^{2}}{\overline{v}}-\frac{|v|^{2}}{\overline{v}})z^{2} -
\frac{2w}{\overline{v}}z + \frac{v}{\overline{v}}) = \overline{v}(z
- 1)^{2}(z - c)(z - \frac{1}{\overline{c}}).
$$
It follows that $v/\bar{v} = c/\bar{c}$ and hence $v/c$ is a real
number. So we have either $v = c|v|/|c|$ or $v = -c|v|/|c|$. Set $t
= (2 + c + \frac{1}{\overline{c}})/2$ and $s = 1 +
\frac{c}{\overline{c}}+2(c + \frac{1}{\overline{c}})$. By comparing
the coefficients of the two polynomials in the above equation, it
follows that $\frac{\overline{w}}{\overline{v}}= t$ and hence
$|w|^{2} = |t|^{2}|v|^{2}$. It also follows that
$\frac{3}{\overline{v}} + \frac{|w|^{2}}{\overline{v}} - v = s$.
This gives us
\begin{equation} \label{v-s}
\frac{3}{\overline{v}} + |t|^{2}v - v =s.
\end{equation}
Note that $\overline{c}s = 2(1+ |c|^{2}) + c + \overline{c} = 1 +
|c|^{2} + |1+c|^{2} > 0$, so if $v = c|v|/|c|$, from (\ref{v-s}) we
get
\begin{equation}\label{t-v-s-1}
(|t|^{2}-1)|v|^{2} -|s||v|+ 3 = 0,
\end{equation}
and if $v = -c|v|/|c|$,  we get
\begin{equation}\label{t-v-s-2}
(|t|^{2}-1)|v|^{2} +|s||v|+ 3 = 0.
\end{equation}

Since $|s|^{2} - 12(|t|^{2}-1) = |c + \frac{1}{c} -2|^{2} >0$ for
all $c \ne 1$, it follows that for $|t|^{2} -1 \ge 0$,
(\ref{t-v-s-2}) has no positive solutions and $|v|$ must satisfy
(\ref{t-v-s-1}). Therefore, $v = c|v|/|c|$, and
\begin{equation}\label{d-1}
|v| = \frac{|s| - |c+\frac{1}{c}-2|}{2(|t|^{2}-1)},
\end{equation}
or
\begin{equation}\label{d-2}
|v| = \frac{|s| + |c+\frac{1}{c}-2|}{2(|t|^{2}-1)}.
\end{equation}

For $|t|^{2} -1 < 0$, we have two cases. In  the first cases, $v =
-c|v|/|c|$ and $|v|$ satisfies (\ref{t-v-s-2}). Since $|s|^{2} -
12(|t|^{2}-1) = |c + \frac{1}{c} -2|^{2} >0$, there is only one
positive solution of (\ref{t-v-s-2}) given by
\begin{equation}\label{d-3}
|v| = \frac{-|s| - |c+\frac{1}{c}-2|}{2(|t|^{2}-1)},
\end{equation}
In the second case,  $v = c|v|/|c|$, and $|v|$ satisfies
(\ref{t-v-s-1}). Since $|s|^{2} - 12(|t|^{2}-1) = |c + \frac{1}{c}
-2|^{2} >0$, there is only one positive solution of (\ref{t-v-s-1})
given by
\begin{equation}\label{d-4}
|v| = \frac{|s| - |c+\frac{1}{c}-2|}{2(|t|^{2}-1)}.
\end{equation}

\begin{ex1}  Let $c = 2$. Then $t = \frac{9}{4}$ and $s = 7$.
Since $|t|^{2} - 1> 0$, it follows that $v  = c|v|/|c|$. By
(\ref{d-1}) and (\ref{d-2}), we have $|v|  =\frac{4}{5}$, or $|v| =
\frac{12}{13}$. Then we have two cases:

Case 1.  $v = \frac{4}{5}$, and $w = \frac{9}{5}$.  $\{p, q\} = \{1,
\frac{4}{5}\}$.

Case 2. $v = \frac{12}{13}$, and $w = \frac{27}{13}$.  $\{p, q\} =
\{1.432575\cdots, 0.644348\cdots\}$.
\end{ex1}

\begin{ex2} Let $c = -2$. Then $t = -\frac{1}{4}$ and $s = -3$.
Since $|t|^{2} - 1<0$, we have again two cases. In the first case,
$v = c|v|/|c|$ and in the second case, $v = -c|v|/|c|$.  By
(\ref{d-3}) and (\ref{d-4}), we get $ v = -4/5$ or $ v = 4$.

Case 1.  $v = -\frac{4}{5}$ and $w = \frac{1}{5}$.
 $\{p, q\} = \{1,  -\frac{4}{5}\}$.

Case 2. $v = 4$ and $w = -1$.
 $\{p, q\} =\{ -0.5 + 1.936491i, -0.5 - 1.936491i\}$.
\end{ex2}
\begin{remark}\label{solution}
For any $c$ with $|c| > 1$, let $(p_{c}, q_{c})$ be one of the
solutions obtained above such that $|p_{c}| > 1$ and $|q_{c}| < 1$.
Then $B_{p_{c}, q_{c}}$ has exactly a double critical point at $1$
and two distinct critical points at $c$ and $1/\overline{c}$.
\end{remark}
Recall that $ t = (2+ c + \frac{1}{\overline{c}})/2$.  By a direct
calculation, it follows that the curve
\begin{equation}\label{gamma}
\gamma = \{c\:\big{|} \:|t|^{2}-1 = 0, |c| > 1\} =  \{re^{it}
\:\big{|}\:r + r^{-1} + 4\cos (t) = 0, r > 1\}
\end{equation}
separates $\widehat{\Bbb C} - \overline{\Delta}$  into two
components.  Let us denote them by $U$ and $V$ respectively (see
Figure 18). Define  $$\Phi: \mathcal{B} \to \widehat{{\Bbb C}}-
\overline{\Delta}$$ by $\Phi(B_{p, q}) = c$.

\begin{lemma}\label{par-S}
The map $\Phi$ is a homeomorphism between $\mathcal{B}$ and $U$.
\end{lemma}
In the four cases of the two examples above, we see that only Case 2
of Example 1 produces the desired Blaschke product $B_{p, q}$ which
satisfies $|p|> 1$ and $|q| < 1$. In the following, we will use
continuation method to show that along this branch, all the other
critical parameters in $U$ can produce a unique desired Blaschke
product $B_{p, q}$, and that along all the other three branches, the
solution pairs $\{p, q\}$ obtained do not satisfy the condition
$|p|> 1$ and $|q|< 1$, that is, either one of them lies in the unit
circle, or both of them belong to the outside of the unit disk.
\begin{proof}
It is clear that $\Phi$ is continuous. First let us prove that for
any $B_{p, q} \in \mathcal{B}$, $\Phi(B_{p,q}) \in U$.  Assume that
this is not true. Let $\Phi(B_{p, q}) = c_{0}$. There are two cases.
\begin{figure}
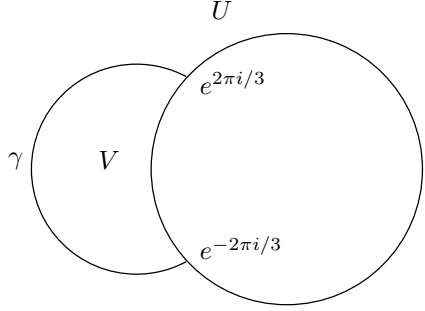

\bigskip
\begin{center}
\centertexdraw { \drawdim cm \linewd 0.02 \move(0 1)

\move(0 -3) \lcir r:1.8

\move(-2 -3) \larc r:1.4 sd:62 ed:298

\move(-1.15 -2) \htext{$e^{2 \pi i/3}$}

\move(-1.15 -4.2) \htext{$e^{-2 \pi i/3}$}

\move(-2.5 -3) \htext{$V$} \move(-3.7 -3) \htext{$\gamma$} \move(-1
-1) \htext{$U$}

}
\end{center}
\vspace{0.2cm} \caption{The critical parameter space $U$}
\end{figure}

In the first case, $c_{0} \in \gamma$ where $\gamma$ is the open
curve segment defined in (\ref{gamma}). That is to say,
$$
|t|^{2} -1 = |(2+ c_{0} + \frac{1}{\overline{c}_{0}})/2|^{2} - 1 =
0.
$$
It follows that $|v|$ must satisfy (\ref{t-v-s-1}), which  is
degenerated to a linear equation in this case. So $|v|$ can be
computed as the limit of (\ref{d-1}) or (\ref{d-2}) by letting $c
\to c_{0}$ from the inside of $U$.  It is easy to see that in
(\ref{d-2}), $|v|$ approaches to the infinity as $c$ approach to
$c_{0}$ (the numerator has a positive lower bound but the
denominator goes to zero). It thus follows that in this case, $|v|$
must be equal to the limit of (\ref{d-1}) as $c$ approaches to
$c_{0}$ from the inside of $U$.  Take a curve segment $\eta \subset
U$ which connects $c_{0}$ and the point $2$ such that
$$
d(\eta, \partial \Delta) > 0.
$$
See Figure 19 for an illustration.
\begin{figure}
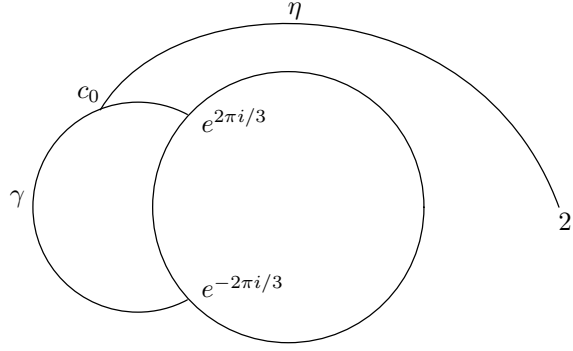

\bigskip
\begin{center}
\centertexdraw { \drawdim cm \linewd 0.02 \move(0 1)

\move(0 -3) \lcir r:1.8

\move(-2 -3) \larc r:1.4 sd:62 ed:298

\move(3.6 -3) \clvec(2.5 0)(-1.5 0 )(-2.5 -1.7)

\move(-2.8 -1.6) \htext{$c_{0}$} \move(3.6 -3.3) \htext{$2$}
\move(-1.15 -2) \htext{$e^{2 \pi i/3}$}

\move(-1.15 -4.2) \htext{$e^{-2 \pi i/3}$}

 \move(-3.7 -3) \htext{$\gamma$}
\move(0 -0.5) \htext{$\eta$}

}
\end{center}
\vspace{0.2cm} \caption{The continuation of the pair $\{p_{c},
q_{c}\}$ along $\eta$}
\end{figure}

For each $c \in \eta$, denote the corresponding values $s, t, v, w,
p, q$ by $$s_{c}, t_{c}, v_{c}, w_{c}, p_{c}, \hbox{ and  } q_{c},$$
respectively. Since $|t_{c}|^{2} - 1 > 0$, $v_{c}$ satisfies
(\ref{t-v-s-1}). We thus have $v_{c} = c |v_{c}|/|c|$. For $c \in
\eta$, we solve $|v_{c}|$ by (\ref{d-1}) and get $w_{c}$ by the
relation
$$\bar{w}_{c}/\bar{v}_{c} = t_{c} = (2 + c + 1/\bar{c})/2.$$
Now we solve the pair  $p_{c}, q_{c}$
which are the two solutions of the quadratic equation $$x^{2} -
w_{c} + v_{c} = 0.$$ Clearly, $p_{c}$ and  $q_{c}$ depend
continuously on $c$. From Case 1 of Example 1, it follows that
$\{p_{2}, q_{2}\} = \{1, 4/5\}$. We now claim that there is a
$\delta > 0$, such that for each $c \in \eta$, either $d(p_{c},
\partial \Delta) > \delta$, or $d(q_{c},
\partial \Delta)
> \delta$. In fact, if this were not true, then we would have a sequence
$\{c_{k}\} \subset \eta$ such that $$p_{c_{k}} \to \partial \Delta
\hbox{ and }q_{c_{k}} \to \partial \Delta.$$ By passing to a
convergent subsequence, it follows from (\ref{def-B}) that there is
some real constant $\alpha$ such that $$B_{p_{c_{k}}, q_{c_{k}}} \to
e^{i\alpha} z$$ uniformly in any compact subset of ${\Bbb C}-
\partial \Delta$.   In
particular, $$d(c_{k}, \partial \Delta) \to 0$$ as $k \to \infty$.
But this is a contradiction with $d(\eta,
\partial \Delta) > 0$.  The claim has been proved.

\begin{figure}
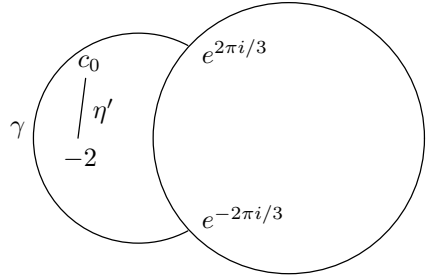

\bigskip
\begin{center}
\centertexdraw { \drawdim cm \linewd 0.02 \move(0 1)

\move(0 -3) \lcir r:1.8

\move(-2 -3) \larc r:1.4 sd:62 ed:298

 \move(-2.8 -2.1)
\htext{$c_{0}$}

\move(-2.7 -2.2) \lvec(-2.8 -3)

\move(-1.15 -2) \htext{$e^{2 \pi i/3}$}

\move(-1.15 -4.2) \htext{$e^{-2 \pi i/3}$}

\move(-3 -3.4) \htext{$-2$}
 \move(-3.7 -3) \htext{$\gamma$}

\move(-2.6 -2.8) \htext{$\eta'$}

}
\end{center}
\vspace{0.2cm} \caption{The continuation of the pair $\{p_{c},
q_{c}\}$ along $\eta'$}
\end{figure}
Since $|p_{c_{0}}| > 1$ and $|q_{c_{0}}|<1$, and $\{p_{2}, q_{2}\} =
\{1, 4/5\}$, it follows that there is a sequence $\{c_{k}\} \subset
\eta$ such that either $p_{c_{k}}$ is contained in some compact set
in the outside of the unit disk and $q_{c_{k}}$ lies in the inside
of the unit disk and approaches $\partial \Delta$, or $q_{c_{k}}$ is
contained in some compact set in the inside of the unit disk and
$p_{c_{k}}$ lies in the outside of the unit disk and approaches
$\partial \Delta$. But by Lemma~\ref{p-q}, both of the two
possibilities imply that $B_{p_{c_{k}}, q_{c_{k}}}$ has two distinct
critical points on $\partial \Delta$ for all $k$ large enough. This
is a contradiction with Remark~\ref{solution}.

In the second case, $c_{0} \in V$. Then we take a curve segment
$\eta' \subset V$ which connects $-2$ and $c$. See Figure 20 for an
illustration.  There are two curves of $\{p_{c}, q_{c}\}$ which are
determined by the two choices of $\{p_{-2}, q_{-2}\}$ in  Example 2
respectively.

For the first choice, $\{p_{-2}, q_{-2}\} = \{1, 4/5\}$. We can get
a contradiction by using the same argument as in the proof of the
first case.

 For the second choice, $\{p_{-2},
q_{-2}\} = \{-0.5+ 1.936491 i, -0.5-1.936491 i\}$. So $|p_{-2}| =
|q_{-2}| > 1$.  Since $|p_{c_{0}}| > 1$ and $|q_{c_{0}}| < 1$, and
since $p_{c}$ and $q_{c}$ can not be both close to $\partial \Delta$
with $|p_{c}| > 1$ and $|q_{c}|< 1$(otherwise  we get a
contradiction by Lemma~\ref{p-q-t} and Remark~\ref{solution}), there
would be a sequence $\{c_{k}\} \subset \eta'$ such that either
$p_{c_{k}}$ is contained in some compact set in the outside of the
unit disk and $q_{c_{k}}$ lies in the inside of the unit disk and
approaches $\partial \Delta$, or $q_{c_{k}}$ is contained in some
compact set in the inside of the unit disk and $p_{c_{k}}$ lies in
the outside of the unit disk and approaches $\partial \Delta$. Again
by Lemma~\ref{p-q}, both of the two possibilities imply that
$B_{p_{c_{k}}, q_{c_{k}}}$ has two distinct critical points on
$\partial \Delta$ for all $k$ large enough. This is a contradiction
with Remark~\ref{solution}.

The above argument implies that $\Phi(\mathcal{B}) \subset U$. Next
we need to prove that for each $c_{0} \in U$, there is a $B_{p, q}
\in \mathcal{B}$ such that $\Phi(B_{p,q}) = c_{0}$. In fact, since
$U$ is simply connected, we can take a curve segment, say $\eta''
\subset U$ to connect the point $2$ and $c_{0}$. For each $c \in
\eta''$, we solve $|v_{c}|$ by (\ref{d-2}) and then get $v_{c} = c
|v_{c}|/|c|$, and a continuous curve of $\{p_{c}, q_{c}\}$.  From
Case (2) of Example 1, we have that $|p_{2}|> 1$ and $|q_{2}|<1$. We
claim that $|p_{c_{0}}|
> 1$ and $|q_{c_{0}}| < 1$. Suppose this were not true.   Then
the same argument as above will induce a contradiction again. This
implies that $\Phi(\mathcal{B}) = U$.
\begin{figure}
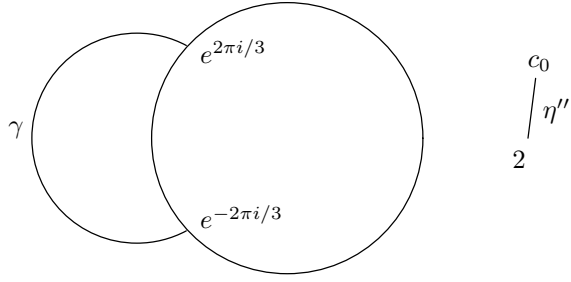

\bigskip
\begin{center}
\centertexdraw { \drawdim cm \linewd 0.02 \move(0 1)

\move(0 -3) \lcir r:1.8

\move(-2 -3) \larc r:1.4 sd:62 ed:298

 \move(3.2 -2.1)
\htext{$c_{0}$}

\move(3.3 -2.2) \lvec(3.2 -3)

\move(-1.15 -2) \htext{$e^{2 \pi i/3}$}

\move(-1.15 -4.2) \htext{$e^{-2 \pi i/3}$}

\move(3 -3.4) \htext{$2$}
 \move(-3.7 -3) \htext{$\gamma$}

\move(3.4 -2.8) \htext{$\eta''$}

}
\end{center}
\vspace{0.2cm} \caption{The continuation of the pair $\{p_{c},
q_{c}\}$ along $\eta''$}
\end{figure}
Finally let us show that $\Phi$ is injective. Assume that for some
$c \in U$,  we have two different pairs $\{p, q\}$ and $\{p', q'\}$
such that $|p|> 1, |q|< 1$, $|p'|>1, |q'| < 1$ and $\Phi(B_{p,q}) =
\Phi(B_{p',q'}) = c$.  Take a curve segment $\eta \subset U$ which
connects $c$ and the point $2$. Then we have two curves of pairs
$\{p_{c}, q_{c}\}$, $c \in \eta$. It follows that one of them  is
determined by (\ref{d-1}), along which we get $\{p_{2}, q_{2}\} =
\{1, 4/5\}$, which is  Case (1) of Example 1. Now the same argument
above will induce a contradiction again. This proves that $\Phi$ is
injective.

Finally let us show that $\Phi^{-1}$ is continuous also. In fact,
for each $c \in U$, compute $|v_{c}|$ by (\ref{d-2}) and get $v_{c}$
by
$$ v_{c} =  c |v_{c}|/|c|.$$  Then we get $w_{c}$ by the
relation
$$
\bar{w}_{c}/\bar{v}_{c} = t_{c} = (2 + c + 1/\bar{c})/2.
$$
Now  the pair $p_{c}, q_{c}$ is determined by the solutions of the
quadratic equation
$$
x^{2} - w_{c} + v_{c} = 0.
$$
By Case (2) of Example 1, and the same argument as before, it
follows that one of the two solutions lies in the outside of the
unit disk, and the other one lies in the inside of the unit disk.
Let $p_{c}$ be the one such that $|p_{c}| > 1$ and $q_{c}$ be the
other one such that $|q_{c}| < 1$. It is clear that $p_{c}$ and
$q_{c}$  depend continuously on $c$.

Note that as $c \to \infty$, by solving (\ref{d-2}) we get $p_{c}
\to 3$ and $q_{c} \to 0$. We can thus define $\Phi(B_{3, 0}) =
\infty$. This completes the proof of the lemma.
\end{proof}

By Proposition 11.7 of \cite{KH}(and see also $\S9$ of \cite{Z1}),
we have
\begin{lemma}\label{G-p-q}
For each $B_{p, q} \in \mathcal{B}$, there is a unique $t \in [0, 2
\pi)$ such that the rotation number of $e^{it}B_{p,q}|\partial
\Delta$ is $\theta$. Moreover, $t$ depends continuously on $B_{p,
q}$.
\end{lemma}
Let us denote $e^{it}B_{p,q}$ by $G_{p,q}$.
\subsection{ The Cross Ratio Function $\lambda(k,l,m,n)$}
For any four distinct points $z_{1}, z_{2}, z_{3}, z_{4}$, their
cross ratios have several definitions. Since the properties
established in this sections are true for any of them, let us simply
use the same notation $C(z_{1}, z_{2}, z_{3},z_{4})$ to denote them.
For $ 0 \le k < l < m < n$, set
$$
\lambda_{k,l,m,n}(c)= C(g_{c}^{k}(1),
g^{l}_{c}(1),g_{c}^{m}(1),g_{c}^{n}(1))
$$
and
$$
\alpha(k, l, m, n) = C(e^{2 \pi ik\theta}, e^{2 \pi il\theta}, e^{2
\pi im\theta}, e^{2 \pi in\theta}).
$$

Let $\xi \subset \widehat{\Bbb C}$ be the simple closed curve in
Lemma~\ref{par-c}. Let  $\Omega_{0}$ and  $\Omega_{\infty}$ denote
the bounded and unbounded components of $\widehat{\Bbb C} -\xi$,
respectively. For $R > 0$ large enough such that $\xi \subset
\{z\big{|}|z| < R\}$, let $U_{R} = \{|z| > R\}$ and $\Omega_{R} =
\Omega_{\infty} - \overline{U_{R}}$.

\begin{lemma}\label{no-finite}
Let $c \in \Omega_{\infty}$. Then the forward orbit of $1$ under
$g_{c}$ is not finite.
\end{lemma}
\begin{proof}
Let us prove it by contradiction. Assume that $g_{c}^{k}(1) =
g_{c}^{l}(1)$ for some integers $0 \le k < l$. Let $u = 1/c$. It
follows that  $u \in \Omega_{0}$ and
$$
g_{u}^{k}(u) = g_{u}^{l}(u).
$$
By Lemma~\ref{main-lem}, $g_{u} \in R_{\theta}^{geom}$ and hence is
modeled by some Blaschke product $G_{p,q} = e^{it}B_{p,q}$ where $t
\in [0, 2 \pi)$ and $B_{p,q} \in \mathcal{B}$(see
Lemma~\ref{G-p-q}). Let $$c_{0} = \Phi(B_{p,q}).$$  Let $U$ be the
parameter space in Lemma~\ref{par-S}. Take a continuous curve
$\gamma: [0, 1] \to U$ such that $\gamma(0) = c_{0}$ and $\gamma(1)
= \infty$.  For each $s \in [0, 1]$, let $B_{s} =
\Phi^{-1}(\gamma(s))$ and $G_{s}$ be the corresponding Blaschke
product determined by $B_{s}$(see Lemma~\ref{G-p-q}).  We thus get a
continuous family of Blaschke products $G_{s}$, $0 \le s \le 1$. Now
for each $0\le s \le 1$, we may perform a quasiconformal surgery on
$G_{s}$ as described in the proof of Lemma~\ref{c-curve} and get a
Siegel rational map $g_{c(s)}$. This surgery induces a surgery map
\begin{equation}\label{Zakeri-S}
{\mathbf{S}}: [0, 1] \to \widehat{\Bbb C} -\{0, 1, -1\}.
\end{equation}
by  ${\mathbf{S}}(s) = c(s)$. It is clear that $${\mathbf{S}}(0) = u
\hbox{ and } {\mathbf{S}}(1) = \infty.$$

By using Zakeri's argument(see $\S12$ of \cite{Z1}), one can show
that the map $\mathbf{S}$ is continuous in $[0, 1]$. That is,
${\mathbf{S}}(s), 0 \le s \le 1$ is a continuous curve connecting
$u$ and the infinity. Since $u \in \Omega_{0}$, by
Lemma~\ref{par-c}, the curve ${\mathbf{S}}(s), 0 \le s \le 1$
intersects $\xi$ at some point.

But on the other hand, since $G_{s}$ has exactly one double critical
point at $1$ and the other two critical points do not lie in the
unit circle for every $0\le s \le 1$, it follows that
${\mathbf{S}}(s)$ does not lies in the boundary of the Siegel disk.
In particular, the curve ${\mathbf{S}}(s), 0 \le s \le 1$ does not
intersect $\xi$. This is a contradiction. The lemma follows.

\end{proof}

From Lemma~\ref{no-finite}, it follows that $\lambda_{k,l,m,n}$ is
holomorphic and has no zeros in $\Omega_{R}$.
\begin{lemma}
$\lambda_{k, l, m, n}(c)$ can be continuously extended to $\partial
\Omega_{R}$. \end{lemma}

\begin{proof}
It suffices to prove that both $\lim_{c \to 1}\lambda_{k, l,m,n}(c)$
and $\lim_{c \to -1}\lambda_{k, l, m,n}$ exist and are finite. In
fact, From (\ref{g-c}) in $\S 4.1$, it follows that $$\lim_{c \to
1}g_{c}(1) = \lambda,$$ and for $z \ne 1$, $$\lim_{c \to 1}g_{c}(z)
= \lambda z,$$ where $\lambda = e^{2 \pi i \theta}$. This implies
that for given $ 0 \le k < l< m< n$, $$\lim_{c \to 1}\lambda_{k, l,
m, n}(c) = \alpha(k,l,m,n).$$ This proves that $\lambda_{k, l, m,
n}(c)$ can be continuously extended to the point $1$.

Now let us consider the case that $c \to -1$.  By solving $g_{c}(z)
= z$, it follows that as $c$ is close to $-1$, $g_{c}$ has three
distinct fixed points $0$, $p_{c}$, and $\infty$ where
$$
p_{c} = \frac{(1- \lambda)2c(1+c)}{4c -\lambda(1+c)^{2}} \to 0
$$ as $c$ approaches  $-1$.
Let $\phi_{c}$ be the M\"{o}bius map such that $\phi_{c}(0) =0$,
$\phi_{c}(1) = 1$ and $\phi_{c}(p_{c}) = \infty$. Then $\phi_{c}
\circ g_{c} \circ \phi_{c}^{-1} = g_{c'}$ where $c' = \phi_{c}(c)$.
By a direct calculation, we have $$\phi_{c}(z) = (1 -p_{c})z/(z -
p_{c}).$$ It follows that $$c' = \phi_{c}(c) = (1 - p_{c})c/(c -
p_{c}).$$ Since $p_{c} \to 0$ as $c \to -1$, it follows that $c' \to
1$. Since the cross ratio  $\lambda_{k, l, m, n}(c)$ is preserved by
M\"{o}bius transformations, it follows that $$\lambda_{k,l,m,n}(c) =
\lambda_{k,l,m,n}(c').$$ We thus get that  $$\lim_{c \to
-1}\lambda_{k,l,m,n}(c) = \lim_{c' \to 1}\lambda_{k,l,m,n}(c') =
\alpha(k,l,m,n).$$
\end{proof}

\begin{lemma} $
\lambda_{k,l,m,n}$ has a removable singularity at $\infty$, and
moreover, $\lim_{c \to \infty}\lambda_{k, l, m, n}(c) \ne 0$.
\end{lemma}
\begin{proof}
It suffices to prove that $\lim_{c \to \infty}\lambda_{k, l, m,
n}(c)$ exists.  From (\ref{g-c}) in $\S 4.1$, it follows that for
any compact set $K$ of the complex plane, $$\lim_{c \to \infty}g_{c}
(z) \to \lambda z -\lambda z^{2}/2 = g_{\infty}(z)$$ uniformly on
$K$. Therefore, for any integer $k \ge 0$, $g_{c}^{k}(1) \to
g_{\infty}^{k}(1)$ as $c \to \infty$. Since $g_{\infty}$ has a
Siegel disk with quasi-circle boundary which passing through $1$, it
follows that the cross ratio $C(g_{\infty}^{k}(1),
g_{\infty}^{l}(1), g_{\infty}^{m}(1), g_{\infty}^{n}(1))$ is defined
and is not equal to $0$.   All of these imply that $$\lim_{c \to
\infty} \lambda_{k,l,m,n}(c) = C(g_{\infty}^{k}(1),
g_{\infty}^{l}(1), g_{\infty}^{m}(1), g_{\infty}^{n}(1))$$ is a
finite non-zero complex value. The lemma follows.
\end{proof}
Let us summarize the above lemmas as the following,
\begin{proposition}\label{nov}
For any integers $0 \le k < l <m <n$, $\lambda_{k, l, m, n}(c)$ is a
non-zero and holomorphic function in $\Omega_{\infty}$. Moreover, it
can be continuously extended to $\partial \Omega_{\infty}$.
\end{proposition}

\begin{remark}\label{supp}
Note that the distortion of a cross ratio by a $K-$quasiconformal
homeomorphism of the sphere is bounded by some constant dependent
only on $K$. This is an important fact which will be used in the
proof of Theorem C.
\end{remark}

\subsection{Proof of Theorem C}
Recall that $\Omega_{\infty}$ is the unbounded component of
$\widehat{{\Bbb C}}-\xi$. For each $c \in \Omega_{\infty}$, let
$\gamma_{c}$ be the closure of $\{g_{c}^{k}(1), k=0, 1, \cdots\}$.

\begin{lemma}\label{Jordan}
For each $c \in \Omega_{\infty}$, $\gamma_{c}$ is a Jordan curve.
\end{lemma}
\begin{proof}
When $c = \infty$, $g_{c}$ is a quadratic polynomial with a bounded
type Siegel disk, and the lemma follows from the well-known theorem
of Douady, Ghys, Herman, and Shishikura \cite{Pe1}. Let us assume
that $c \ne \infty$.

First let us show that $\gamma_{c}$ is contained in some compact set
of ${\Bbb C}$.  In fact, if this were not true, there would be a
subsequence, say $k_{1}, k_{2}, \cdots$, such that
$g_{c}^{k_{i}}(1) \to \infty$ and $e^{2 \pi ik_{i} \theta} \to t$
for some $t \in \partial \Delta$ as $i\to \infty$.  Since $g_{c}$
fixes the infinity, it follows that $g_{c}^{k_{i}+1}(1) \to \infty$
also. Take integers $m, n \ge 0$ such that $e^{2 \pi i m\theta},
e^{2 \pi i n\theta}, t, e^{2 \pi i\theta}t$ are all distinct with
each other. On the one hand, we have
\begin{equation}\label{cross-1}
\frac{( g_{c}^{m}(1) - g_{c}^{n}(1))(g_{c}^{k_{i}}(1)-
g_{c}^{k_{i}+1}(1))}{( g_{c}^{m}(1) -
g_{c}^{k_{i}}(1))(g_{c}^{n}(1)-g_{c}^{k_{i}+1}(1))} \to 0
\end{equation}
as $ i\ \to \infty$. On the other hand, By Proposition~\ref{nov},
the cross ratio function on the left hand of (\ref{cross-1}) is
holomorphic in $c$ and has no zeros in $\Omega_{\infty}$, and
moreover, it can be continuously extended to $\partial
\Omega_{\infty} = \xi$. It follows that its minimum of its modulus
is obtained on $\partial \Omega_{\infty} = \xi$. Since $e^{2 \pi
ik_{i} \theta} \to t$ and $e^{2 \pi i m\theta}, e^{2 \pi i n\theta},
t, e^{2 \pi i\theta}t$ are all distinct,  the cross ratio
$$
\frac{(e^{2 \pi i m\theta} - e^{2 \pi i n\theta})( e^{2 \pi i
k_{i}\theta}- e^{2 \pi i (k_{i}+1)\theta})}{ (e^{2 \pi i m\theta}-
e^{2 \pi i k_{i}\theta})( e^{2 \pi i n\theta} -  e^{2 \pi i
(k_{i}+1)\theta})}
$$
is uniformly bounded away from zero for all $i$ large enough.  By
Lemma~\ref{distortion} and Remark~\ref{supp}, it follows that the
modulus of the cross ratio function on the left hand of
(\ref{cross-1}) has a positive lower bound when restricted on $\xi$.
This is a contradiction.

Now for $c \in \Omega_{\infty}$, define a map $T_{c}: \{e^{2 \pi i k
\theta}, k = 0, 1, \cdots\} \to \widehat{\Bbb C}$ by $T_{c}(e^{2 \pi
i k \theta}) = g_{c}^{k}(1)$. First let us show that $T_{c}$ is
uniformly continuous.  To see this, note that $\theta$ is
irrational, and therefore there is an $M > 0$ dependent only on
$\theta$ such that  for any  $0< \delta < 1/100$, and any $k, l$
with  $$|e^{2 \pi i k\theta} - e^{2 \pi i \l\theta}| < \delta,$$
there exist  integers $0 \le m, n \le  M $ such that
$$|e^{2 \pi i m\theta} - e^{2 \pi i n\theta}| < 1/4, |e^{2 \pi i
m\theta} - e^{2 \pi i k\theta}| > 1/4, \hbox{ and } |e^{2 \pi i
n\theta} - e^{2 \pi i l\theta}| > 1/4.$$ The existence of such $M$
is obvious since for $M$ large, the orbit segment $\{e^{2 \pi t m
\theta}, 0 \le t \le M\}$ will be dense enough in $\partial \Delta$
so that one can find two elements $e^{2 \pi i m \theta}$ and $e^{2
\pi i n \theta}$ in this orbit segment which satisfy the above three
inequalities.

For such $m$ and $n$, we have
\begin{equation}
\bigg{|}\frac{( e^{2 \pi i m\theta}  - e^{2 \pi i n\theta}) (e^{2
\pi i k\theta} - e^{2 \pi i l\theta})}{( e^{2 \pi i m\theta}- e^{2
\pi i k\theta})( e^{2 \pi i l\theta} - e^{2 \pi i n\theta}) }
\bigg{|} < 4 \delta.
\end{equation}
By Remark~\ref{supp} and Lemma~\ref{distortion}, it follows that
there is a positive function $k(\delta)$ satisfying $k(\delta) \to
0$ as $\delta \to 0$, such that for all $t \in \xi$,
\begin{equation}\label{t-b}
\bigg{|}\frac{(g_{t}^{m}(1) - g_{t}^{n}(1))(g_{t}^{k}(1)-
g_{t}^{l}(1))}{(g_{t}^{m}(1) -
g_{t}^{k}(1))(g_{t}^{l}(1)-g_{t}^{n}(1)}\bigg{|}< k(\delta).
\end{equation}
From (\ref{t-b}),  Proposition 4.1, and the maximal modulus
principle, we have
\begin{equation}
\bigg{|}\frac{(g_{c}^{m}(1) - g_{c}^{n}(1))(g_{c}^{k}(1)-
g_{c}^{l}(1))}{(g_{c}^{m}(1) -
g_{c}^{k}(1))(g_{c}^{l}(1)-g_{c}^{n}(1)}\bigg{|}< k(\delta)
\end{equation}
for all $c \in \Omega_{\infty}$.

Since $0\le m < n$ are bounded by $M$ which depends only on
$\theta$, for any given $c$,
$$
|g_{c}^{m}(1) - g_{c}^{n}(1)|
$$
has a positive lower bound. Since we have proved that the absolute
value of the denominator of the above fraction has an upper bound in
the beginning of the proof, it follows that
$$
|g_{c}^{k}(1) - g_{c}^{l}(1)| < C k(\delta)
$$
for some uniform $C >
0$. This implies the  uniform continuity of $T_{c}$. Now we can
continuously extend $T_{c}$ to the unit circle.

We now need only to prove that $T_{c}$ is injective. We prove this
by contradiction. Assume that there exist $x, y \in \partial \Delta$
such that $T_{c}(x) = T_{c}(y)$  and $x \ne y$. Take subsequences
$k_{i}$, $l_{i}$ such that $e^{2 \pi i k_{i} \theta} \to x$ and
$e^{2 \pi i l_{i} \theta} \to y$ as $i \to \infty$.  Take integers
$m, n$ such that $T_{c}(e^{2 \pi i m \theta}), T_{c}( e^{2 \pi i n
\theta})$ and $T_{c}(x)$ are all distinct. It follows that $ e^{2
\pi i m \theta}, e^{2 \pi i m \theta}, x$ and $y$ are all distinct.
Then there is a uniform $\delta > 0$ such that
$$
\big{|}\frac{( e^{2 \pi i m\theta}  - e^{2 \pi i n\theta}) (e^{2 \pi
i k_{i}\theta} - e^{2 \pi i l_{i}\theta})}{( e^{2 \pi i m\theta}-
e^{2 \pi i k_{i}\theta})( e^{2 \pi i l_{i}\theta} - e^{2 \pi i
n\theta}) } \big{|} \ge \delta
$$
for all $i$ large enough. By Lemma~\ref{distortion} and
Remark~\ref{supp}, it follows that there is a constant $C(\delta) >
0$ which depends only on $\delta$ such that
$$
\bigg{|}\frac{(g_{t}^{m}(1) - g_{t}^{n}(1))(g_{t}^{k_{i}}(1)-
g_{t}^{l_{i}}(1))}{(g_{t}^{m}(1) -
g_{t}^{k_{i}}(1))(g_{t}^{n}(1)-g_{t}^{l_{i}}(1)}\bigg{|} > C(\delta)
$$
for all $t \in \xi$. This, together with  Proposition~\ref{nov} and
the minimal modulus principle, implies
$$
\bigg{|}\frac{(g_{c}^{m}(1) - g_{c}^{n}(1))(g_{c}^{k_{i}}(1)-
g_{c}^{l_{i}}(1))}{(g_{c}^{m}(1) -
g_{c}^{k_{i}}(1))(g_{c}^{n}(1)-g_{c}^{l_{i}}(1)}\bigg{|} >
C(\delta).
$$
Since $T_{c}(e^{2 \pi i m \theta}), T_{c}( e^{2 \pi i n \theta})$
and $T_{c}(x)$ are distinct with each other and $T(x) = T(y)$, the
absolute value of the denominator of the above fraction has a
positive lower bound. But  the numerator goes to zero as  $i \to
\infty$. This is a contradiction. The lemma follows.

\end{proof}

Using the above argument in the proof of the uniform continuity of
$T_{c}$, the reader shall easily supply a proof of the following
lemma,
\begin{lemma}\label{e-x}
Define $T:  \Omega_{\infty} \times \partial \Delta \to {\Bbb C}$ by
$T(c, x) = T_{c}(x)$.  Then $T$ is continuous.
\end{lemma}

The following lemma characterizes a quasi-circle in the complex
plane by the lower bound of the cross ratios of every four ordered
points on it( Lemma 9.8 \cite{Po}),
\begin{lemma}\label{q-c}
For each $\delta > 0$, there is a $K(\delta) > 1$ such that for any
simple closed curve $\gamma \subset \Bbb C$, if for every four
ordered points $z_{1}, z_{2}, z_{3}, z_{4} \in \gamma$, the
following inequality hold,
\begin{equation}\label{f-i}
\big{|}\frac{(z_{1} - z_{3})(z_{2}-z_{4})}{(z_{2} -
z_{3})(z_{1}-z_{4})}\big{|} \ge \delta,
\end{equation}
then $\gamma$ is a $K(\delta)-$quasi-circle. Similarly, for each $K
> 1$, there is a $\delta(K) >0$ such that for any $K-$quasi-circle
$\gamma \subset \Bbb C$, the following inequality hold for every
four ordered points $z_{1}, z_{2}, z_{3},z_{4}$ on $\gamma$,
\begin{equation}\label{s-i}
\big{|}\frac{(z_{1} - z_{3})(z_{2}-z_{4})}{(z_{2} -
z_{3})(z_{1}-z_{4})}\big{|} \ge \delta(K).
\end{equation}
\end{lemma}

Now let us prove Theorem C. By Lemma~\ref{Jordan}, $\gamma_{c}$ is a
simple closed curve. By Lemma~\ref{distortion}, there exists a
uniform $1 < K < \infty$ such that the boundary of the Siegel disk
of $g_{c}$ for every $c \in \xi$ is a $K-$quasi-circle. By
Lemma~\ref{q-c}, the inequality (\ref{s-i}) holds for every $c \in
\xi$ and any four ordered points $z_{1}, z_{2}, z_{3}, z_{4}$ on
$\gamma_{c}$. By Proposition~\ref{nov} and minimal modulus
principle, it also holds for every $c \in \Omega_{\infty}$. By
Lemma~\ref{q-c} again, it follows that there is a uniform $1 < K'<
\infty$ such that $\gamma_{c}$ is a $K'-$quasi-circle for every $c
\in \Omega_{\infty}$.

Now we need only to show that $\gamma_{c}$ is the boundary of the
Siegel disk of $g_{c}$ which is centered at the origin. By
Lemma~\ref{e-x}, $\gamma_{c}$ moves continuously as $c$ varies in
$\Omega_{\infty}$. Let $D_{c}$ be the bounded component of ${\Bbb C}
- \gamma_{c}$. We will show that $D_{c}$ is the Siegel disk of
$g_{c}$. First let us show that $g_{c}$ is holomorphic on $D_{c}$.
When $c = \infty$, this is obviously true. As $c$ varies from the
infinity to any value in $\Omega_{\infty}$, the finite pole of
$g_{c}$, say $p_{c}$, varies continuously. But $g_{c}(\gamma_{c}) =
\gamma_{c}$ and $\infty \notin \gamma_{c}$, it follows that
$\gamma_{c}$ does not meet the pole $p_{c}$. It follows that $p_{c}
\notin D_{c}$ for otherwise there is some $c$ such that $\gamma_{c}$
meets $p_{c}$, which is a contradiction.  This implies that $g_{c}$
is holomorphic on $D_{c}$. Since $$g_{c}(\partial D_{c}) =
g_{c}(\gamma_{c}) = \partial D_{c},$$ it follows that
$$g_{c}(D_{c}) = D_{c}.$$ This implies that $D_{c}$ is a periodic
Fatou component of $g_{c}$.  Since $0 \notin \gamma_{c}$ for every
$c \in \Omega_{\infty}$, and $0 \in D_{\infty}$, by the same
argument as above, it follows that $0 \in D_{c}$ for every $c \in
\Omega_{\infty}$. Because $g_{c}'(0) = e^{2 \pi i \theta}$ and
$g_{c}(0) = 0$,  it follows that $D_{c}$ is the Siegel disk of
$g_{c}$ which is centered at the origin, and in particular,
$\partial D_{c}$ passes through the critical point $1$ of $g_{c}$.
This completes the proof of Theorem C.

Let $\xi$ be the simple closed curve in Lemma~\ref{par-c}. Recall
that $\Omega_{0}$ is the bounded component of $\widehat{\Bbb C} -
\xi$ and $\Omega_{\infty}$ the unbounded one. For $c \in
\widehat{\Bbb C}-\{0, 1, -1\}$, let $D_{c}$ be the Siegel disk of
$g_{c}$ which is centered at the origin. Based on the proof of
Theorem C, the reader shall easily draw the following conclusion,
\begin{proposition}\label{d-c} Let $c \in \widehat{\Bbb C} - \{0, 1, -1\}$. We have
(1) if  $c \in \xi$, $\partial D_{c}$ passes through both of the
critical points $1$ and $c$, (2) if $c \in \Omega_{\infty}$,
$\partial D_{c}$ passes through $1$ only, (3) if $c \in \Omega_{0}$,
$\partial D_{c}$ passes through $c$ only.

\end{proposition}

\begin{corollary}\label{Her}
Let $f$ be a degree-3 rational map with a bounded type Herman ring.
Then each boundary component of the Herman ring is a quasi-circle
which passes through at least one but at most two of the critical
points of $f$.
\end{corollary}
\begin{proof}
This is because for any boundary component $\gamma$ of the Herman
ring, by using a quasi-conformal surgery, one can get a quadratic
rational map with  a  Siegel disk which has the same rotation number
as the Herman ring and which has $\gamma$ as its boundary. We leave
the details to the reader.
\end{proof}

\begin{center}
\section{Appendix}
\end{center}

\subsection{Thurston's characterization theory on postcritically finite rational maps}
Since the Thurston's characterization theorem used in this paper is
slightly different from the one presented in \cite{DH}, we will give
a brief introduction of this theory, which has been adapted to our
situation: we use a larger invariant set $X \supseteqq P_{f}$,
instead of the $\emph{postcritical set}$ $P_{f}$. The proof is
completely the same as the one presented in \cite{DH}.

Let $f: S^{2} \to S^{2}$ be a $\emph{postcritically finite}$
branched covering map. Let $X \subset S^{2}$ be a finite set such
that $f(X) \subseteqq X$ and $P_{f} \subseteqq X$. A simple closed
curve in $S^{2} - X$ is said to be $\emph{non-peripheral}$ if
$\gamma$ is not homotopic to a point in $S^{2} - X$. A
$\emph{multi-curve}$ of $f$ in $S^{2} - X$ is a family of disjoint,
non-homotopic and non-peripheral curves. We say a
$\emph{multi-curve}$ $\Gamma$ is $f-$stable if for any $\gamma \in
\Gamma$, any $\emph{non-peripheral}$ component of $f^{-1}(\gamma)$
is homotopic in $S^{2} - X$ to one of the elements in $\Gamma$.

Two branched covering maps $f$ and $g$ are said to be
$\emph{combinatorially equivalent}$ with respect to the set $X$ if
there are two homeomorphisms of the sphere $\phi, \psi$ which are
isotopic to each other rel $X$ such that $f = \phi^{-1} \circ g
\circ \psi$.

Let $\Gamma  = \{\gamma_{1}, \cdots, \gamma_{n}\}$ be a $f-$stable
$\emph{multi-curve}$.  For each $\gamma_{j}$, let $\gamma_{i, j,
\alpha}, \alpha \in \Lambda$ be the $\emph{non-peripheral}$
components of $f^{-1}(\gamma_{j}$ which is homotopic to
$\gamma_{i}$.  Define
$$
a_{i,j}  = \frac{1}{d_{i,j,\alpha}}.
$$
The matrix $A = (a_{i, j})_{n \times n}$ is called the Thurston
linear transformation matrix of $f$.  $\Gamma$ is called a Thurston
obstruction if the maximal eigenvalue of $A$ is greater than $1$.

Associated to each $\emph{postcritically finite}$ rational map $f$,
one can construct an orbifold $\mathcal{O}$$_{f}$$ = (S^{2},
\nu_{f})$ by defining $\nu_{f}: S^{2} \to {\Bbb Z}^{+} \cup
\{\infty\}$ to be the minimal function satisfying the following two
conditions,
\begin{itemize}
\item[1.] $\nu_{f}(x) = 1$ for $x \notin P_{f}$,
\item[2.] $\nu_{f}(x)$ is a multiple of $\nu_{f}(y)\deg_{y}f$ for each $y \in f^{-1}(x)$.
\end{itemize}

An orbifold is called $\emph{hyperbolic}$ if
$$
\chi_{f} = 2 - \sum_{\nu_{f}(x) \ge 2}(1 - \frac{1}{\nu_{f}(x)}) <
0.
$$

\begin{T}
Let $f$ be a postcritically finite branched covering map of the
sphere. Let $X$ be a finite set such that $f(X) \subseteqq X$, and
$P_{f} \subseteqq X$. Assume that the orbifold $\mathcal{O}$$_{f}$
is hyperbolic.  Then $f$ is combinatorially equivalent to a rational
map with respect to $X$ if and only  $f$ has no Thurston
obstructions in $S^{2} - X$.
\end{T}

\subsection{Short simple closed geodesics}

In this appendix, we present a few results on the simple closed
geodesics in a hyperbolic Riemann surface, and for detailed proofs
of these results, we refer the reader to $\S6$ and $\S7$ of [DH].

\begin{A.1}[Corollary \:6.6, \cite{DH}] Let $X$ be a hyperbolic Riemann surface and $\gamma_{1}, \gamma_{2}$ be two simple closed geodesics with length $< \log(\sqrt2 + 1)$. Then either $\gamma_{1} = \gamma_{2}$ or $\gamma_{1} \cap \gamma_{2} = \emptyset$.
\end{A.1}

\begin{A.2}[Corollary \: 6.7, \cite{DH}] Let $X$ be a hyperbolic Riemann surface. Let $\gamma$ be a geodesic in $X$ which intersects itself transversally at least once. Then $l_{X}(\gamma) > 2 \log(\sqrt2 + 1)$.
\end{A.2}

\begin{A.3}[Theorem \:7.1, \cite{DH}] Let $X$ be a hyperbolic Riemann surface, $P \subset X$ a finite set, with $|P| = p>0$. Choose $L <  \log(\sqrt2 + 1)$.
Let $X' = X - P$. Let $\gamma$ be a simple closed geodesic in $X$
and $\{\gamma_{1}', \cdots, \gamma_{s}'\}$ be the simple closed
geodesic in $X'$ which is homotopic to $\gamma$ in $X$ with length
$< L$. Then
$$
\frac{1}{l} - \frac{2}{\pi}-\frac{p+1}{L} < \sum_{1 \le i \le
s}\frac{1}{l_{i}'} < \frac{1}{l} + \frac{2(p+1)}{\pi}.
$$
\end{A.3}

\begin{A.4}[Proposition \:7.2, \cite{DH}] Let $P \subset S^{2}$ be a finite set, and $\gamma$ be a non-peripheral curve in $S^{2} - P$. Let $\phi, \psi: S^{2} \to {\Bbb P}^{1}$ be quasiconformal homeomorphisms.   If $dist_{T(S^{2}, P)}(\phi, \psi) < K$, then
$$
e^{-2K}\|\gamma\|_{\phi,P} \le \|\gamma\|_{\psi,P} \le
e^{2K}\|\gamma\|_{\phi,P}.
$$
\end{A.4}

\bibliographystyle{amsalpha}

\end{document}